\documentclass[11pt]{amsart}
\usepackage[a4paper,margin=2.7cm]{geometry}
\usepackage[T1]{fontenc}
\usepackage{lmodern}
\usepackage{amsmath,amssymb,amsthm,mathtools,mathrsfs}
\usepackage[expansion=false]{microtype}
\usepackage{booktabs,array}
\usepackage{graphicx}
\usepackage{xcolor}
\usepackage{float}
\usepackage[hypertexnames=false,hidelinks]{hyperref}
\hypersetup{pdftitle={Noncompactness for the constant Q2N-curvature problem},
pdfauthor={},bookmarksdepth=2}

\allowdisplaybreaks
\numberwithin{equation}{section}
\makeatletter
\renewcommand{\l@subsection}{\@tocline{2}{0pt}{3pc}{3pc}{}}
\makeatother

\newtheorem{theorem}{Theorem}[section]
\newtheorem{maintheorem}{Theorem}

\newtheorem{proposition}[theorem]{Proposition}
\newtheorem{lemma}[theorem]{Lemma}
\newtheorem{corollary}[theorem]{Corollary}
\theoremstyle{definition}

\theoremstyle{remark}
\newtheorem{remark}[theorem]{Remark}
\newtheorem{question}[theorem]{Question}

\newcommand{\R}{\mathbb R}
\newcommand{\Sph}{\mathbb S}
\newcommand{\tr}{\operatorname{tr}}
\newcommand{\Ric}{\operatorname{Ric}}
\newcommand{\Vol}{\operatorname{Vol}}
\newcommand{\Id}{\mathrm{Id}}
\newcommand{\cE}{\mathcal E}
\newcommand{\cA}{\mathcal A}
\newcommand{\cL}{\mathscr L}
\newcommand{\eps}{\varepsilon}

\title[Noncompactness for the constant $Q_{2N}$-curvature problem]{Noncompactness for the constant \\ $Q_{2N}$-curvature problem}

\author{Liuwei Gong}
\address[Liuwei Gong]{Department of Mathematics, Chinese University of Hong Kong, Shatin, NT, Hong Kong}
\email{lwgong@math.cuhk.edu.hk}

\author{Seunghyeok Kim}
\address[Seunghyeok Kim]{Department of Mathematics and Research Institute for Natural Sciences, College of Natural Sciences, Hanyang University, 222 Wangsimni-ro Seongdong-gu, Seoul 04763, Republic of Korea}
\email{shkim0401@hanyang.ac.kr}
\email{shkim0401@gmail.com}

\author{Juncheng Wei}
\address[Juncheng Wei]{Department of Mathematics, Chinese University of Hong Kong, Shatin, NT, Hong Kong}
\email{wei@math.cuhk.edu.hk}

\date{November 4, 2026}

\begin{document}
\begin{abstract}
For every integer $N\ge4$, we construct a fixed smooth, non-locally-conformally-flat metric on the $n$-dimensional unit sphere $\Sph^n$ for which the constant $Q_{2N}$-curvature equation admits an $L^\infty$-unbounded sequence of positive solutions.
The construction applies for $n\ge2N+20$ when $N=4,5$, $n\ge2N+19$ when $6\le N\le8$, $n\ge2N+18$ when $9\le N\le17$, and $n\ge2N+17$ when $N\ge18$.
For each $L \in \mathbb{N} \cup \{0\}$, the metric can be chosen arbitrarily close to the round metric in the $C^L$ norm.
It can also be chosen to have a coercive GJMS operator, a strictly positive Green function, and positive Green function mass wherever the Weyl tensor and its covariant derivatives through order $\lfloor(n-2N-4)/2\rfloor$ vanish.
Together with the known cases $N=1,2,3$, this establishes noncompactness at every order and proves the sharpness of the companion compactness ranges under the same operator and mass hypotheses.

The construction relies on our formula for the fixed-volume Hessian of total $Q_{2N}$-curvature in \emph{transverse-traceless} (TT) directions at closed Einstein metrics satisfying $\operatorname{Ric}=(n-1)g$.
For every $N \in \mathbb{N}$ and $n>2N$, this Hessian is represented by a degree-$N$ polynomial in the Lichnerowicz Laplacian, extending Matsumoto's formula for $n=2N$ \cite{Matsumoto}.

For the metric perturbations generated by a fixed algebraic Weyl tensor, the quadratic reduced energy with fixed bubble center is a positive multiple of this Hessian restricted to a four-dimensional TT space. 
Differentiation with respect to the bubble scale then yields a finite matrix whose exact sign analysis establishes the required indefiniteness in the stated dimension ranges.
After continuation to real $n$, this matrix changes from negative definite to indefinite at $n_N=2N+a_*+c_*N^{-1}+O(N^{-2})$ as $N\to\infty$,
with $a_*\approx16.201871$ and $c_*\approx13.512880$, explaining the eventual bound $n\ge2N+17$.
\end{abstract}
\maketitle
\tableofcontents

\section{Introduction}
We study loss of compactness for higher-order constant curvature equations in a fixed conformal class.
We construct a single perturbation of the round metric that admits positive solutions with unbounded maxima.

Let $N \in \mathbb{N}$ and let $(M^n,g)$ be a smooth closed Riemannian manifold with $n>2N$.
We use $\Delta_g\coloneqq\operatorname{div}_g\nabla$ and let $P_{2N,g}$ denote the GJMS operator with leading term $(-\Delta_g)^N$ introduced in \cite{GJMS}, with the convention
\[
P_{2N,g}(1)=\frac{n-2N}{2}Q_{2N,g}.
\]
We study positive solutions of
\begin{equation}\label{eq:introPDE}
P_{2N,g}u=c_{n,N}u^{\frac{n+2N}{n-2N}}\quad\text{on }M.
\end{equation}
Here $c_{n,N}\coloneqq\Gamma(n/2+N)/\Gamma(n/2-N)>0$.
By conformal covariance, the metric $u^{4/(n-2N)}g$ then has constant $Q_{2N}$-curvature $2c_{n,N}/(n-2N)$, the value on the unit round sphere.

\medskip
The fourth-order conformally covariant operator was introduced by Paneitz \cite{Paneitz}; Branson \cite{Branson} developed the associated curvature invariants and higher-order operators.
Graham--Jenne--Mason--Sparling \cite{GJMS} constructed the conformally invariant powers of the Laplacian from the ambient metric of Fefferman--Graham \cite{FG}.
Graham--Zworski \cite{GrahamZworski} obtained the operators from scattering theory on asymptotically hyperbolic manifolds.
Juhl \cite{Juhl} derived explicit formulas for the GJMS operators and $Q_{2N}$-curvatures in terms of compositions of natural second-order operators and volume coefficients.
Fefferman--Graham \cite{FGJuhl} proved these formulas directly from the ambient construction.

For $N=1$, equation \eqref{eq:introPDE} is the Yamabe equation.
The existence problem, introduced by Yamabe \cite{Yamabe}, was resolved through the work of Trudinger \cite{Trudinger}, Aubin \cite{Aubin}, and Schoen \cite{SchoenExistence}.
Compactness concerns all positive solutions in a fixed conformal class, rather than only a minimizing solution.
Schoen initiated this question and proved compactness in the locally conformally flat case, excluding manifolds conformally diffeomorphic to the round sphere \cite{SchoenCompactness}.
Khuri--Marques--Schoen \cite{KhuriMarquesSchoen} established compactness for $3\le n\le24$, with the same exclusion, assuming the Riemannian positive mass theorem needed in their proof; Brendle--Wang \cite{BrendleWang} have since supplied a proof of the positive mass theorem in arbitrary dimension.
Brendle \cite{BrendleBlowup} constructed noncompact examples for $n\ge52$, and Brendle--Marques \cite{BrendleMarques} extended the construction to $25\le n\le51$.

For $N=2$ and $n\ge5$, positivity of solutions requires additional arguments because a fourth-order operator need not satisfy a maximum principle.
Qing--Raske \cite{QingRaske,QingRaskeExistence} obtained existence and compactness on locally conformally flat manifolds with positive Yamabe invariant and Poincar\'e exponent less than $(n-4)/2$, with the round conformal class excluded for compactness.
Gursky--Malchiodi \cite{GurskyMalchiodi} established a strong maximum principle and constructed a constant-$Q_4$ metric by a nonlocal flow when $R_g\ge0$, $Q_{4,g}\ge0$, and $Q_{4,g}\not\equiv0$.
Hang--Yang \cite{HangYang} replaced the scalar-curvature assumption by positivity of the Yamabe invariant, using a variational approach.
For $n\ge6$, Gursky--Hang--Lin \cite[Corollaries 1.1--1.2]{GurskyHangLin} obtained existence under positivity of the Yamabe invariant and the Paneitz operator.

For fourth-order compactness, Hebey--Robert \cite{HebeyRobert} also treated locally conformally flat manifolds under their strong positivity condition.
Li--Xiong \cite[Theorem 1.1]{LiXiong} proved compactness for $5\le n\le9$, outside the round conformal class, assuming invertibility of the Paneitz operator, positivity of its Green function, and the corresponding positive-mass property.
They also treated locally conformally flat manifolds of positive Yamabe type and, for $n\ge8$, manifolds with nowhere-vanishing Weyl tensor, under the same invertibility and Green-function assumptions.
Gong--Kim--Wei \cite[Theorem 1.1]{GongKimWei} extended the dimension range to $5\le n\le24$, outside the round conformal class, under those operator and mass hypotheses, or under the geometric assumptions $Y(M,[g])>0$, $Q_{4,g}\ge0$, and $Q_{4,g}\not\equiv0$; here $Y(M,[g])$ is the Yamabe invariant.
Wei--Zhao \cite{WeiZhao} constructed noncompact examples for every $n\ge25$.

For $N=3$, Chen--Hou \cite[Theorem 1.1]{ChenHou} proved existence for prescribed positive $Q_6$ on closed positive Einstein manifolds with $n\ge10$, under a nonvanishing Weyl condition and vanishing conditions on derivatives of the prescribed function at a maximum point.
Gong--Kim--Wei \cite[Theorems 1.3--1.4]{GongKimWei} proved compactness for $7\le n\le26$, outside the round conformal class, assuming invertibility, positivity of the Green function, and the corresponding positive-mass property, and constructed noncompact examples for every $n\ge27$.

For arbitrary $N \in \mathbb{N}$ with $n>2N$, Qing--Raske \cite[Theorems 1.1--1.2]{QingRaskeExistence} obtained existence and compactness in the locally conformally flat case with positive Yamabe invariant and Poincar\'e exponent less than $(n-2N)/2$, again excluding the round conformal class for compactness.
Mazumdar--V\'etois \cite[Theorem 1.2]{MazumdarVetois} proved existence of a constant-$Q_{2N}$ metric under positivity of the higher-order Yamabe invariant and of the Green function, assuming in addition positive mass at some point if $2N+1\le n\le2N+3$ or the manifold is locally conformally flat.
More recently, Mazumdar--Ndiaye \cite[Theorem 1.1]{MazumdarNdiaye} obtained existence in these latter settings by the barycenter method, assuming positivity of the higher-order Yamabe invariant and of the Green function, without an assumption on the sign of the mass.
For compactness, Mazumdar--Premoselli \cite[Theorem 1.1]{MazumdarPremoselli} assume invertibility and positivity of the Green function and treat locally conformally flat manifolds with positive mass, as well as the full range $2N+1\le n\le2N+5$ under the appropriate mass hypothesis.
In this range, positive mass is required at every point for $n\le2N+3$ and at every zero of the Weyl tensor for $2N+4\le n\le2N+5$.
They also prove compactness for $n\ge2N+4$ when the Weyl tensor is nowhere zero.

At order four, Gong--Lee--Wei \cite[Sections 3--4]{GLW} use the second variation of localized total $Q_4$-curvature in their proof of convergence of the Gursky--Malchiodi flow.
Their estimates for corrected bubbles use the decomposition of the metric variation at the round sphere into a TT tensor and a Lie derivative.

\medskip
Our main contribution is a method that treats the second variation of total $Q_{2N}$-curvature in dimensions $n>2N$ and the noncompactness construction for arbitrary $N$ within a single framework.
The ambient construction of Fefferman--Graham \cite{FGJuhl} allows us to compute the second variation in TT directions via the ordinary differential equation \eqref{eq:TTambient} and its Taylor recurrence, with $N$ treated as a parameter.
This avoids deriving the complete curvature formulas for the extended obstruction tensors and volume coefficients at each order before differentiating them.
On the selected family of symmetric two-tensors whose components are polynomials in Euclidean coordinates, polynomial ambient extensions express the entries of the energy matrix $E$ in \eqref{eq:Ematrix} as finite sums; the identity \eqref{eq:telescoping} writes each summand as a difference of consecutive terms, so that the sum reduces to its endpoints.
This keeps the size of the algebraic systems independent of $N$.
The rational formulas of Theorem \ref{thm:closedentry} yield exact matrix signs over infinite ranges of $(N,n)$.

We use the mass condition from the companion compactness theorem \cite[Theorem 1.1]{CompactnessMass}.
For a smooth metric $g$ on $\Sph^n$ with $n>2N+3$, we define
\begin{equation}\label{eq:mass-locus}
\mathcal Z(g):=\left\{\xi\in\Sph^n:
\nabla_g^jW_g(\xi)=0\text{ for }0\le j\le
\left\lfloor\frac{n-2N-4}{2}\right\rfloor\right\}.
\end{equation}
Here $W_g$ is the Weyl curvature tensor.
If $P_{2N,g}$ is invertible and $\xi\in\mathcal Z(g)$, the \emph{Green function mass} $m_g(\xi)$ is the constant remaining in the expansion of the Green function in a fixed smooth family of conformal normal coordinates after subtraction of the leading radial singularity and all nonradial terms through homogeneity zero.
The mass condition is
\begin{equation}\label{eq:mass-condition}
m_g(\xi)>0\quad\text{for every }\xi\in\mathcal Z(g).
\tag{M}
\end{equation}

Our first main result combines the new construction for $N\ge4$ with the known cases $N=1,2,3$ due, respectively, to Brendle and Brendle--Marques, Wei--Zhao, and Gong--Kim--Wei, whose metric-perturbation framework we adapt \cite{BrendleBlowup,BrendleMarques,WeiZhao,GongKimWei}.

\begin{maintheorem}\label{thm:main}
Let $N\ge1$ and $n>2N$ be integers satisfying
\begin{equation}\label{eq:dimensionrange}
n\ge
\begin{cases}
2N+23 &\text{for } N=1,\\
2N+21 &\text{for } N=2,3,\\
2N+20 &\text{for } N=4,5,\\
2N+19 &\text{for } N=6,7,8,\\
2N+18 &\text{for } 9\le N\le17,\\
2N+17 &\text{for } N\ge18.
\end{cases}
\end{equation}
There exist a smooth, non-locally-conformally-flat metric $g$ on $\Sph^n$ and a sequence $\{u_j\}_{j \in \mathbb{N}} \subset C^\infty(\Sph^n)$ of positive solutions of \eqref{eq:introPDE} with $\max_{\Sph^n}u_j\to\infty$ as $j\to\infty$.
For every fixed $L \in \mathbb{N} \cup \{0\}$ and every $\epsilon>0$, the metric can be chosen to satisfy $\|g-g_c\|_{C^L(g_c)}<\epsilon$, where $g_c$ is the unit round metric on $\Sph^n$ and the $C^L(g_c)$-norm is defined in \eqref{eq:Sobolev-norms}.
It can also be chosen so that $P_{2N,g}$ is coercive, its Green function is strictly positive off the diagonal, and the mass condition \eqref{eq:mass-condition} holds.
\end{maintheorem}

Section \ref{sec:positive-mass} specifies the normalization of the Green function mass and proves the mass assertion in Theorem \ref{thm:main}.
The proof computes the second variation of this finite part at the round metric and establishes positivity uniformly over the points in \eqref{eq:mass-locus}, which may move with the perturbation.

\begin{remark}\label{rem:compactness-in-preparation}
For every new case $N\ge4$, our construction uses a radial polynomial of degree $M=3$ in $|x|^2$, giving a polynomial metric perturbation of degree $2M+2=8$.
The integrability condition $n>2N+4M+4$ therefore becomes $n>2N+16$, whereas degree $M=4$ would require $n>2N+20$.
Thus the choice $M\ge4$ suggested immediately before \cite[Conjecture 1.6]{GongKimWei} is not necessary for the higher-order construction.
Theorem \ref{thm:cross} further places the determinant crossing between $2N+16$ and $2N+17$ for all sufficiently large $N\in\mathbb N$, explaining the eventual transition between compactness and noncompactness.
The companion compactness theorem \cite[Theorem 1.1]{CompactnessMass}, together with the mass assertion in Theorem \ref{thm:main}, proves that the dimension ranges in \eqref{eq:dimensionrange} are sharp under the same operator and mass hypotheses.
\end{remark}

The second main result identifies the quadratic energy underlying the construction.
Throughout this paper, metric variations are symmetric covariant two-tensors.
At a reference metric $g$, a TT tensor $h$ satisfies $\tr_g h=0$ and $\nabla_g^i h_{ij}=0$.
For the fixed order $N$, the total $Q_{2N}$-curvature is
\[
\mathcal Q(g)\coloneqq\int_M Q_{2N,g}\,dv_g.
\]
We use $g_E$ for an Einstein reference metric satisfying $\Ric_{g_E}=(n-1)g_E$.
For tensors $h,k$ that are trace-free with respect to $g_E$, the fixed-volume second variation is the mixed derivative along $g_E\exp(s_1g_E^{-1}h+s_2g_E^{-1}k)$ at $s_1=s_2=0$; this path preserves the volume form pointwise.
We set $\cA\coloneqq-(\Delta_L+2(n-1))$, with the Lichnerowicz Laplacian convention specified in \eqref{eq:Aconvention}; on the unit round sphere, $\cA=\nabla^*\nabla+2$ on TT tensors.
Here $\nabla^*$ denotes the formal $L^2$ adjoint of the covariant derivative $\nabla$.
In Theorem \ref{thm:ambientHessian}, we pair tensors using the $L^2(g_E)$ inner product.

\begin{maintheorem}\label{thm:ambientHessian}
Let $N\ge1$, $n>2N$, and let $(M^n,g_E)$ be a smooth closed Einstein manifold with $\Ric_{g_E}=(n-1)g_E$.
For smooth TT tensors $h,k$, the pointwise first variation $D_g(P_{2N,g}1)|_{g_E}[h]$ vanishes, and
\begin{equation}\label{eq:TTpolynomial}
D^2\mathcal Q|_{g_E}(h,k)
=\left\langle h,\mathscr P_N^{(n)}(-\cA)k\right\rangle_{L^2(g_E)}.
\end{equation}
Here $\mathscr P_N^{(n)}(b)$ is a polynomial of degree $N$, with coefficients rational in $n$, determined by the finite recurrences \eqref{eq:qrecurrence}, \eqref{eq:thetarecurrence} and the coefficient formula \eqref{eq:WronskianHessian}.
In particular, if $-\cA h=bh$, then $D^2\mathcal Q|_{g_E}(h,h)=\mathscr P_N^{(n)}(b)\|h\|_{L^2(g_E)}^2$.
The bilinear form \eqref{eq:TTpolynomial} extends continuously to TT tensors in the Sobolev space $H^N(M;\operatorname{Sym}^2T^*M)$.
\end{maintheorem}

\noindent Along the fixed-volume path $g_E\exp(s_1g_E^{-1}h+s_2g_E^{-1}k)$, the second variation of the normalized functional $\Vol(g)^{-(n-2N)/n}\mathcal Q(g)$ is $\Vol(g_E)^{-(n-2N)/n}$ times \eqref{eq:TTpolynomial}.
Thus the fixed-volume Hessian of $\mathcal Q$ and the Hessian of the normalized functional differ by this explicit positive factor.

\begin{remark}\label{rem:higher-Einstein-linearization}
Theorem \ref{thm:ambientHessian} also determines the linearization in TT directions of the higher Einstein tensors introduced by Gover--{\O}rsted \cite[Section 4.3]{GoverOrsted} as variational gradients of total $Q_{2N}$-curvature.
For $n>2N$, we set $\mathsf E_{2N}\coloneqq(n-2N)^{-1}\operatorname{grad}_{L^2}\mathcal Q$, where $\operatorname{grad}_{L^2}\mathcal Q$ denotes the $L^2$-gradient with respect to the metric.
Gover--\O rsted showed that the trace-free part of $\mathsf E_{2N}$ vanishes at Einstein metrics, which means that $\mathsf E_{2N}(g_E)=Q_{2N,g_E}\, g_E/(2n)$.
For every smooth TT tensor $h$, we obtain
\[
D\mathsf E_{2N}|_{g_E}[h]
=\frac{1}{n-2N}\mathscr P_N^{(n)}(-\cA)h
+\frac{Q_{2N,g_E}}{2n}h.
\]
Indeed, the linearized trace and divergence identities, together with $D Q_{2N}|_{g_E}[h]=0$, show that subtracting the last term gives a TT tensor.
Differentiating the variational definition along the fixed-volume path identifies its pairing with any TT tensor $k$ with $(n-2N)^{-1}D^2\mathcal Q|_{g_E}(h,k)$, so \eqref{eq:TTpolynomial} gives the formula.
Equivalently, the linearization of the trace-free part of $\mathsf E_{2N}$ is $(n-2N)^{-1}\mathscr P_N^{(n)}(-\cA)$ on TT tensors.
The first four Hessian polynomials $\mathscr P_N^{(n)}$ are displayed in \eqref{eq:lowP}.
\end{remark}

\begin{remark}\label{rem:obstruction-comparison}
Appendix \ref{app:obstruction-comparison} compares the linearizations of the obstruction tensor and the extended obstruction tensors in TT directions.
It explains the product factorization when $n=2N$ and the distinction from the formulas when $n>2N$.
Lin--Yuan \cite[Proposition 3.4]{LinYuanII} computed the linearization of the Bach tensor in TT directions at an Einstein metric; for $n=4$, the Bach tensor is the ambient obstruction tensor.
Explicit linearizations of the first three extended obstruction tensors appear in \eqref{eq:extended-first}--\eqref{eq:extended-third}.
\end{remark}

\begin{remark}\label{rem:volume-analogue}
Appendix \ref{app:volume-hessians} proves the analogue of Theorem \ref{thm:ambientHessian} for the total renormalized volume coefficients $\mathcal V_N(g)\coloneqq\int_Mv_N(g)\,dv_g$.
It computes the fixed-volume Hessian of $\mathcal V_N$ in TT directions on closed positive Einstein manifolds, proves its strict negativity on the round sphere, and relates it to the Hessian of $\mathcal Q$.
The first three examples are given in \eqref{eq:volume-low-one}--\eqref{eq:volume-low-three}.
\end{remark}

Matsumoto \cite[Theorem 0.2]{Matsumoto} expressed the second variation of total $Q_{2N}$-curvature when $n=2N$ on conformally Einstein manifolds as a product of shifted Lichnerowicz Laplacians.
His tensor GJMS operators also admit product formulas on weighted TT tensors when $n>2N$ \cite[Theorem 4.2]{Matsumoto}; those formulas do not identify the fixed-volume Hessian of total scalar $Q_{2N}$-curvature when $n>2N$.
Theorem \ref{thm:ambientHessian} computes this Hessian on every smooth closed Einstein manifold $(M^n,g_E)$ satisfying $\Ric_{g_E}=(n-1)g_E$.
Proposition \ref{prop:Matsumoto-comparison} records Matsumoto's Hessian formula for $n=2N$ in our conventions; see \eqref{eq:Matsumoto-critical}.
Guillarmou--Moroianu--Schlenker \cite[Section 7]{GuillarmouMoroianuSchlenker} also reduce linearized Einstein equations to hypergeometric ordinary differential equations in their study of the renormalized-volume Hessian at Fuchsian--Einstein fillings.

\medskip
For the noncompactness construction, we restrict Theorem \ref{thm:ambientHessian} to a finite-dimensional subspace of spherical TT tensors generated by algebraic Weyl tensors.
The third main result describes the change in definiteness of the scale derivative of this restricted quadratic form.
Let $W$ be a nonzero algebraic Weyl tensor on $\R^n$, with its defining identities recalled in \eqref{eq:Weyl-symmetries}.
We consider the polynomial perturbations
\[
H_r(W)_{ij}(x)\coloneqq |x|^{2r}W_{ikjl}x^kx^l,\qquad 0\le r\le3,\qquad
h_a\coloneqq\tau H_0(W)+\sum_{j=1}^3a_jH_j(W).
\]
Here $a\coloneqq(\tau,a_1,a_2,a_3)^{\mathsf T}$, and we set $a_0\coloneqq\tau$ in indexed sums.
The tensors $H_0,H_1,H_2,H_3$ have degrees $2,4,6,8$.
The function $F_{h_a}(\xi,\eps)$ is the quadratic coefficient obtained after solving for the linear scalar correction, multiplied by $8/(n-2N)$; its precise definition is \eqref{eq:modelF}.
The parameters $\xi\in\R^n$ and $\eps>0$ denote the center and scale, respectively, of the standard bubble $U_{\xi,\eps}$ defined in \eqref{eq:roundbubble}.
After normalizing the algebraic Weyl tensor $W$ by \eqref{eq:Wnorm}, we define the symmetric $4\times4$ matrix $S_N(n)$ by
\begin{equation}\label{eq:intro-dilation-matrix}
\left.\partial_\eps F_{h_a}(0,\eps)\right|_{\eps=1}
=a^{\mathsf T}S_N(n)a.
\end{equation}
Thus $S_N(n)$ is obtained by differentiating the restricted Hessian in the bubble scale; its entries are specified in \eqref{eq:Ematrix}.
The entries of $S_N(n)$ are well defined for $n>2N+16$ and admit explicit formulas for real $n$ in that range.
Changing the normalization of $W$ multiplies the matrix by a positive scalar and does not change its numbers of positive and negative eigenvalues or its determinant zeros.

If $S_N(n)$ is negative definite, then $\partial_\eps F_{h_a}(0,\eps)|_{\eps=1}<0$ for every nonzero $a$, so no nonzero choice of $a$ makes $\partial_\eps F_{h_a}(0,\eps)$ vanish at $\eps=1$.
When the matrix becomes indefinite, such choices exist.
With suitable Weyl tensors, Proposition \ref{prop:seed} constructs a polynomial perturbation for which the critical point is nondegenerate in all translation and scale variables.
The following theorem locates this change for large $N\in\mathbb N$.

\begin{maintheorem}\label{thm:cross}
There are constants $a_*\in(16,17)$ and $c_*>0$ and a fixed neighborhood of $a_*$ such that, for all sufficiently large $N\in\mathbb N$, the function $n\mapsto\det S_N(n)$ has exactly one zero $n_N$ with $n_N-2N$ in that neighborhood.
Moreover,
\begin{equation}\label{eq:crossstatement}
n_N=2N+a_*+\frac{c_*}{N}+O(N^{-2})\quad\text{as }N\to\infty,
\end{equation}
where $a_*\approx16.2018714780$ and $c_*\approx13.5128796221$.
The constant $a_*$ is the unique root in the open interval $(16,17)$ of the polynomial $G_0$ in Appendix \ref{app:crosspoly}, and $c_*=-G_1(a_*)/G_0'(a_*)$ for the polynomial $G_1$ in that appendix.
When $n-2N$ lies in that neighborhood, $S_N(n)$ is negative definite for $n<n_N$ and has exactly one positive eigenvalue for $n>n_N$.
\end{maintheorem}

\begin{figure}[htbp]
\centering
\includegraphics[width=0.56\textwidth]{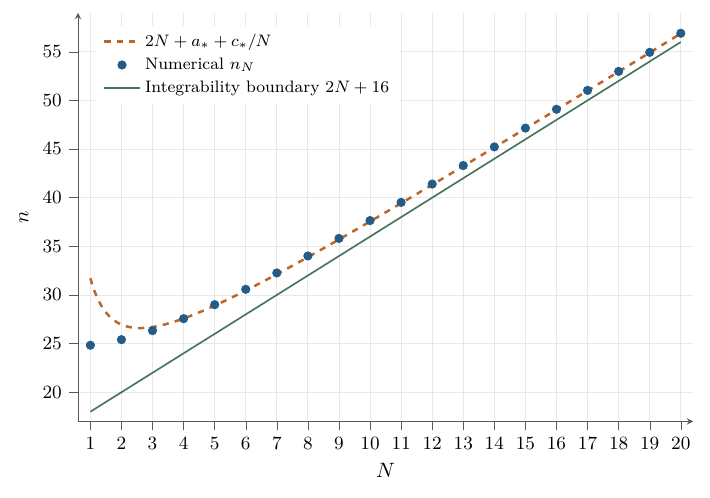}
\caption{The points labeled ``Numerical $n_N$'' are real determinant crossings for the family spanned by $H_0,H_1,H_2,H_3$, compared with \eqref{eq:crossstatement} and the integrability boundary $n=2N+16$.
For $4\le N\le20$, the first integer above each crossing is the sharp bound in Theorem \ref{thm:main}.
The points for $N=1,2,3$ also show the determinant crossings for the cubic family.
For $N=2$, this crossing is approximately $25.406$; this does not conflict with the Wei--Zhao construction at $n=25$, which uses a quartic radial polynomial and therefore lies outside the cubic family considered here.}
\label{fig:dimension-crossing}
\end{figure}

The constants $a_*$ and $c_*$, and the determinant polynomials in Appendix \ref{app:crosspoly} that determine them, are independent of the coefficients $\tau,a_1,a_2,a_3$ chosen for the blow-up profile.
They are determined by the Hessian of $\mathcal Q$ and its scale variation on the four-dimensional family of metric tensors spanned by $H_0,H_1,H_2,H_3$.
Thus they encode the variational structure of $Q_{2N}$ on this family; choosing the fixed coefficient tails in \eqref{eq:fixed-tails} below does not enter their definition.

\begin{question}\label{ques:integrability-gap}
The polynomial symmetric two-tensors of degree at most eight satisfy the integrability condition $n>2N+16$, whereas Theorem \ref{thm:cross} gives $n_N-(2N+16)\to a_*-16\approx0.2018714780$ as $N\to\infty$.
What geometric mechanism explains this small positive limiting gap between the determinant crossing and the integrability boundary?
\end{question}

Since $16<a_*<17$, the first integer dimension above the local transition is $2N+17$ for all sufficiently large $N\in\mathbb N$, which explains the eventual dimension range in Theorem \ref{thm:main}.
Theorem \ref{thm:algebraic} establishes this bound for every $N\ge18$, as well as the stated bounds for $4\le N\le17$, using the polynomial inequalities in Proposition \ref{prop:tail-certificates}.
Figure \ref{fig:dimension-crossing} compares the numerically computed crossings for $1\le N\le20$ with the two-term approximation in \eqref{eq:crossstatement}.
The continuation to real $n$ concerns only the matrix formulas, not a curvature equation in noninteger dimension.

For the construction with $N\ge4$, the polynomial metric perturbations have degree at most eight, and all the stated dimensions satisfy $n>2N+16$.
The inequality $n>2N+16$ ensures convergence of the compactly supported approximations in the tensor Sobolev space $H^N$.
Specifically, let $\chi_R(x)\coloneqq\chi(|x|/R)$, where $\chi$ is a fixed smooth cutoff equal to one on $[0,1]$ and zero on $[2,\infty)$.
For each polynomial profile $H$, let $\bar H\coloneqq\Omega^2H$, where $\Omega$ is defined in \eqref{eq:roundbubble}.
The compactified tensors $\Omega^2\chi_RH$ converge to $\bar H$ in $H^N(\Sph^n;\operatorname{Sym}^2T^*\Sph^n)$ as $R\to\infty$, with the Sobolev norm defined by $g_c$.
The corresponding source terms $f_{\chi_RH,\xi,\eps}$, defined in \eqref{eq:linear-metric-source}, converge to $f_{H,\xi,\eps}$ in $\mathcal H^*=(\dot H^N(\R^n))^*$ as $R\to\infty$; see Lemmas \ref{lem:cutoff-convergence} and \ref{lem:source-limits}.

We call the three nonconstant coefficients $(a_1,a_2,a_3)$ of the radial cubic polynomial the \emph{coefficient tail}.
The following two choices suffice to prove that $S_N(n)$ is indefinite throughout the dimension ranges in Theorem \ref{thm:main}:
\begin{equation}\label{eq:fixed-tails}
(a_1,a_2,a_3)=
\begin{cases}
(100,-20,1) &\text{for } 4\le N\le17,\\
(150,-29,1) &\text{for } N\ge18.
\end{cases}
\end{equation}
Each tail is independent of both $n$ and $N$ within its stated range.
Once indefiniteness is established, Proposition \ref{prop:seed} selects coefficients for which the reduced energy has a nondegenerate critical point in the bubble center and scale.

\medskip
\noindent\textbf{Novelty of the paper.}
The first difficulty is to compute the Hessian of total $Q_{2N}$-curvature in TT directions for arbitrary $N \in \mathbb{N}$ and $n>2N$, despite the complexity of the underlying nonlinear invariants.
Graham's construction \cite[Theorem 1.2]{GrahamExtended} and Juhl's formulas \cite[Theorem 1.2]{Juhl} provide a systematic route from the Schouten tensor and extended obstruction tensors, through the ambient metric and renormalized volume coefficients, to $Q_{2N}$.
Following this route by direct curvature expansion requires solving the nonlinear ambient equations and then differentiating the resulting formulas.
Graham describes the construction of the volume coefficients as an ``indirect, highly nonlinear, inductive algorithm'' \cite[p. 1958]{GrahamExtended}, a difficulty also emphasized by Guo--Han--Li \cite[pp. 258--259]{GuoHanLi}.
This complexity is already visible in the second extended obstruction tensor $\boldsymbol\Omega^{(2)}$ and the fourth volume coefficient $v_4$, whose curvature expressions involve derivatives and multiple nonlinear contractions, although $v_4$ has a compact expression in terms of the extended obstruction tensors \cite[Section 2]{GrahamExtended}.

To avoid these full curvature expansions, we use Fefferman--Graham's ambient proof of Juhl's formulas \cite{FGJuhl}, which expresses the GJMS operators and $Q_{2N}$-curvature through the ambient Laplacian acting on functions and the relative volume density.
Proposition \ref{prop:holographic-Q} restates the formula \cite[(2.7)]{FGJuhl} in our notation and sign conventions, with $v^{1/2}$ corresponding to their function $w$.
Starting from the conjugation identity \cite[(2.4)]{FGJuhl}, Lemma \ref{lem:scalar-ambient-obstruction} constructs the homogeneous ambient harmonic extension of the constant function $1$ on $M$ to the required Taylor order.
Formula \eqref{eq:obstructionnormalization} then recovers $Q_{2N}$ from the coefficient of $\rho^{N-1}$ in the remaining error in the harmonic equation.

Matsumoto's ambient linearization and Hessian formula for $n=2N$ \cite{Matsumoto} provide the variational background.
For $n>2N$, we linearize the ambient Ricci-flat equations at the Einstein metric $g_E$ before taking the required derivatives in the ambient normal variable $\rho$, introduced in Subsection \ref{subsec:ambient-preliminaries}.
Proposition \ref{prop:TTambient} shows that the linearized equations preserve TT tensors and reduce to the scalar ordinary differential equation \eqref{eq:TTambient} on each Lichnerowicz eigenspace.
The Taylor recurrence \eqref{eq:qrecurrence} then computes the required coefficients of the first variation of $g_\rho$ for arbitrary $N$.

The second variation of $g_\rho$ is needed only through the integrated mixed variation of $\log\det g_\rho$.
The bilinear version of \eqref{eq:ThetaODE} determines the integrated mixed variation of $\log\det g_\rho$ from the first variations of $g_\rho$.
After differentiating \eqref{eq:obstructionnormalization} in TT directions, we use the integration-by-parts identity \eqref{eq:ambient-integration-by-parts} to eliminate the unknown mixed variation of the harmonic extension.
This yields the Hessian formula \eqref{eq:WronskianHessian} without computing the full second variation of the ambient metric.
Equation \eqref{eq:ThetaODE} also determines the Hessians of total renormalized volume coefficients in TT directions, as shown in Appendix \ref{app:volume-hessians}.
For order six, \cite[Remark 10.14]{GongKimWei} simplifies the Pohozaev quadratic form by pairing the second metric variation of the GJMS operator with the bubble and its dilation field; our formula for $D^2\mathcal Q|_{g_E}$ applies to every order with $n>2N$.

\medskip
The second difficulty is to keep the algebra bounded as $N$ grows.
Although the Taylor recurrence \eqref{eq:qrecurrence} requires only finitely many steps at each order, the number of steps depends on $N$.
For $H_0,H_1,H_2,H_3$, we instead derive ambient extensions of degree at most eight in the variable $R$ of \eqref{eq:ambient-R}.
The telescoping identity \eqref{eq:telescoping} evaluates the remaining order-dependent sum through the linear system \eqref{eq:gaussian-system}, whose size is at most eight and whose determinant has an explicit product formula.
Hypergeometric telescoping is classical \cite{Gosper}, and related summation identities have been used for GJMS operators and $Q_{2N}$-curvatures \cite{JuhlKrattenthaler}.
Here the polynomial extensions and the uniform degree bound yield the entries of the energy matrix $E$ as rational functions of $n$ and $N^2$.

These rational formulas make it possible to prove signs over infinite parameter ranges using a single quadratic discriminant.
For $N\ge18$, the discriminant \eqref{eq:tail-discriminant}, after multiplication by a positive denominator and shifts of the variables, is a polynomial with positive coefficients throughout $n\ge2N+17$.
For $4\le N\le17$, the same argument with the other coefficient tail uses one polynomial in the shifted dimension per order; Proposition \ref{prop:tail-certificates} gives the exact coefficient checks.
An expansion of the normalized determinant also quantifies the local large-$N$ crossing in Theorem \ref{thm:cross}.
The crossing describes the change in definiteness within the family spanned by $H_0,H_1,H_2,H_3$.

For the polynomial metric perturbations $h$ generated by algebraic Weyl tensors, Proposition \ref{prop:energybridge} and the normalization in \eqref{eq:modelF} identify $F_h(0,\eps)$ with $D^2\mathcal Q|_{g_c}(\Omega^2h(\eps\,\cdot),\Omega^2h(\eps\,\cdot))$.
Lemma \ref{lem:weyl-TT} shows that the basis tensors $\Omega^2H_0,\ldots,\Omega^2H_3$ belong to $H^N(\Sph^n;\operatorname{Sym}^2T^*\Sph^n)$, defined with respect to $g_c$, and that radial cutoffs preserve the TT condition.
When differentiating $F_h(\xi,\eps)$ with respect to the bubble center $\xi$, the compactifications of the translated perturbations need not be TT with respect to $g_c$.
We therefore retain the term $-\langle f_H,L_U^{-1}f_{\widetilde H}\rangle/2$ in $\cE(H,\widetilde H)$, defined in \eqref{eq:correctedE}, for the relevant metric directions $H,\widetilde H$.
Using the indefiniteness of $S_N(n)$, the symmetries of the chosen Weyl tensor, and Lemma \ref{lem:nullcone}, Proposition \ref{prop:seed} selects coefficients for which $(0,1)$ is a nondegenerate critical point of $F_h$.
The argument ensures that the Hessian in the bubble center and scale is invertible but does not determine its sign, so the critical point may be a saddle.
The $C^1$ expansion \eqref{eq:fixedreducedexpansion} gives uniform convergence of the reduced-energy gradients, and the nonzero local degree of $\nabla F_h$ then preserves a critical point even when it is a saddle; see Theorem \ref{thm:fixed-metric-transfer}.
The projected reduction and fixed-metric gluing adapt the framework of \cite{BrendleBlowup,BrendleMarques,WeiZhao,GongKimWei}.

\medskip
\noindent\textbf{Organization of the paper.}
Section \ref{sec:preliminaries} fixes the conventions, recalls the ambient construction of the GJMS operators and $Q_{2N}$-curvature, and introduces the variational formulation used in the construction.
Section \ref{sec:ambient} proves Theorem \ref{thm:ambientHessian} and compares the resulting Hessian formula with that of Matsumoto \cite{Matsumoto} for $n=2N$.
Section \ref{sec:reduction} constructs an inverse for the projected linearized equation and derives the reduced energy, including the scalar correction.
It also justifies approximation of the polynomial metric perturbations by compactly supported tensors.
Section \ref{sec:weylmatrix} restricts the fixed-volume Hessian of $\mathcal Q$ at $g_c$ to the compactified metric perturbations generated by algebraic Weyl tensors and selects a nondegenerate critical point of the reduced energy.
Section \ref{sec:boundedambient} derives exact formulas for the associated matrix entries and proves the indefiniteness of the matrix governing the scale derivative of the reduced energy.
Section \ref{sec:analysis} combines localized metric perturbations into a single smooth metric on $\Sph^n$.
Section \ref{sec:applications} then completes the noncompactness proof and establishes the asymptotic dimension transition, proving Theorems \ref{thm:main} and \ref{thm:cross}, respectively.
Appendices \ref{app:cauchy}--\ref{app:certificates} record the rational summands, determinant coefficients, and exact-arithmetic verification used in the matrix analysis.
Appendix \ref{app:obstruction-comparison} compares the linearizations of the obstruction tensor and extended obstruction tensors in TT directions and derives the factorization of the Hessian operator for $n=2N$.
Appendix \ref{app:volume-hessians} uses \eqref{eq:ThetaODE} to compute the Hessians of total renormalized volume coefficients in TT directions and compares them with the Hessian of $\mathcal Q$.

\medskip
\noindent\textbf{Acknowledgement.}
L. Gong thanks Mingxiang Li, Yanyan Li and Wei Yuan for helpful discussions.
S. Kim thanks Haixia Chen for stimulating discussion.
L. Gong and J. Wei gratefully acknowledge the support of the Research Grants Council of Hong Kong (RGC) for the project titled \emph{On Fujita equation in the critical or supercritical regime}.
S. Kim acknowledges support from the Basic Science Research Program through the National Research Foundation of Korea (NRF), funded by the Ministry of Science and ICT (RS2025-00558417), as well as from the Samsung Science and Technology Foundation (Project Number SSTF-BA2601-01).

\medskip
\noindent\textbf{AI assistance.}
The authors formulated the method of relating compactness and noncompactness problems to the second variation of total $Q_{2N}$-curvature in TT directions and analyzing it through ambient metric theory independently of AI assistance in September 2025, following completion of \cite{GongKimWei}.
OpenAI's Codex was subsequently used to assist with algebraic calculations, ambient recurrences and polynomial coefficient choices, develop exact-arithmetic verification programs, refine exposition, notation and proof details, and carry out additional proof and consistency checks.
The finite computations used in the paper are specified in Appendix \ref{app:certificates}.
The authors are responsible for the final text and mathematical claims.

\section{Preliminaries}\label{sec:preliminaries}
We collect the conventions, ambient equations, and variational framework used to compute the fixed-volume Hessian of $\mathcal Q$ in TT directions and to construct the noncompactness examples.

Spatial tensor indices $i,j,k,l,m$ range from $1$ to $n$; repeated spatial indices are summed.
An index labeling a polynomial tensor is summed only when a summation sign is displayed.
We write $u_+\coloneqq\max\{u,0\}$ and set $B_R(x)\coloneqq\{y\in\R^n:|y-x|<R\}$.
The symbols $R_g$ and $\Ric_g$ denote scalar and Ricci curvature, respectively.
All tensor contractions on $\R^n$ use the Euclidean metric $\delta$ unless another metric is specified.

A polynomial symmetric two-tensor $H=(H_{ij})$ on $\R^n$ means a covariant two-tensor with $H_{ij}=H_{ji}$ whose components are real polynomials in the standard Euclidean coordinates.
For an integer $d\ge0$, the condition $\deg H\le d$ means that every component has total degree at most $d$; homogeneity of degree $d$ also refers to the components.
We use such tensors as directions of metric variation.

For a symmetric two-tensor $h$, the components of $g\exp(g^{-1}h)$ are $g_{ik}(\exp(g^{-1}h))^k{}_j$.
Here $g^{-1}h$ is an endomorphism and $\exp$ denotes its matrix exponential.
In Euclidean coordinates, $\exp(h)$ abbreviates $\delta\exp(\delta^{-1}h)$.
For tensors $h,k$ that are trace-free with respect to $g_E$, fixed-volume second variations mean mixed derivatives along
\begin{equation}\label{eq:fixed-volume-path}
g_E\exp(s_1g_E^{-1}h+s_2g_E^{-1}k)
\quad\text{at }s_1=s_2=0.
\end{equation}
On an Einstein metric $g_E$ with $\Ric_{g_E}=(n-1)g_E$, our Lichnerowicz Laplacian convention is
\begin{equation}\label{eq:Aconvention}
(\Delta_Lh)_{ij}=\nabla^k\nabla_k h_{ij}
+2R_{ikjl}h^{kl}-2(n-1)h_{ij},\qquad
\cA\coloneqq-(\Delta_L+2(n-1)).
\end{equation}
An \emph{algebraic Weyl tensor} is a real four-tensor $W=(W_{ijkl})$ satisfying
\begin{equation}\label{eq:Weyl-symmetries}
W_{ijkl}=-W_{jikl}=-W_{ijlk}=W_{klij},\qquad
W_{ijkl}+W_{iklj}+W_{iljk}=0,\qquad
\sum_{i=1}^n W_{ijil}=0.
\end{equation}
It need not be the curvature tensor of a previously chosen metric.

For a scalar function or covariant tensor $T$ on $(M,g)$, we use
\begin{equation}\label{eq:basic-norms}
\|T\|_{L^r(M,g)}\coloneqq\left(\int_M|T|_g^r\,dv_g\right)^{\frac{1}{r}},\qquad
\|T\|_{L^\infty(M,g)}\coloneqq\mathop{\mathrm{ess\,sup}}_M|T|_g
\quad(1\le r<\infty).
\end{equation}
For an integer $m\ge0$, we define the Sobolev and differentiability norms of $T$ by
\begin{equation}\label{eq:Sobolev-norms}
\|T\|_{W^{m,r}(M,g)}^r
\coloneqq\sum_{j=0}^m\int_M|\nabla_g^jT|_g^r\,dv_g,\qquad 1\le r<\infty,\qquad
\|T\|_{C^m(g)}\coloneqq\sum_{j=0}^m\sup_M|\nabla_g^jT|_g.
\end{equation}
We write $H^m\coloneqq W^{m,2}$.
We define the scalar Sobolev space $H^m(\Sph^n)$ using the round metric $g_c$.
For symmetric covariant two-tensors we write $H^m(\Sph^n;\operatorname{Sym}^2T^*\Sph^n)$; it is the completion of smooth sections of $\operatorname{Sym}^2T^*\Sph^n$ in \eqref{eq:Sobolev-norms}, with $g=g_c$.
Equivalently, its elements and their weak covariant derivatives through order $m$ belong to $L^2(\Sph^n,g_c)$.
For either tensor or scalar norms, $H^m(g_c)$ abbreviates this same derivative norm.
Throughout the paper, when no metric is specified in a Sobolev or differentiability norm on $\Sph^n$, we use $g_c$; any fixed smooth metric defines equivalent norms.
On $\R^n$, the measure is $dx$ and $\partial^jT$ is the full array of order-$j$ Euclidean derivatives.

\medskip
Throughout the paper, including the appendices, $(\alpha)_r$ denotes the Pochhammer symbol, or rising factorial:
\[
(\alpha)_0\coloneqq1,\qquad
(\alpha)_r\coloneqq\prod_{j=0}^{r-1}(\alpha+j)
\quad(r\ge1).
\]
Every empty product equals $1$.
In particular, $\prod_{j=1}^r(x-2j)=2^r(x/2-r)_r$.
The Gauss hypergeometric function ${}_2F_1$ and the generalized hypergeometric function ${}_3F_2$ are defined by \cite[Sections 15.2 and 16.2]{DLMF}
\[
\begin{aligned}
{}_2F_1(\alpha_1,\alpha_2;\beta_1;z)
&=\sum_{r=0}^{\infty}
\frac{(\alpha_1)_r(\alpha_2)_r}{(\beta_1)_r}\frac{z^r}{r!},\\
{}_3F_2\!\left(\begin{matrix}\alpha_1,\alpha_2,\alpha_3\\\beta_1,\beta_2\end{matrix};z\right)
&=\sum_{r=0}^{\infty}
\frac{(\alpha_1)_r(\alpha_2)_r(\alpha_3)_r}
{(\beta_1)_r(\beta_2)_r}\frac{z^r}{r!}.
\end{aligned}
\]
For lower parameters outside $\{0,-1,-2,\ldots\}$, these series converge for $|z|<1$.
When a lower parameter is a nonpositive integer, we use only the explicitly indicated truncation, ending before a denominator vanishes.
In particular, modulo $z^N$ with a lower parameter $1-N$, only terms with $0\le r<N$ are retained; no value of the infinite series at that parameter is required.
The notation $[z^r]f$ denotes the coefficient of $z^r$ in a Taylor or Laurent expansion of $f$; $B(a,b)$ denotes the beta function.

\subsection{Ambient equations and the GJMS construction}
\label{subsec:ambient-preliminaries}
For a smooth positive function $u$, we set $\hat g\coloneqq u^{4/(n-2N)}g$.
The conformal covariance of the GJMS operator is
\begin{equation}\label{eq:covariance}
P_{2N,\hat g}(\varphi)
=u^{-\frac{n+2N}{n-2N}}P_{2N,g}(u\varphi).
\end{equation}
For a fixed Riemannian representative $g$ on $M^n$, the ambient metric $\tilde{g}$ is defined near $\{\rho=0\}$ in the $(n+2)$-dimensional space $(0,\infty)_{\mathsf t}\times M\times\R_\rho$.
It has Lorentzian signature and normal form \cite[Theorem 2.9]{FG}
\begin{equation}\label{eq:ambientnormal}
\tilde{g}\coloneqq2\rho\,d{\mathsf t}^2+2\mathsf t\,d\mathsf t\,d\rho
+{\mathsf t}^2 g_\rho,\qquad g_0=g.
\end{equation}
Here $\mathsf t>0$ and $\rho$ is near zero.
At $\rho=0$, the slice with $\mathsf t$ fixed has metric ${\mathsf t}^2g$.
Thus $\mathsf t$ parametrizes conformal rescalings, while $\rho$ is transverse to the hypersurface $\rho=0$.
The Taylor coefficients of $g_\rho$ are obtained by imposing $\Ric(\tilde{g})=0$ successively at $\rho=0$.
For even $n$, the expansion can be obstructed at order $n/2$, but the coefficients needed for $P_{2N,g}$ are uniquely determined when $2N<n$.
We use only these finite Taylor expansions, not an exact Ricci-flat metric away from $\rho=0$.

Let $\widetilde\Delta$ denote the Laplacian of the ambient metric $\tilde{g}$, acting on functions.
For a smooth function $\varphi$ on $M$, we choose a smooth extension $f(x,\rho)$ with $f(x,0)=\varphi(x)$.
The ambient formula for the GJMS operator is
\begin{equation}\label{eq:ambient-gjms}
P_{2N,g}\varphi
=\left.(-\widetilde\Delta)^N
\bigl({\mathsf t}^{N-\frac{n}{2}}f(x,\rho)\bigr)\right|_{\mathsf t=1,\,\rho=0}.
\end{equation}
The restriction is independent of $f$; its homogeneity $N-n/2$ determines the conformal weights in \eqref{eq:covariance} \cite{GJMS}.
For $\varphi=1$, \eqref{eq:ambient-gjms} determines $Q_{2N,g}$ through $P_{2N,g}(1)=(n-2N)Q_{2N,g}/2$.
We differentiate the Ricci-flat equations and this scalar construction at the Einstein reference metric to compute the second variation of total $Q_{2N}$-curvature without expanding its full curvature formula.

Juhl's explicit formulas \cite[Theorems 1.1--1.2]{Juhl} express the GJMS operators and $Q_{2N}$-curvatures through compositions of natural second-order operators and coefficients of the volume expansion.
Fefferman--Graham \cite[Sections 1--2]{FGJuhl} derive those identities from \eqref{eq:ambient-gjms}, after conjugating the ambient Laplacian by the square root of the relative volume density.
We recall their ordered holographic formula \cite[(2.7)]{FGJuhl}.
We write $g_0=g$ and set
\[
v=v(g,\rho)\coloneqq\left(\frac{\det g_\rho}{\det g_0}\right)^{\frac{1}{2}},
\qquad
v^{\frac{1}{2}}=\left(\frac{\det g_\rho}{\det g_0}\right)^{\frac{1}{4}}.
\]
Thus $v(g,0)=1$.
We write $v^{1/2}$ for the quantity denoted by $w$ in \cite{FGJuhl}.
In \eqref{eq:holographic-operators}, the differential $d$ and divergence act only in the variables on $M$; we compute the divergence with respect to $g_0$.
We define the scalar potential $U_\rho$ and, for $k\in\R$, the second-order operators
\begin{equation}\label{eq:holographic-operators}
\begin{aligned}
U_\rho&\coloneqq v^{-\frac{1}{2}}
\left(-2\rho\partial_\rho^2+(n-2)\partial_\rho
+\operatorname{div}_{g_0}(g_\rho^{-1}d)\right)v^{\frac{1}{2}},\\
\mathcal R_k&\coloneqq
-2\rho\partial_\rho^2+2k\partial_\rho
+\operatorname{div}_{g_0}(g_\rho^{-1}d)-U_\rho.
\end{aligned}
\end{equation}

\begin{proposition}\label{prop:holographic-Q}
For $n>2N$, the ambient family in \eqref{eq:ambientnormal} determines $Q_{2N,g}$ by
\begin{equation}\label{eq:holographic-Q}
(-1)^NQ_{2N,g}
=-2\left.
\left[\prod_{j=1}^{N-1}\mathcal R_{-N+2j-1}\right]
\partial_\rho v^{\frac{1}{2}}\right|_{\rho=0}.
\end{equation}
The ordered product is
\[
\mathcal R_{1-N}
\mathcal R_{3-N}\cdots
\mathcal R_{N-3},
\]
with the rightmost factor acting first; the factors need not commute.
For $N=1$, the product is the identity.
Only the finite ambient Taylor coefficients required at this order enter the formula.
\end{proposition}

Section \ref{sec:ambient} uses these operators to construct the homogeneous ambient harmonic extension of the constant function $1$ on $M$.
The coefficient identified in \eqref{eq:obstructionnormalization} determines $Q_{2N}$; differentiating it twice yields the Hessian of total $Q_{2N}$-curvature without expanding the full nonlinear curvature invariants.

For the ambient metric \eqref{eq:ambientnormal}, the Ricci-flat equations below determine the Taylor coefficients of $g_\rho$ at $\rho=0$.
Primes denote $\rho$-derivatives, and all contractions and covariant derivatives in this display are taken with $g_\rho$:
\begin{align}
0={}&\rho g''_{ij}-\rho g^{kl}g'_{ik}g'_{jl}
+\frac\rho2 g^{kl}g'_{kl}g'_{ij}
-\left(\frac n2-1\right)g'_{ij}
-\frac12g^{kl}g'_{kl}g_{ij}+\Ric(g_\rho)_{ij},\label{eq:FGevolution}\\
0={}&\frac12g^{kl}(\nabla_k g'_{il}-\nabla_i g'_{kl}),\label{eq:FGmixed}\\
0={}&-\frac12\tr(g_\rho^{-1}g''_\rho)
+\frac14\tr\bigl((g_\rho^{-1}g'_\rho)^2\bigr).
\label{eq:FGnormal}
\end{align}
These are the normal-form equations of Fefferman--Graham \cite[(3.17)]{FG}.
Evaluating \eqref{eq:FGevolution} at zero yields
\[
g'_0=2\mathsf P_g,\qquad
\mathsf P_g\coloneqq\frac1{n-2}\left(\Ric_g-\frac{R_g}{2(n-1)}g\right).
\]
Here $\mathsf P_g$ is the \emph{Schouten tensor}.

\begin{lemma}\label{lem:Einstein-ambient}
If $\Ric(g_E)=(n-1)g_E$, the family
\begin{equation}\label{eq:roundambient}
g_\rho^0=a^2g_E,\qquad a\coloneqq1+\frac{\rho}{2},
\end{equation}
solves the normal-form Ricci equations with initial metric $g_E$.
\end{lemma}
\begin{proof}
This is the Einstein ambient metric given immediately after \cite[(3.17)]{FG}, with Schouten constant $1/2$; it follows by substitution in \eqref{eq:FGevolution}--\eqref{eq:FGnormal}.
\end{proof}

\subsection{Energy and bubbles}
On a smooth closed manifold $M$, we set
\begin{equation}\label{eq:geometricJ}
\mathcal J(g,u)\coloneqq\frac12\int_M uP_{2N,g}u\,dv_g
-\frac{(n-2N)c_{n,N}}{2n}\int_M(u_+)^{\frac{2n}{n-2N}}\,dv_g.
\end{equation}
\begin{lemma}\label{lem:energy-invariance}
The functional $\mathcal J$ is invariant under simultaneous pullback of the metric and scalar function by a diffeomorphism, and under
\[
(g,u)\mapsto(e^{2\omega}g,e^{-\frac{(n-2N)\omega}{2}}u).
\]
\end{lemma}
\begin{proof}
Naturality and change of variables give diffeomorphism invariance.
For conformal invariance, \eqref{eq:covariance} cancels the volume weight in the quadratic term, while the nonlinear term transforms with the inverse volume weight $e^{-n\omega}$.
\end{proof}
For a smooth compactly supported trace-free tensor $h$, the volume form of $\exp(h)$ is exactly $dx$.
On $\R^n$, the corresponding functional is
\begin{equation}\label{eq:EuclideanJ}
J_h(u)\coloneqq\frac12\int_{\R^n}uP_{2N,\exp(h)}u\,dx
-\frac{(n-2N)c_{n,N}}{2n}\int_{\R^n}(u_+)^{\frac{2n}{n-2N}}\,dx.
\end{equation}

\medskip
For a tensor $h$ as in \eqref{eq:EuclideanJ}, the Euler--Lagrange operator is
\begin{equation}\label{eq:weak-residual}
\mathscr E_h(u)\coloneqq DJ_h(u)
=P_{2N,\exp(h)}u-c_{n,N}(u_+)^{\frac{n+2N}{n-2N}}.
\end{equation}
Thus the full weak equation is $\langle\mathscr E_h(u),\phi\rangle=0$ for every $\phi$ in the energy space $\mathcal H$ of \eqref{eq:homogeneous-norm}.
We fix a smooth cutoff $\chi:[0,\infty)\to[0,1]$ equal to one on $[0,1]$ and zero on $[2,\infty)$, and write $\chi_R(x)\coloneqq\chi(|x|/R)$.

We use stereographic projection from the north pole $(0,\ldots,0,1)\in\Sph^n\subset\R^{n+1}$.
In the resulting $x$-coordinates on the sphere minus the north pole,
\begin{equation}\label{eq:roundbubble}
g_c=\Omega^2\delta,\qquad \Omega(x)\coloneqq\frac{2}{1+|x|^2},\qquad
U_{\xi,\eps}(x)\coloneqq
\left(\frac{2\eps}{\eps^2+|x-\xi|^2}\right)^{\frac{n-2N}{2}}.
\end{equation}
Here $\delta$ is the Euclidean metric.
We set $U\coloneqq U_{0,1}$; then
\begin{equation}\label{eq:bubble-equation}
(-\Delta)^NU_{\xi,\eps}=c_{n,N}U_{\xi,\eps}^{\frac{n+2N}{n-2N}}
\quad\text{in }\R^n.
\end{equation}

Let $\mathcal H\coloneqq\dot H^N(\R^n)$ be the completion of $C_c^\infty(\R^n)$ for the norm
\begin{equation}\label{eq:homogeneous-norm}
\|u\|_{\mathcal H}^2
\coloneqq\int_{\R^n}|\partial^Nu|^2\,dx
=\int_{\R^n}u(-\Delta)^Nu\,dx,\qquad u\in C_c^\infty(\R^n).
\end{equation}
The Sobolev inequality identifies the completion with a subspace of $L^{2n/(n-2N)}(\R^n)$; the last pairing in \eqref{eq:homogeneous-norm} is extended by continuity.
Its dual is denoted by $\mathcal H^*$.
Brackets between an element of $\mathcal H^*$ and an element of $\mathcal H$ denote the dual pairing; tensor brackets use the $L^2$ pairing of the specified reference metric.
On the round sphere, for symmetric covariant two-tensors $H,\widetilde H$, this pairing is
\[
\langle H,\widetilde H\rangle_{L^2(g_c)}
\coloneqq\int_{\Sph^n}\langle H,\widetilde H\rangle_{g_c}\,dv_{g_c}.
\]
For $u,v\in C_c^\infty(\R^n)$, we define $\langle u,v\rangle_{\mathcal H}\coloneqq\int_{\R^n}u(-\Delta)^Nv\,dx$ and extend this inner product by continuity to $\mathcal H$.
For Banach spaces $X,Y$, our dual and operator norms are
\[
\|\ell\|_{X^*}\coloneqq\sup_{\|v\|_X\le1}|\ell(v)|,\qquad
\|A\|_{X\to Y}\coloneqq\sup_{\|v\|_X\le1}\|Av\|_Y.
\]
In particular, $H^{-m}(\Sph^n)$ denotes the dual of the scalar $H^m(\Sph^n)$ under the distributional extension of the $L^2$ pairing.
The tensor space $H^{-m}(\Sph^n;\operatorname{Sym}^2T^*\Sph^n)$ is defined analogously using the tensor $L^2(g_c)$ pairing.
For a relatively compact open parameter set $\mathcal U\subset\R^n\times(0,\infty)$, an integer $r\ge0$, and a $C^r$ map $F$ with values in $X$ defined near $\overline{\mathcal U}$, we use
\begin{equation}\label{eq:parameter-norm}
\|F\|_{C^r(\overline{\mathcal U};X)}
\coloneqq\sum_{|\beta|\le r}\sup_{(\xi,\eps)\in\overline{\mathcal U}}
\|\partial_{\xi,\eps}^\beta F(\xi,\eps)\|_X.
\end{equation}
Here $\partial_{\xi,\eps}^\beta=\partial_{\xi_1}^{\beta_1}\cdots\partial_{\xi_n}^{\beta_n}\partial_\eps^{\beta_{n+1}}$ for a multi-index $\beta$ in the $n+1$ parameters.
The set $\mathcal U$ is specified in each application; $C^r(\overline{\mathcal U})$ means $X=\R$, and $o_{C^r(\overline{\mathcal U};X)}(1)$ means convergence to zero in this norm.
The Sobolev inequality implies
\begin{equation}\label{eq:dual-bound}
\|f\|_{\mathcal H^*}\le C\|f\|_{L^{\frac{2n}{n+2N}}(\R^n)}.
\end{equation}
The translation and dilation Jacobi fields are
\[
Z_{\xi,\eps,i}\coloneqq\partial_{\xi_i}U_{\xi,\eps}\quad(1\le i\le n),
\qquad Z_{\xi,\eps,n+1}\coloneqq\partial_\eps U_{\xi,\eps}.
\]
We set
\[
K_{\xi,\eps}\coloneqq\operatorname{span}\{Z_{\xi,\eps,\alpha}:1\le\alpha\le n+1\},
\qquad
L_{U_{\xi,\eps}}\coloneqq(-\Delta)^N-\frac{n+2N}{n-2N}c_{n,N}U_{\xi,\eps}^{\frac{4N}{n-2N}}.
\]
Proposition \ref{prop:projected-inverse} shows that they span the full kernel of $L_{U_{\xi,\eps}}$.
We write $K\coloneqq K_{0,1}$ and $L_U\coloneqq L_{U_{0,1}}$.
The space $K_{\xi,\eps}^\perp$ consists of the functions in $\mathcal H$ orthogonal to every $Z_{\xi,\eps,\alpha}$, $\alpha\in\{1,\ldots,n+1\}$, in the $\mathcal H$ inner product.
Since $L_{U_{\xi,\eps}}Z_{\xi,\eps,\alpha}=0$, this is equivalent to
\begin{equation}\label{eq:weighted-orthogonality}
\int_{\R^n}U_{\xi,\eps}^{\frac{4N}{n-2N}}Z_{\xi,\eps,\alpha}v\,dx=0
\qquad(1\le\alpha\le n+1).
\end{equation}
The Gram matrix of the functions $Z_{\xi,\eps,1},\ldots,Z_{\xi,\eps,n+1}$ in $\mathcal H$ is
\begin{equation}\label{eq:Gram-matrix}
G_{\xi,\eps}\coloneqq\bigl(\langle Z_{\xi,\eps,\alpha},Z_{\xi,\eps,\beta}\rangle_{\mathcal H}\bigr)_{\alpha,\beta=1}^{n+1}.
\end{equation}
\begin{lemma}\label{lem:Jacobi-identifications}
The matrix $G_{\xi,\eps}$ is positive definite.
The orthogonal projection onto $K_{\xi,\eps}^\perp$ is
\begin{equation}\label{eq:Jacobi-projection}
\Pi_{\xi,\eps}v\coloneqq v-\sum_{\alpha,\beta=1}^{n+1}Z_{\xi,\eps,\alpha}(G_{\xi,\eps}^{-1})_{\alpha\beta}
\langle Z_{\xi,\eps,\beta},v\rangle_{\mathcal H}.
\end{equation}
Both $G_{\xi,\eps}$ and $\Pi_{\xi,\eps}$ depend smoothly on $(\xi,\eps)$.
For $(\xi,\eps)$ near any fixed $(\xi_0,\eps_0)\in\mathbb R^n\times(0,\infty)$, the restriction $\Pi_{\xi,\eps}|_{K_{\xi_0,\eps_0}^\perp}:K_{\xi_0,\eps_0}^\perp\to K_{\xi,\eps}^\perp$ is an isomorphism.
\end{lemma}
\begin{proof}
At $(\xi,\eps)=(0,1)$ the translation fields are independent odd functions, while the dilation field is nonzero and radial.
Translation and dilation preserve this independence, so $G_{\xi,\eps}$ is positive definite.
The Gram-matrix formula gives \eqref{eq:Jacobi-projection}, and \eqref{eq:roundbubble} implies smooth dependence on $(\xi,\eps)$.
Near $(\xi_0,\eps_0)$, the compositions with $\Pi_{\xi_0,\eps_0}$ are small operator-norm perturbations of the identity, and hence are invertible.
This proves the last assertion.
\end{proof}
These maps identify $K_{\xi,\eps}^\perp$ with $K_{\xi_0,\eps_0}^\perp$.
Parameter derivatives of operators on $K_{\xi,\eps}^\perp$ are taken after these identifications; operator and dual norms use the inherited $\mathcal H$ norm.

\begin{proposition}\label{prop:jointcritical}
The pair $(g_c,1)$ is stationary for $\mathcal J$ with respect to both the metric and the scalar variable.
The Euclidean pair $(\delta,U_{\xi,\eps})$ has the same property for compactly supported metric variations.
\end{proposition}
\begin{proof}
The derivative of \eqref{eq:geometricJ} in $u$ vanishes at $(g_c,1)$ because $P_{2N,g_c}1=c_{n,N}$.
The derivative in the metric is invariant under round isometries and hence is a constant multiple of $\int_{\Sph^n}\tr_{g_c}h\,dv_{g_c}$.
On the homothetic path it is the derivative at zero of
\[
\mathcal J(e^{2r}g_c,1)
=c_{n,N}\Vol(g_c)\left(\frac12e^{(n-2N)r}-\frac{n-2N}{2n}e^{nr}\right),
\]
whose derivative at $r=0$ vanishes.
Conformal covariance \eqref{eq:covariance} and translation--dilation invariance imply the same statement at $(\delta,U_{\xi,\eps})$ for compactly supported metric variations.
\end{proof}

\section{Second variation of total \texorpdfstring{$Q_{2N}$}{Q(2N)}-curvature}\label{sec:ambient}

We prove Theorem \ref{thm:ambientHessian} at a closed Einstein metric $(M^n,g_E)$ with $\Ric_{g_E}=(n-1)g_E$.
All tensor operations and integrals in this section use $g_E$.
We first fix an $L^2$-normalized TT eigentensor $h$, so that $-\cA h=bh$ and $\|h\|_{L^2}=1$.
The first ambient variation and the averaged determinant variation are then governed by the scalar ordinary differential equations \eqref{eq:TTambient} and \eqref{eq:ThetaODE}.
We use the scalar harmonic extension to evaluate the ordered holographic formula of Proposition \ref{prop:holographic-Q}.

\subsection{Ambient variations and the scalar extension}
\label{subsec:ambient-linearization}
We prescribe $g_0(s)=g_E\exp(sg_E^{-1}h)$ and denote its formal ambient family by $g_\rho(s)$.
Only the finite Taylor coefficients determined by the normal-form Ricci equations \eqref{eq:FGevolution}--\eqref{eq:FGnormal} are used.
At $s=0$, we have $g_\rho(0)=a^2g_E$, where $a=1+\rho/2$.
A dot denotes a metric variation $\partial_s|_{s=0}$ and a prime denotes $\partial_\rho$.
We define
\begin{equation}\label{eq:ambient-first-variation}
h_\rho\coloneqq a^{-2}\dot g_\rho.
\end{equation}
Proposition \ref{prop:TTambient} gives $h_\rho=\psi_b(\rho)h$.
Accordingly, the second-order expansion is
\[
g_\rho(s)=a^2g_E+s a^2\psi_b(\rho)h
+\frac{s^2}{2}a^2g_EK_h(\rho,x)+O(s^3).
\]
Here $K_h$ is a $g_E$-self-adjoint endomorphism with $K_h(0,x)=(g_E^{-1}h(x))^2$.
For $\rho\ne0$, no proportionality between these endomorphisms is assumed.

\begin{proposition}\label{prop:TTambient}
The normalized first variation is $h_\rho=\psi_bh$, where the required Taylor coefficients of $\psi_b$ satisfy
\begin{equation}\label{eq:TTambient}
2\rho a^2\psi_b''+
\bigl(n\rho a-(n-2)a^2\bigr)\psi_b'-b\psi_b=0,
\qquad \psi_b(0)=1.
\end{equation}
In particular, $\psi_b'(0)=-b/(n-2)$, and $\psi_{b,m}\coloneqq\psi_b^{(m)}(0)$ satisfies \eqref{eq:qrecurrence}.
\end{proposition}
\begin{proof}
Insert $\dot g_\rho=a^2\psi_bh$ into the linearizations of \eqref{eq:FGevolution}--\eqref{eq:FGnormal}.
Since $h$ is TT and $\Delta_Lh=(b-2(n-1))h$, the Ricci variation is $D\Ric_{g_E}[h]=-\Delta_Lh/2$.
The variation of the tangential equation \eqref{eq:FGevolution} is one half of the left-hand side of \eqref{eq:TTambient}, multiplied by $h$.
The variations of the mixed and normal constraints \eqref{eq:FGmixed}--\eqref{eq:FGnormal} vanish because $h$ is divergence-free and trace-free.
Uniqueness of the required normal-form jets \cite[Theorem 2.9]{FG} therefore identifies this scalar solution with the first ambient variation.
In particular,
\begin{equation}\label{eq:raised-index-variation}
(g_\rho^{-1}g_\rho')\big|_{s=0}=a^{-1}\Id,
\qquad
\left.\partial_s(g_\rho^{-1}g_\rho')\right|_{s=0}
=\psi_b'g_E^{-1}h.
\end{equation}
Evaluating \eqref{eq:TTambient} at zero and then differentiating gives
\begin{equation}\label{eq:qrecurrence}
\begin{split}
\psi_{b,0}&=1,\qquad \psi_{b,1}=-\frac b{n-2},\\
\psi_{b,m+1}&=
\frac{2m^2-b}{n-2m-2}\psi_{b,m}
+\frac{m(m-1)(2m+n-2)}{4(n-2m-2)}\psi_{b,m-1}.
\end{split}
\end{equation}
The last line is needed for $1\le m\le N-2$; its denominators are nonzero for $n>2N$.
The finite ambient construction depends smoothly on the initial metric, so these Taylor calculations commute with metric variation.
\end{proof}

\medskip
The second ambient variation is needed only through its averaged trace.
For the $L^2$-normalized TT eigentensor $h$ with $-\cA h=bh$, we define
\begin{equation}\label{eq:ThetaDefinition}
\Theta_b(\rho)\coloneqq
\int_M\left.\partial_s^2
\log\frac{\det g_\rho(s)}{a^{2n}\det g_E}\right|_{s=0}\,dv_{g_E}.
\end{equation}

\begin{lemma}\label{lem:ThetaODE}
One has $\Theta_b=\int_M\tr K_h\,dv_{g_E}-\psi_b^2$ and
\begin{equation}\label{eq:ThetaODE}
(a^2\Theta_b')'=-a^2(\psi_b')^2,
\qquad \Theta_b(0)=0,
\qquad \Theta_b'(0)=\frac b{2(n-1)}.
\end{equation}
Thus $\Theta_b$ depends only on $b$ and $n$ to the required Taylor order.
\end{lemma}
\begin{proof}
The second derivative of $\log\det g_\rho$ is $\tr K_h-\psi_b^2|h|^2$, proving the first assertion and the initial value $\Theta_b(0)=0$.
Write the normal Ricci equation \eqref{eq:FGnormal} as
\[
(\log\det g_\rho)''+
\frac12\tr\bigl((g_\rho^{-1}g_\rho')^2\bigr)=0.
\]
Differentiate twice in $s$, use \eqref{eq:raised-index-variation}, and integrate.
Since $\|h\|_{L^2}=1$, the resulting equation is $\Theta_b''+a^{-1}\Theta_b'+(\psi_b')^2=0$, which is \eqref{eq:ThetaODE}.
At $\rho=0$, $\tr(g^{-1}g'_0)=R_g/(n-1)$.
Along the volume-preserving exponential path, the fixed-volume Hessian of total scalar curvature is $\langle h,(\Delta_L+2(n-1))h\rangle/2=b/2$, giving the stated initial derivative.
This standard Hessian identity also follows directly by differentiating $-\Ric+(R/2)g$ and using $D\Ric_{g_E}[h]=-\Delta_Lh/2$.

For later computations, differentiating \eqref{eq:ThetaODE} gives, with $\theta_m\coloneqq\Theta_b^{(m)}(0)$,
\begin{equation}\label{eq:thetarecurrence}
\begin{split}
\theta_0&=0,\qquad \theta_1=\frac b{2(n-1)},\\
\theta_{m+2}&=-\sum_{j=0}^m\binom mj
\frac{(-1)^j j!}{2^j}\theta_{m-j+1}
-\sum_{j=0}^m\binom mj\psi_{b,j+1}\psi_{b,m-j+1}.
\end{split}
\end{equation}
This uses only Leibniz's rule and $(a^{-1})^{(j)}(0)=(-1)^j j!/2^j$.
\end{proof}

\medskip
We next construct the scalar harmonic extension used to recover $Q_{2N}$.
We set $\nu\coloneqq N-n/2$.
The constant function $1$ on $M$ lifts to the homogeneous function $\mathsf t^{\nu}$ on the ambient hypersurface $\rho=0$.
Its \emph{formal harmonic extension} has the form $\mathsf t^{\nu} f(s,x,\rho)$, where
\[
f(s,x,0)=1,
\qquad
\widetilde\Delta\bigl(\mathsf t^{\nu} f\bigr)=O(\rho^{N-1}).
\]
The harmonicity condition for the ambient extension determines the Taylor coefficients of $f$ through degree $N-1$; the next coefficient of the harmonic equation is the obstruction measured by $Q_{2N}$.
No exact harmonic extension away from $\rho=0$ is required.
We denote by $f_N$ the Taylor polynomial through degree $N-1$ of the scalar factor at the Einstein reference metric.
Lemma \ref{lem:scalar-ambient-obstruction} determines $f_N$.

\begin{lemma}\label{lem:scalar-ambient-obstruction}
The harmonic-extension condition is equivalent, through the required order, to
\begin{equation}\label{eq:ambientlaplacian}
\mathcal R_{N-1}(v^{\frac{1}{2}}f)
=v^{\frac{1}{2}}\Bigl[-2\rho f''+ \left(2N-2-2\rho\frac{v'}v\right)f'+\Delta_{g_\rho}f+\nu\frac{v'}v f\Bigr]
=O(\rho^{N-1}).
\end{equation}
At the Einstein metric $g_E$, $f_N$ is independent of $x$ and satisfies
\begin{equation}\label{eq:Lw}
L_{\nu} f_N=O(\rho^{N-1}),\qquad f_N(0)=1,\qquad
L_{\nu}\coloneqq-2\rho\partial_\rho^2+
\left(2N-2-\frac{n\rho}{a}\right)\partial_\rho+
\frac{\nu n}{2a}.
\end{equation}
This polynomial is
\begin{equation}\label{eq:scalarhypergeom}
f_N(\rho)=\sum_{r=0}^{N-1}
\frac{(\frac{n}{2}-N)_r(\frac{n}{2})_r}{(1-N)_r\,r!}(-\frac{\rho}{2})^r.
\end{equation}
The curvature is recovered from this extension by
\begin{equation}\label{eq:obstructionnormalization}
Q_{2N,g}
=\frac{2^{N-1}((N-1)!)^2}{\nu}
[\rho^{N-1}]\mathcal R_{N-1}(v^{\frac{1}{2}}f).
\end{equation}
\end{lemma}
\begin{proof}
The conjugation identity \cite[(2.4)]{FGJuhl} gives
\[
\mathcal R_{N-1}(v^{\frac{1}{2}}f)
=v^{\frac{1}{2}}\mathsf t^{2-\nu}
\widetilde\Delta(\mathsf t^{\nu} f).
\]
Expanding the already defined operators \eqref{eq:holographic-operators} gives \eqref{eq:ambientlaplacian}.
The coefficient $[\rho^r]f$ enters the coefficient of $\rho^{r-1}$ in the left-hand side of \eqref{eq:ambientlaplacian} with multiplier $2r(N-r)$.
This multiplier is nonzero for $1\le r<N$, proving uniqueness of the required extension coefficients.
A degree-$N$ term in $f$ cannot change the obstruction coefficient, because its multiplier $2N(N-N)$ is zero.

At the Einstein reference, $v=a^n$ and the recursion starting from $1$ gives spatially constant coefficients.
Substitution of $v'/v=n/(2a)$ gives \eqref{eq:Lw}.
Multiplying by $a$ and comparing consecutive coefficients gives
\[
\frac{[\rho^{r+1}]f_N}{[\rho^r]f_N}
=-\frac{(r+\frac{n}{2}-N)(r+\frac{n}{2})}{2(r+1)(r+1-N)}
\qquad(0\le r\le N-2).
\]
Starting from $f_N(0)=1$, this yields \eqref{eq:scalarhypergeom}, using the Pochhammer convention fixed in Section \ref{sec:preliminaries}.
Thus the polynomial is the computed Taylor jet of the geometric extension.

For the normalization, the definitions in Subsection \ref{subsec:ambient-preliminaries} give
\[
\mathcal R_k(\rho u)=\rho\mathcal R_{k-2}u+2ku,
\qquad
\mathcal R_{N-1}v^{\frac{1}{2}}=2\nu\,\partial_\rho v^{\frac{1}{2}}.
\]
The first identity implies that the full ordered product $\mathcal R_{1-N}\cdots\mathcal R_{N-1}$ maps a multiple of $\rho$ to a multiple of $\rho$.
Indeed, commuting $\rho$ to the left cancels the constant terms because the weights $1-N,3-N,\ldots,N-1$ sum to zero.
Since $v^{1/2}(f-1)=O(\rho)$, the restrictions of this product at $\rho=0$ agree for $v^{1/2}f$ and $v^{1/2}$.
By \eqref{eq:ambientlaplacian}, the first applied factor $\mathcal R_{N-1}$ produces a remainder starting at degree $N-1$.
The remaining factors evaluate that leading term with multiplier $(-2)^{N-1}((N-1)!)^2$.
The second identity and the quoted holographic formula \eqref{eq:holographic-Q} now give \eqref{eq:obstructionnormalization}.
\end{proof}

\subsection{The Hessian of \texorpdfstring{$\mathcal Q$}{Q} and comparison at \texorpdfstring{$n=2N$}{n=2N}}
\label{subsec:critical-comparison}
We now differentiate the scalar harmonic extension $\mathsf t^\nu f(s,x,\rho)$ characterized by \eqref{eq:ambientlaplacian}.
The determinant equation \eqref{eq:ThetaODE} supplies the metric-dependent source, and the integration-by-parts identity \eqref{eq:ambient-integration-by-parts} removes the unknown second variation of the extension.
For smooth TT tensors $h,k$, let $\Theta_{h,k}$ denote the integrated mixed logarithmic determinant variation along $g_E\exp(s_1g_E^{-1}h+s_2g_E^{-1}k)$, as in \eqref{eq:ThetaDefinition}.
The resulting coefficient formula is
\begin{equation}\label{eq:WronskianHessian}
D^2\mathcal Q|_{g_E}(h,k)=
\frac{2^{N-1}((N-1)!)^2}{\nu}
[\rho^{N-1}]\,a^n\Theta'_{h,k}
\left(\frac{\nu}{2} f_N^2-\rho f_Nf_N'\right).
\end{equation}
For an eigentensor, this uses $\Theta_b$ through order $N$ and $\psi_b$ through order $N-1$; for $N=1$, only $\Theta_b'(0)$ is needed.

\begin{proof}[Proof of Theorem \ref{thm:ambientHessian}]
First take a TT eigentensor $h$ with $-\cA h=bh$ and $\|h\|_{L^2}=1$.
The first variation of $v$ vanishes pointwise, and a variation of the spatial Laplacian annihilates the spatially constant function $f_N$.
Differentiating \eqref{eq:ambientlaplacian} and using the uniqueness of its Taylor recursion gives $\dot f=0$ through degree $N-1$.
Its obstruction also has zero first variation, proving the first-variation assertion for eigentensors.

We define $u(\rho)\coloneqq\int_M\ddot f\,dv_{g_E}$, so $u(0)=0$.
Differentiate \eqref{eq:ambientlaplacian} twice and integrate.
Since $v'/v=(\log\det g_\rho)'/2$, the integrated second variation of $v'/v$ is $\Theta_b'/2$.
The first variations of $f$ vanish, so the only surviving spatial term is $a^{-2}\Delta_{g_E}\ddot f$, whose integral is zero.
Using the normalization \eqref{eq:obstructionnormalization}, we obtain
\[
L_{\nu}u+\Theta_b'\left(\frac{\nu}{2} f_N-\rho f_N'\right)
=\frac{\nu\,D^2\mathcal Q|_{g_E}(h,h)}
{2^{N-1}((N-1)!)^2}\rho^{N-1}+O(\rho^N).
\]
Here multiplication by $v^{1/2}=1+O(\rho)$ does not change the obstruction coefficient, and its second variation vanishes at $\rho=0$.

The operator in \eqref{eq:Lw} satisfies the Lagrange identity
\begin{equation}\label{eq:ambient-integration-by-parts}
\rho^{-N}a^n(f_NL_{\nu}u-uL_{\nu}f_N)
=\partial_\rho\!\left[-2\rho^{1-N}a^n
(f_Nu'-f_N'u)\right].
\end{equation}
The expression $f_Nu'-f_N'u$ is the ordinary Wronskian.
The derivative of a Laurent series has zero coefficient of $\rho^{-1}$.
Since $u(0)=0$ and $L_{\nu}f_N=O(\rho^{N-1})$, this gives
\[
[\rho^{N-1}](a^nf_NL_{\nu}u)
=[\rho^{N-1}](a^nuL_{\nu}f_N)=0.
\]
Multiplying the scalar variation equation by $a^nf_N=1+O(\rho)$ and taking its coefficient of degree $N-1$ proves \eqref{eq:WronskianHessian} for $h=k$.

The recurrences \eqref{eq:qrecurrence} and \eqref{eq:thetarecurrence} make $\psi_{b,m}$ and $\theta_m$ polynomials in the real parameter $b$ of degree at most $m$, with coefficients rational in $n$.
Thus \eqref{eq:WronskianHessian} defines a polynomial $\mathscr P_N^{(n)}(b)$ of degree at most $N$.
Only $\theta_N$ can contribute $b^N$, and its multiplier in \eqref{eq:WronskianHessian} is $2^{N-2}(N-1)!$.
For $N\ge2$, the resulting coefficient of $b^N$ is
\begin{equation}\label{eq:TT-leading-coefficient}
\frac{(-1)^{N+1}(N-1)!}{4}
\sum_{r=1}^{N-1}\binom{N-2}{r-1}
\frac{1}
{(\frac{n}{2}-r)_r\,(\frac{n}{2}-N+r)_{N-r}}.
\end{equation}
Indeed, the leading coefficient of $\psi_{b,r}$ is $(-1)^r/[2^r(n/2-r)_r]$, and the quadratic sum in \eqref{eq:thetarecurrence} supplies the leading coefficient of $\theta_N$.
All displayed Pochhammer factors are positive for $n>2N$, so the degree is exactly $N$.
For $N=1$, \eqref{eq:WronskianHessian} directly gives $\mathscr P_1^{(n)}(b)=b/[4(n-1)]$.

For TT eigentensors $h,k$ with eigenvalues $b,c$, take mixed derivatives of the ambient and scalar harmonic-extension equations in the metric directions $h,k$.
The determinant equation becomes
\[
(a^2\Theta_{h,k}')'=-a^2\psi_b'\psi_c'\langle h,k\rangle,
\qquad \Theta_{h,k}(0)=0,
\qquad \Theta_{h,k}'(0)=\frac{c\langle h,k\rangle}{2(n-1)}.
\]
If $b\ne c$, self-adjointness gives $\langle h,k\rangle=0$, hence $\Theta_{h,k}=0$.
If $b=c$, uniqueness gives $\Theta_{h,k}=\Theta_b\langle h,k\rangle$.
Applying the same coefficient extraction to the mixed second derivative of $f$ proves \eqref{eq:WronskianHessian} and \eqref{eq:TTpolynomial} on pairs of eigentensors and then on finite spectral sums.

On a closed Einstein manifold, $\Delta_L$ is self-adjoint and elliptic and preserves the TT subspace.
The spectral truncations of smooth TT tensors converge in $C^\infty$.
The required ambient jets and curvature variations depend smoothly on finitely many metric derivatives, so the identities pass to arbitrary smooth TT tensors.
In particular, the bilinear form of \eqref{eq:ThetaODE} is
\[
(a^2\Theta_{h,k}')'=-a^2\langle h_\rho',k_\rho'\rangle,
\qquad
\Theta_{h,k}'(0)=-\frac{\langle h,\cA k\rangle}{2(n-1)},
\]
where $h_\rho,k_\rho$ are their normalized first ambient variations.
Finally, a degree-$N$ polynomial in $\cA$ defines, after integration by parts, a bilinear form bounded by $C\|h\|_{H^N}\|k\|_{H^N}$.
Density of finite sums of TT eigentensors in the weak TT subspace of $H^N$ gives the asserted continuous extension.
\end{proof}

Substitution of the finite recurrences gives the first four polynomials, each in its range $n>2N$:

\begin{equation}\label{eq:lowP}
\begin{aligned}
\mathscr P_1^{(n)}(b)&=\frac b{4(n-1)},\\
\mathscr P_2^{(n)}(b)&=-\frac{b^2}{(n-2)^2}
+\frac{n^2-4}{8(n-1)}b,\\
\mathscr P_3^{(n)}(b)&=\frac{8b^2(b-2)}{(n-2)^2(n-4)}
-\frac{n+2}{n-2}b^2
+\frac{3(n^2-4)(n^2-16)}{64(n-1)}b,\\
\mathscr P_4^{(n)}(b)&=-\frac{96(n-5)b^2(b-2)^2}{(n-2)^2(n-4)^2(n-6)}
+\frac{12(n^2-2n-12)b^2(b-2)}{(n-2)^2(n-6)}\\
&\quad-\frac{(n+2)(5n^4-40n^3-36n^2+768n-960)}{8(n-2)^2(n-6)}b^2\\
&\quad+\frac{(n^2-4)(n^2-16)(n^2-36)}{64(n-1)}b.
\end{aligned}
\end{equation}

\medskip
When $n=2N\ge4$, $Q_{2N}$ is defined by the ambient construction at this order, rather than by dividing $P_{2N,g}(1)$ by $(n-2N)/2$.
\begin{proposition}\label{prop:Matsumoto-comparison}
If $n=2N\ge4$, then for smooth TT tensors $h,k$ at a closed Einstein metric $g_E$ satisfying $\Ric_{g_E}=(n-1)g_E$,
\begin{equation}\label{eq:Matsumoto-critical}
D^2\mathcal Q|_{g_E}(h,k)
=-\frac14\left\langle h,
\prod_{m=0}^{N-1}\bigl(\cA+2m(2N-2m-1)\bigr)k
\right\rangle_{L^2(g_E)}.
\end{equation}
Equivalently, the Hessian operator is obtained by evaluating the following polynomial at $b=-\cA$:
\begin{equation}\label{eq:critical-polynomial}
\frac{(-1)^{N+1}}4
\prod_{j=0}^{N-1}\bigl[b-j(2N-j-1)\bigr].
\end{equation}
\end{proposition}
\begin{proof}
Formula \eqref{eq:Matsumoto-critical} is \cite[Theorem 0.2]{Matsumoto}, with Schouten constant $1/2$ and Lichnerowicz Laplacian $-\Delta_L$.
The reindexing $j=\min\{2m,2N-1-2m\}$ permutes $\{0,\ldots,N-1\}$ and identifies $2m(2N-2m-1)=j(2N-j-1)$.
Substituting $b=-\cA$ gives \eqref{eq:critical-polynomial}.
\end{proof}
This comparison fixes the normalization at $n=2N$.
The Hessian for $n>2N$ in Theorem \ref{thm:ambientHessian} is computed independently by the finite coefficient extraction \eqref{eq:WronskianHessian}.
Appendix \ref{app:obstruction-comparison} derives this product from the residue of the extended obstruction tensor and compares the linearizations for $n>2N$.
The determinant variation also determines the Hessians of total renormalized volume coefficients in TT directions; Appendix \ref{app:volume-hessians} extracts them directly from $a^n\Theta_{h,k}$ and relates them to \eqref{eq:WronskianHessian}.

\section{Lyapunov--Schmidt reduction and quadratic energy}\label{sec:reduction}
We construct the projected inverse and solve the nonlinear equation on the orthogonal complement of the translation and dilation fields.
We then compute the quadratic reduced energy for polynomial perturbations.
Critical points of the reduced energy solve the full equation.
In Sections \ref{sec:reduction}--\ref{sec:applications}, the reference metric on the sphere is $g_c$; an omitted base metric in $D^2\mathcal Q$ means $g_c$.

\subsection{The projected equation and estimates}\label{sec:analyticinverse}
\label{subsec:metric-perturbations}\label{subsec:projected-equation}
\begin{proposition}\label{prop:projected-inverse}
The kernel of $L_{U_{\xi,\eps}}$ on $\mathcal H$ is exactly $K_{\xi,\eps}$.
Its restriction from $K_{\xi,\eps}^\perp$ to $(K_{\xi,\eps}^\perp)^*$ is an isomorphism.
In particular,
\begin{equation}\label{eq:inverseestimate}
\|\phi\|_{\mathcal H}\le C\|L_{U_{\xi,\eps}}\phi\|_{(K_{\xi,\eps}^\perp)^*},
\qquad \phi\in K_{\xi,\eps}^\perp,
\end{equation}
uniformly on compact bubble-parameter sets.
After the identifications of Lemma \ref{lem:Jacobi-identifications}, the inverse depends smoothly on $(\xi,\eps)$.
\end{proposition}
\begin{proof}
Gover's factorization \cite[Theorem 1.2]{Gover}, with our Laplacian sign convention, gives
\begin{equation}\label{eq:roundfactorization}
P_{2N,g_c}=\prod_{r=1}^{N}
\left[-\Delta_{g_c}+(\frac{n}{2}-r)(\frac{n}{2}+r-1)\right].
\end{equation}
Since the eigenvalue of $-\Delta_{g_c}$ on spherical harmonics of degree $\ell$ is $\ell(\ell+n-1)$, the corresponding eigenvalue of $P_{2N,g_c}$ is
\[
\Lambda_\ell\coloneqq\prod_{r=1}^N
(\ell+\frac{n}{2}-r)(\ell+\frac{n}{2}+r-1)
=\frac{\Gamma(\ell+\frac{n}{2}+N)}{\Gamma(\ell+\frac{n}{2}-N)}.
\]
By conformal covariance, $u\mapsto u/U$ identifies $\mathcal H$ isometrically with the scalar space $H^N(\Sph^n)$ equipped with the energy norm $\bigl(\int_{\Sph^n}vP_{2N,g_c}v\,dv_{g_c}\bigr)^{1/2}$.
By \eqref{eq:roundfactorization}, this norm is equivalent to the derivative norm \eqref{eq:Sobolev-norms}.
Relative to the energy inner product, the linearized operator has eigenvalues
\[
1-\frac{(n+2N)c_{n,N}}{(n-2N)\Lambda_\ell}.
\]
Since $\Lambda_1=(n+2N)c_{n,N}/(n-2N)$, the degree-one harmonics form the kernel.
Multiplication by $U$ identifies them with the $n$ translation fields and the dilation field.
The degree-zero eigenvalue is $-4N/(n-2N)<0$, while
\[
1-\frac{(n+2N)c_{n,N}}{(n-2N)\Lambda_\ell}
\ge1-\frac{\Lambda_1}{\Lambda_2}
=\frac{4N}{n+2N+2}>0\qquad(\ell\ge2).
\]
The spectral resolution therefore proves bounded invertibility on the orthogonal complement, including the negative constant mode.
The change of variables $x=\xi+\eps y$, with $u(x)=\eps^{-(n-2N)/2}v(y)$, preserves the $\mathcal H$ norm and proves the assertion at every $(\xi,\eps)\in\mathbb R^n\times(0,\infty)$.

The formula \eqref{eq:roundbubble} shows that $U_{\xi,\eps}^{4N/(n-2N)}$ and its parameter derivatives depend smoothly on $(\xi,\eps)$ in $L^{n/(2N)}(\R^n)$.
H\"older's inequality and Lemma \ref{lem:Jacobi-identifications} imply operator-norm smoothness after identification with $K_{\xi_0,\eps_0}^\perp$.
Differentiating the identity between the operator and its inverse on this fixed space proves smoothness and uniform bounds for the inverse and its parameter derivatives on compact sets.
\end{proof}
A dual functional is restricted to $K_{\xi,\eps}^\perp$ whenever it is an argument of $L_{U_{\xi,\eps}}^{-1}$.

We next estimate the metric perturbation and solve the projected equation.
For a smooth compactly supported trace-free tensor $h$, we set
\begin{equation}\label{eq:metricnorm}
\kappa(h)\coloneqq\|h\|_{L^\infty(\R^n)}+
\sum_{r=1}^{2N}\|\partial^rh\|_{L^{\frac{n}{r}}(\R^n)}.
\end{equation}
\begin{lemma}\label{lem:metric-perturbation}
The norm $\kappa$ is invariant under translations and coordinate dilations.
For every smooth compactly supported trace-free tensor $h$ with $\kappa(h)$ sufficiently small,
\begin{equation}\label{eq:operatorperturbation}
\|P_{2N,\exp(h)}-(-\Delta)^N\|_{\mathcal H\to\mathcal H^*}
\le C\kappa(h).
\end{equation}
The corresponding multilinear estimates hold for the metric derivatives of this operator.
\end{lemma}
\begin{proof}
The invariance follows by a change of variables.
We write the formally self-adjoint GJMS operator in bilinear form.
Each term is a coefficient multiplying derivatives of orders $a,b\le N$ of $u,v$, respectively, with total order $r=2N-a-b$ of the metric derivatives.
This follows from the local natural construction and homogeneity under coordinate dilation.
The coefficient is a finite sum of products of derivatives of $h$ of total order $r$, times smooth bounded functions of $h$.
H\"older's inequality bounds the coefficient in $L^{n/r}(\R^n)$; the case $r=0$ is estimated in $L^\infty(\R^n)$.
The Sobolev inequalities
\[
\|\partial^a u\|_{L^{\frac{2n}{n-2N+2a}}(\R^n)}\le C\|u\|_{\mathcal H}
\qquad(0\le a\le N)
\]
prove \eqref{eq:operatorperturbation}.
Integration by parts transfers any excess derivatives on one test factor to the coefficient or the other factor.
Products containing additional metric factors satisfy the corresponding multilinear estimates.
\end{proof}

The reduction follows \cite[Section 2]{BrendleBlowup}, \cite[Section 2]{BrendleMarques}, \cite[Sections 5--8]{WeiZhao}, and \cite[Subsections 11.2--11.4]{GongKimWei}.
We fix a sufficiently small open parameter set $\mathcal U\Subset\R^n\times(0,\infty)$, so that the identifications in \eqref{eq:Jacobi-projection} are defined on a neighborhood of $\overline{\mathcal U}$.
In this subsection, we take parameter norms over $\overline{\mathcal U}$.
The constants in this subsection may depend on $n,N,\mathcal U$, but not on the metric perturbation $h$.
We set $\gamma\coloneqq\min\{1,4N/(n-2N)\}$.

For the nonlinear Taylor remainders, we write $p\coloneqq(n+2N)/(n-2N)$ and $q\coloneqq2n/(n-2N)=p+1$.

\begin{lemma}\label{lem:nonlinear-remainders}
Set
\[
\begin{aligned}
\mathcal N_p(U,w)&\coloneqq(U+w)_+^p-U^p-pU^{p-1}w,\\
\mathcal N_q(U,w)&\coloneqq(U+w)_+^q-U^q-qU^{q-1}w
-\frac12q(q-1)U^{q-2}w^2.
\end{aligned}
\]
The map $u\mapsto(u_+)^p$ is $C^1$ from $\mathcal H$ to $\mathcal H^*$.
Uniformly on $\overline{\mathcal U}$, for $\|w\|_{\mathcal H}\le1$,
\begin{equation}\label{eq:nonlinear-remainder-bound}
\|\mathcal N_p(U_{\xi,\eps},w)\|_{\mathcal H^*}\le C\|w\|_{\mathcal H}^{1+\gamma},
\qquad
\|D_w\mathcal N_p(U_{\xi,\eps},w)\|_{\mathcal H\to\mathcal H^*}
\le C\|w\|_{\mathcal H}^{\gamma}.
\end{equation}
For $U>0$,
\begin{equation}\label{eq:nonlinear-U-derivative}
|U\partial_U\mathcal N_p(U,w)|\le
\begin{cases}
C|w|^p &\text{for }1<p<2,\\
C\bigl(U^{p-2}|w|^2+|w|^p\bigr) &\text{for }p\ge2.
\end{cases}
\end{equation}
If $\|w_{\xi,\eps}\|_{C^1(\overline{\mathcal U};\mathcal H)}\le a\le1$, then
\begin{equation}\label{eq:scalar-energy-remainder}
\left\|\int_{\R^n}\mathcal N_q(U_{\xi,\eps},w_{\xi,\eps})\,dx\right\|_{C^1(\overline{\mathcal U})}
\le Ca^{2+\gamma}.
\end{equation}
\end{lemma}
\begin{proof}
Taylor's formula on $|w|\le U/2$ and the power bounds on its complement give the pointwise remainders.
H\"older's inequality, Sobolev embedding, and \eqref{eq:dual-bound} yield \eqref{eq:nonlinear-remainder-bound}; the same split proves \eqref{eq:nonlinear-U-derivative}.
For $2<q<3$ the scalar remainder satisfies
\[
|\mathcal N_q|+|U\partial_U\mathcal N_q|\le C|w|^q,\qquad
|\partial_w\mathcal N_q|\le C|w|^{q-1}.
\]
For $q\ge3$, the respective bounds are $C(U^{q-3}|w|^3+|w|^q)$ and $C(U^{q-3}|w|^2+|w|^{q-1})$.
Since $|\partial_{\xi,\eps}U_{\xi,\eps}|\le CU_{\xi,\eps}$, H\"older's inequality proves \eqref{eq:scalar-energy-remainder}, including its parameter derivative.
\end{proof}

\begin{proposition}\label{prop:projected-family}
There are constants $\kappa_0,C>0$ such that, for every smooth compactly supported trace-free $h$ with $\kappa(h)\le\kappa_0$, there is a locally unique $C^1$ family $w_{h,\xi,\eps}\in K_{\xi,\eps}^\perp$ satisfying
\begin{equation}\label{eq:projectedequation}
\langle\mathscr E_h(U_{\xi,\eps}+w_{h,\xi,\eps}),\phi\rangle=0
\qquad\text{for every }\phi\in K_{\xi,\eps}^\perp.
\end{equation}
Moreover,
\begin{equation}\label{eq:projected-family-bound}
\|w_{h,\xi,\eps}\|_{C^1(\overline{\mathcal U};\mathcal H)}
\le C\|\mathscr E_h(U_{\xi,\eps})\|_{C^1(\overline{\mathcal U};\mathcal H^*)}
\le C\kappa(h).
\end{equation}
The projected linearization at $U_{\xi,\eps}+w_{h,\xi,\eps}$ has a uniformly bounded inverse.
The reduced energy is the $C^1$ function $(\xi,\eps)\mapsto J_h(U_{\xi,\eps}+w_{h,\xi,\eps})$.
A parameter $(\xi,\eps)$ is a critical point of this function if and only if $U_{\xi,\eps}+w_{h,\xi,\eps}$ solves the full weak equation $\mathscr E_h(U_{\xi,\eps}+w_{h,\xi,\eps})=0$.
\end{proposition}
\begin{proof}
In this proof only, write $z=(\xi,\eps)$ and abbreviate the corresponding parameter subscripts and arguments by $z$; write $z_0=(\xi_0,\eps_0)$ when a reference parameter is needed.
Using \eqref{eq:weak-residual}, the projected equation \eqref{eq:projectedequation} is equivalent to
\[
w=-L_{U_z}^{-1}\bigl\{\mathscr E_h(U_z)
+(P_{2N,\exp(h)}-(-\Delta)^N)w-c_{n,N}\mathcal N_p(U_z,w)\bigr\}.
\]
By Lemma \ref{lem:nonlinear-remainders}, its Lipschitz constant on a ball of radius $a$ is bounded by $C(\kappa(h)+a^\gamma)$.
Proposition \ref{prop:projected-inverse} and \eqref{eq:operatorperturbation} therefore establish a contraction on a ball of radius $C\sup_z\|\mathscr E_h(U_z)\|_{\mathcal H^*}$ for small $\kappa_0$.
The bound $C(\kappa(h)+a^\gamma)<1$ also proves local uniqueness and invertibility of the projected linearization.
The implicit-function theorem, after the identifications of Lemma \ref{lem:Jacobi-identifications}, proves $C^1$ dependence on $z$.

For the quantitative derivative estimate, we use Lemma \ref{lem:nonlinear-remainders} and $|\partial_zU_z|\le CU_z$.
We use \eqref{eq:Jacobi-projection} to express the projected equation on $K_{z_0}^\perp$, and differentiate in $z$.
The coefficient of the differentiated unknown is the uniformly invertible projected linearization.
Derivatives of the projections and of $L_{U_z}$ contribute $O(\|w\|_{\mathcal H})$; \eqref{eq:nonlinear-U-derivative} contributes $O(\|w\|_{\mathcal H}^{1+\gamma})$.
Together with the derivative of $\mathscr E_h(U_z)$, these prove \eqref{eq:projected-family-bound}.
The argument uses only first derivatives of the power map on $\mathcal H$, including when $n>6N$.

Since $\mathscr E_h(U_z+w_{h,z})$ annihilates $K_z^\perp$, it has the form
\[
\langle\mathscr E_h(U_z+w_{h,z}),\phi\rangle
=\sum_{\alpha=1}^{n+1}b_\alpha\langle Z_{z,\alpha},\phi\rangle_{\mathcal H}
\qquad(\phi\in\mathcal H).
\]
Differentiating $\langle Z_{z,\alpha},w_{h,z}\rangle_{\mathcal H}=0$, we obtain
\begin{equation}\label{eq:multiplier-matrix}
\partial_{z_\beta}J_h(U_z+w_{h,z})
=\sum_{\alpha=1}^{n+1}b_\alpha\left[(G_z)_{\alpha\beta}
-\langle\partial_{z_\beta}Z_{z,\alpha},w_{h,z}\rangle_{\mathcal H}\right].
\end{equation}
By \eqref{eq:projected-family-bound}, the matrix in \eqref{eq:multiplier-matrix} is $G_z+O(\kappa(h))$.
It is invertible for small $\kappa_0$, since $G_z$ is uniformly positive definite on $\overline{\mathcal U}$.
Thus a critical point has $b_\alpha=0$ for every $\alpha\in\{1,\ldots,n+1\}$.
The converse follows directly by differentiation of the energy.
\end{proof}

\subsection{The quadratic reduced energy}
\label{sec:polynomial-reduction}\label{sec:energybridge}
\label{subsec:local-energy}\label{subsec:TT-energy}
We compute the quadratic coefficient of the reduced energy from Proposition \ref{prop:projected-family} for a general trace-free polynomial perturbation.
The linear scalar correction is retained before restricting to the Weyl tensors \eqref{eq:Weylchain}, as in the energy expansions of \cite[Section 3]{BrendleBlowup}, \cite[Section 3]{BrendleMarques}, \cite[Section 8]{WeiZhao}, and \cite[Subsection 11.4]{GongKimWei}.

For a symmetric polynomial two-tensor $H$ on $\R^n$, we write $\bar H\coloneqq\Omega^2H$ in the stereographic chart \eqref{eq:roundbubble}.
We impose the following integrability condition:
\begin{equation}\label{eq:admissibility}
\deg H\le d,\qquad n>2N+2d.
\end{equation}
The tensor $\bar H$ need not be smooth at the north pole.
Lemma \ref{lem:cutoff-convergence} proves that the smooth tensors $\Omega^2\chi_RH$ converge to $\bar H$ in $H^N(\Sph^n;\operatorname{Sym}^2T^*\Sph^n)$.

\begin{lemma}\label{lem:cutoff-convergence}
Under \eqref{eq:admissibility}, $\bar H\in H^N(\Sph^n;\operatorname{Sym}^2T^*\Sph^n)$ and
\begin{equation}\label{eq:formtail}
\|\bar H-\Omega^2\chi_RH\|_{H^N(g_c)}\le CR^{d-\frac{n-2N}{2}},\qquad R\ge1.
\end{equation}
The same estimate holds after any fixed number of derivatives of $H(\xi+\eps\,\cdot)$, uniformly on compact parameter sets.
\end{lemma}
\begin{proof}
The inverted coordinates $y=x/|x|^2$ extend across the north pole, where $y=0$.
In these coordinates, $|\nabla^r\bar H|_{g_c}=O(|y|^{-d-r})$.
The squared $L^2$ norm on $|y|\le C/R$ is bounded by
\[
C\int_0^{\frac{C}{R}}t^{n-1-2d-2r}\,dt
\le C R^{2d+2r-n}\qquad(0\le r\le N).
\]
The terms in $\nabla^r[(1-\chi_R)\bar H]$ with derivatives on $\chi_R$ satisfy the same $L^2$ bound.
Taking the square root and using $r\le N$ proves \eqref{eq:formtail}.
Parameter differentiation does not increase the polynomial degree.
\end{proof}

\begin{lemma}\label{lem:quadratic-continuity}
At $(g_c,1)$, the second metric derivative and the mixed metric--scalar derivative of $\mathcal J$ are continuous on the corresponding $H^N$ spaces.
The quadratic form obtained by eliminating the scalar variable on $K^\perp$ has the same continuity.
In particular, Lemma \ref{lem:cutoff-convergence} defines these pairings on polynomial tensors satisfying \eqref{eq:admissibility}.
\end{lemma}
\begin{proof}
Both second variations have total differential order at most $2N$.
Integration by parts places at most $N$ derivatives on each factor, giving continuity on the stated $H^N$ spaces.
The bounded projected inverse from Proposition \ref{prop:projected-inverse} preserves this continuity when the scalar variable is eliminated.
\end{proof}

For a trace-free metric direction $H$, we define
\begin{equation}\label{eq:linear-metric-source}
f_{H,\xi,\eps}\coloneqq D_gP_{2N,\delta}[H]U_{\xi,\eps},
\qquad f_H\coloneqq f_{H,0,1}.
\end{equation}
The derivative in \eqref{eq:linear-metric-source} is taken along $\exp(\eta H)$ at $\eta=0$.
By \eqref{eq:weak-residual} and the bubble equation \eqref{eq:bubble-equation}, $\eta f_{H,\xi,\eps}$ is the first-order term in $\mathscr E_{\eta H}(U_{\xi,\eps})$.

\begin{lemma}\label{lem:source-limits}
Let $H$ be trace-free, with $\deg H\le d$ and $n>2N+2d$, and put $H_R\coloneqq\chi_RH$.
Then $f_{H,\xi,\eps}\in\mathcal H^*$,
\begin{equation}\label{eq:source-tail}
\|f_{H_R,\xi,\eps}-f_{H,\xi,\eps}\|_{C^2(\overline{\mathcal U};\mathcal H^*)}\le CR^{d-\frac{n-2N}{2}},
\qquad
\langle f_{H,\xi,\eps},Z\rangle=0\quad(Z\in K_{\xi,\eps}).
\end{equation}
The estimate in \eqref{eq:source-tail} holds on any fixed compact parameter set $\overline{\mathcal U}\subset\R^n\times(0,\infty)$, with $C$ allowed to depend on that set.
\end{lemma}
\begin{proof}
Each term in the first metric variation has total differential order $2N$.
A term with $r$ derivatives on $H_R$ and $2N-r$ derivatives on $U_{\xi,\eps}$ is bounded by $(1+|x|)^{d-r}(1+|x|)^{-n+r}$.
Consequently
\begin{equation}\label{eq:firstforcingdecay}
\sum_{|\beta|\le2}|\partial_{\xi,\eps}^\beta f_{H_R,\xi,\eps}(x)|
\le C(1+|x|)^{d-n}.
\end{equation}
The product-rule terms in $f_{H_R,\xi,\eps}$ with derivatives on $\chi_R$ satisfy \eqref{eq:firstforcingdecay} on $R\le|x|\le2R$.
Since
\[
\|(1+|x|)^{d-n}\|_{L^{\frac{2n}{n+2N}}(\R^n\setminus B_R(0))}
\le C R^{d-\frac{n-2N}{2}},
\]
\eqref{eq:firstforcingdecay} and \eqref{eq:dual-bound} imply $f_{H,\xi,\eps}\in\mathcal H^*$ and the norm estimate in \eqref{eq:source-tail}.

For compactly supported $H_R$, Proposition \ref{prop:jointcritical} implies $\left.\partial_\eta J_{\eta H_R}(U_{\xi,\eps})\right|_{\eta=0}=0$.
Differentiating with respect to each bubble parameter and using self-adjointness yields $\langle f_{H_R,\xi,\eps},Z_{\xi,\eps,\alpha}\rangle=0$.
The limit in $\mathcal H^*$ proves the identity for $H$.
\end{proof}

We now eliminate the scalar correction to define the quadratic reduced energy.
For compactly supported trace-free tensors $H,\widetilde H$, we define
\begin{equation}\label{eq:correctedE}
\cE(H,\widetilde H)\coloneqq\frac14\int_{\R^n} U\,D_g^2P_{2N,\delta}[H,\widetilde H]U\,dx
-\frac12\langle f_H,L_U^{-1}f_{\widetilde H}\rangle.
\end{equation}
The mixed derivative is taken in logarithmic coordinates $g\coloneqq\exp(\eta_1H+\eta_2\widetilde H)$.
For general metric variations, $\cE$ is defined by stationary elimination of the scalar variable in the quadratic Taylor expansion of \eqref{eq:geometricJ} at $(g_c,1)$, retaining all volume-measure derivatives, and then by the conformal invariance in Lemma \ref{lem:energy-invariance}.
The scalar correction is determined on $K^\perp$ by Proposition \ref{prop:projected-inverse}.
This extension is needed in Lemma \ref{lem:gauge}, since $\mathcal L_{\chi_RX}\delta$ need not be trace-free.
Lemma \ref{lem:quadratic-continuity} extends the pairings to polynomial tensors satisfying \eqref{eq:admissibility}.

For a trace-free polynomial $h$ satisfying \eqref{eq:admissibility}, we define
\begin{equation}\label{eq:modelF}
F_h(\xi,\eps)\coloneqq\frac{8}{n-2N}\,
\cE\bigl(h(\xi+\eps\,\cdot),h(\xi+\eps\,\cdot)\bigr).
\end{equation}

\begin{lemma}\label{lem:quadratic-coefficient}
For a compactly supported trace-free tensor $H_R$, stationary elimination of $w\in K^\perp$ in the quadratic Taylor expansion of $J_{\eta H_R}(U+w)$ gives the coefficient $\cE(H_R,H_R)$ of $\eta^2$.
For polynomial $h$ satisfying \eqref{eq:admissibility}, the corresponding coefficient at $U_{\xi,\eps}$, obtained by replacing $h$ with $\chi_Rh$ and taking $R\to\infty$ as in Lemmas \ref{lem:cutoff-convergence} and \ref{lem:source-limits}, is $(n-2N)F_h(\xi,\eps)/8$, and $F_h$ is smooth.
\end{lemma}
\begin{proof}
Taylor expansion at $(\eta,w)=(0,0)$ yields
\begin{equation}\label{eq:metric-scalar-expansion}
\begin{split}
J_{\eta H_R}(U+w)={}&J_0(U)
+\frac{\eta^2}{4}\int_{\R^n} U\,D_g^2P_{2N,\delta}[H_R,H_R]U\,dx
+\eta\langle f_{H_R},w\rangle\\
&+\frac12\langle L_Uw,w\rangle
+o\bigl((|\eta|+\|w\|_{\mathcal H})^2\bigr).
\end{split}
\end{equation}
The linear terms vanish by Proposition \ref{prop:jointcritical}.
The stationary scalar correction on $K^\perp$ is
\begin{equation}\label{eq:linear-scalar-correction}
w=-\eta L_U^{-1}f_{H_R}.
\end{equation}
Substitution and polarization give \eqref{eq:correctedE}; only invertibility of $L_U$ on $K^\perp$ is needed.
Translation and dilation invariance, followed by the cutoff limits in Lemmas \ref{lem:cutoff-convergence} and \ref{lem:source-limits}, identify the coefficient at $U_{\xi,\eps}$ with $(n-2N)F_h(\xi,\eps)/8$.
Its smoothness follows from polynomial dependence of $h(\xi+\eps\,\cdot)$ and bilinearity of $\cE$.
\end{proof}

\begin{lemma}\label{lem:gauge}
For a compactly supported smooth vector field $X$ and a metric direction $H$ with $\Omega^2H\in H^N(\Sph^n;\operatorname{Sym}^2T^*\Sph^n)$,
\begin{equation}\label{eq:gaugenull}
\cE(\mathcal L_X\delta,H)=0.
\end{equation}
Identity \eqref{eq:gaugenull} also holds when $X$ is polynomial of degree at most $d+1$ and $n>2N+2d$.
In particular, constant and affine trace-free symmetric tensors are null directions of $\cE$.
\end{lemma}
\begin{proof}
For a compactly supported vector field, Lemma \ref{lem:energy-invariance} and Proposition \ref{prop:jointcritical} show that the Hessian of $\mathcal J$ at $(\delta,U)$ annihilates the direction $(\mathcal L_X\delta,XU)$.
Eliminating the scalar variable proves \eqref{eq:gaugenull}.
A component of $XU$ in $K$ does not change this conclusion: $L_U$ annihilates it, and the mixed metric--scalar variation annihilates $K$ by differentiating the stationarity in Proposition \ref{prop:jointcritical}.

For a polynomial vector field $X$ of degree at most $d+1$, we first replace $X$ by $\chi_RX$.
Terms in $\mathcal L_{\chi_RX}\delta$ containing derivatives of $\chi_R$ have the same polynomial growth bounds as $\mathcal L_X\delta$.
Hence \eqref{eq:formtail} implies $\Omega^2\mathcal L_{\chi_RX}\delta\to\Omega^2\mathcal L_X\delta$ in $H^N(\Sph^n;\operatorname{Sym}^2T^*\Sph^n)$ when $n>2N+2d$.
Lemma \ref{lem:quadratic-continuity} then proves the identity for $X$.

Let $C_{ij}$ and $A_{ijk}$ be arbitrary real constants with $C_{ij}=C_{ji}$ and $A_{ijk}=A_{jik}$; no other symmetry is assumed.
For the symmetric affine tensor $H_{ij}(x)\coloneqq C_{ij}+A_{ijk}x^k$, we define the polynomial vector field
\[
X_i\coloneqq\frac12C_{ij}x^j+
\frac14(A_{ijk}+A_{ikj}-A_{jki})x^jx^k.
\]
Direct differentiation verifies $\partial_iX_j+\partial_jX_i=H_{ij}$.
This proves the assertion for constant and affine tensors.
\end{proof}

\medskip
We next compare this quadratic form with the reduced energy of the projected solution.
We fix a trace-free polynomial tensor $h$ of degree at most $d$, with $n>2N+2d$, and use the parameter set $\mathcal U$ from Subsection \ref{subsec:projected-equation}.
In the following comparison of reduced energies, constants may depend on $N,n,h,\chi,\mathcal U$, but not on $\mu,\lambda,\rho$.

For $\mu>0$ and $0<\lambda<\rho$, we set
\begin{equation}\label{eq:onebump}
k_{\mu,\lambda,\rho}(x)\coloneqq\mu\lambda^d
\chi(\frac{|x|}{\rho})h(\frac{x}{\lambda}).
\end{equation}
We set
\begin{equation}\label{eq:localscales}
\eta\coloneqq\mu\lambda^d,\qquad R\coloneqq\frac{\rho}{\lambda},\qquad h_R\coloneqq\chi_Rh,
\end{equation}
so that the rescaled metric is $\exp(\eta h_R)$.

\begin{lemma}\label{lem:local-metric-estimates}
For each fixed $r\in\mathbb N\cup\{0\}$,
\begin{equation}\label{eq:jetbound}
\|\partial^rk_{\mu,\lambda,\rho}\|_{L^\infty(\R^n)}
\le C_r\mu\rho^{d-r},\qquad
\kappa(k_{\mu,\lambda,\rho})\le C\mu\rho^d.
\end{equation}
For $\eta R^d$ sufficiently small,
\begin{equation}\label{eq:local-residual-estimate}
\|\mathscr E_{\eta h_R}(U_{\xi,\eps})-\eta f_{h,\xi,\eps}\|_{C^1(\overline{\mathcal U};\mathcal H^*)}
\le C\eta\bigl(R^{d-\frac{n-2N}{2}}+\eta R^d\bigr).
\end{equation}
The constants are independent of $\mu,\lambda,\rho$.
\end{lemma}
\begin{proof}
Every homogeneous component of $h$ has degree at most $d$.
Differentiating the cutoff product proves \eqref{eq:jetbound}; scale invariance also gives $\kappa(\eta h_R)\le C\eta R^d=C\mu\rho^d$.
Estimate \eqref{eq:source-tail} bounds $\eta(f_{h_R,\xi,\eps}-f_{h,\xi,\eps})$ in $C^1(\overline{\mathcal U};\mathcal H^*)$.
For the remaining terms, the derivative bounds $|\partial^rh_R(x)|\le C_r(1+|x|)^{d-r}$ on $|x|\le2R$ and the total differential order $2N$ give
\begin{equation}\label{eq:quadratic-residual-bound}
\sum_{|\beta|\le1}\left|
\partial_{\xi,\eps}^\beta\bigl[\mathscr E_{\eta h_R}(U_{\xi,\eps})-\eta f_{h_R,\xi,\eps}\bigr](x)
\right|
\le C\eta^2R^d(1+|x|)^{d-n}.
\end{equation}
Indeed, we retain one metric factor together with the bubble derivatives as in \eqref{eq:firstforcingdecay}; every additional metric factor is bounded by $C\eta R^d$, with its derivatives supplying the corresponding inverse powers of $1+|x|$.
Taylor's formula in the metric variable gives \eqref{eq:quadratic-residual-bound} uniformly when $\eta R^d$ is small.
Since $(1+|x|)^{d-n}\in L^{2n/(n+2N)}$ under \eqref{eq:admissibility}, \eqref{eq:dual-bound} bounds this remainder by $C\eta^2R^d$ in $C^1(\overline{\mathcal U};\mathcal H^*)$ and proves \eqref{eq:local-residual-estimate}.
\end{proof}

We write $w_{\xi,\eps}\coloneqq w_{\eta h_R,\xi,\eps}$.

\begin{proposition}\label{prop:localreduction}
For $R\ge1$ and $\eta>0$ with $\eta R^d$ sufficiently small,
\begin{equation}\label{eq:correctionbound}
\|w_{\xi,\eps}\|_{C^1(\overline{\mathcal U};\mathcal H)}\le C\eta
\end{equation}
and
\begin{equation}\label{eq:correctionlimit}
\|\eta^{-1}w_{\xi,\eps}+L_{U_{\xi,\eps}}^{-1}f_{h,\xi,\eps}\|_{C^1(\overline{\mathcal U};\mathcal H)}
\le C\bigl(R^{d-\frac{n-2N}{2}}+\eta R^d+\eta^\gamma\bigr).
\end{equation}
The corresponding energy satisfies
\begin{equation}\label{eq:localreducedexpansion}
\left\|\eta^{-2}\bigl[J_{\eta h_R}(U_{\xi,\eps}+w_{\xi,\eps})-J_0(U)\bigr]
-\frac{n-2N}{8}F_h(\xi,\eps)\right\|_{C^1(\overline{\mathcal U})}
\le C\bigl(R^{d-\frac{n-2N}{2}}+\eta R^d+\eta^\gamma\bigr).
\end{equation}
All constants are independent of $\eta,R$.
\end{proposition}
\begin{proof}
In this proof only, write $z=(\xi,\eps)$ and abbreviate the corresponding parameter subscripts and arguments by $z$.
Equation \eqref{eq:local-residual-estimate} implies $\|\mathscr E_{\eta h_R}(U_z)\|_{C^1(\overline{\mathcal U};\mathcal H^*)}\le C\eta$.
Proposition \ref{prop:projected-family} proves \eqref{eq:correctionbound}.
We divide the projected equation by $\eta$ and subtract the equation for $-L_{U_z}^{-1}f_{h,z}$.
The source error is bounded by \eqref{eq:local-residual-estimate}, the operator error by $C\eta R^d\|w_z/\eta\|_{\mathcal H}$, and the nonlinear error by $C\eta^\gamma$.
The differentiated estimates \eqref{eq:nonlinear-remainder-bound}--\eqref{eq:nonlinear-U-derivative}, together with the smooth projected inverse, imply the same bounds after one parameter derivative.
This proves \eqref{eq:correctionlimit}.

\medskip
The first metric variation of the energy vanishes by Proposition \ref{prop:jointcritical}.
After integration over the angular variable in $\Sph^{n-1}$, the absolute value of each quadratic metric term, and of its first parameter derivatives, is bounded by
\begin{equation}\label{eq:quadraticmajorant}
C\eta^2 r^{2d+2N-n-1},\qquad r\ge1.
\end{equation}
Indeed the two metric factors, two bubble factors, $2N$ derivatives, and volume measure produce the exponent $2d-2(n-2N)-2N+n-1=2d+2N-n-1$.
Integration over $r\in(R,\infty)$ therefore bounds the change in the quadratic metric term caused by replacing $h$ with $h_R$ by $C\eta^2R^{2d+2N-n}$.
Each additional metric factor is bounded by $C\eta R^d$, so terms of order at least three in the metric perturbation contribute $O(\eta^3R^d)$.
The pairing of $\mathscr E_{\eta h_R}(U_z)-\eta f_{h_R,z}$ with $w_z$, and the pairing of $(P_{2N,\exp(\eta h_R)}-(-\Delta)^N)w_z$ with $w_z$, have the same bound by \eqref{eq:quadratic-residual-bound}, \eqref{eq:operatorperturbation}, and \eqref{eq:correctionbound}.

Lemma \ref{lem:nonlinear-remainders} and \eqref{eq:correctionbound} bound the scalar energy remainder by $C\eta^{2+\gamma}$ in $C^1(\overline{\mathcal U})$.

At quadratic order the mixed and scalar terms are $\eta\langle f_{h_R,z},w_z\rangle+\langle L_{U_z}w_z,w_z\rangle/2$.
Substitution of \eqref{eq:correctionlimit} yields the scalar correction in \eqref{eq:correctedE}; Lemma \ref{lem:quadratic-coefficient} identifies the full quadratic coefficient at $U_z$ as $(n-2N)F_h(z)/8$.
Equations \eqref{eq:correctionlimit}--\eqref{eq:quadraticmajorant} and the scalar remainder estimate \eqref{eq:scalar-energy-remainder} prove \eqref{eq:localreducedexpansion}; here $R^{2d+2N-n}\le R^{d-(n-2N)/2}$ for $R\ge1$.
\end{proof}

For metric directions whose compactifications are TT with respect to $g_c$, the scalar correction vanishes and the quadratic energy is a positive multiple of the Hessian of total $Q_{2N}$-curvature.
\begin{proposition}\label{prop:energybridge}
Let $H,\widetilde H$ be trace-free tensor directions whose compactifications
\[
\bar H\coloneqq\Omega^2H,\qquad
\overline{\widetilde H}\coloneqq\Omega^2\widetilde H
\]
belong to $H^N(\Sph^n;\operatorname{Sym}^2T^*\Sph^n)$ and are weakly TT with respect to $g_c$.
Here $f_H$, $f_{\widetilde H}$, and $\cE$ are extended continuously from smooth spherical data; the last pairing in \eqref{eq:energybridge} is the duality pairing between the tensor spaces $H^N$ and $H^{-N}$.
Then $f_H=f_{\widetilde H}=0$ at $U=U_{0,1}$, and
\begin{equation}\label{eq:energybridge}
\cE(H,\widetilde H)=\frac{n-2N}{8}
D^2\mathcal Q(\bar H,\overline{\widetilde H})
=\frac{n-2N}{8}
\langle\bar H,\mathscr P_N^{(n)}(-\cA)\overline{\widetilde H}\rangle.
\end{equation}
\end{proposition}
\begin{proof}
For a smooth spherical TT tensor $\bar H$, Theorem \ref{thm:ambientHessian} gives $D_g(P_{2N,g}1)|_{g_c}[\bar H]=0$.
Conformal covariance therefore implies $f_H=0$, so the scalar correction vanishes.
The metric path $g_\eta=\Omega^2\exp(\eta H)$ preserves the round volume form, and \eqref{eq:covariance} implies
\[
J_{\eta H}(U)
=\frac{n-2N}{4}\mathcal Q(g_\eta)
-\frac{(n-2N)c_{n,N}}{2n}\Vol(g_c).
\]
Equivalently, this identity is evaluated on the compact sphere using \eqref{eq:geometricJ}; it does not require $H$ to have compact support.
The coefficient of $\eta^2$ is $(n-2N)D^2\mathcal Q|_{g_c}(\bar H,\bar H)/8$.
Polarization proves \eqref{eq:energybridge} for smooth TT data.
For weak TT data in $H^N$, use smooth TT approximation as in Theorem \ref{thm:ambientHessian}.
The second metric derivative and mixed metric--scalar derivative of $\mathcal J$ at $(g_c,1)$, and the form $\cE$, are continuous in these norms by Lemma \ref{lem:quadratic-continuity}.
Hence the source still vanishes and \eqref{eq:energybridge} holds in the limit.
\end{proof}

\section{Metric variations generated by Weyl tensors and parameter derivatives}\label{sec:weylmatrix}
We now restrict the fixed-volume Hessian of $\mathcal Q$ at $g_c$ in Theorem \ref{thm:ambientHessian} to the compactified polynomial TT tensors generated by algebraic Weyl tensors.
The resulting energy matrix $E$ and its scale derivative $S$ are defined in \eqref{eq:Ematrix}.
Indefiniteness of $S$, together with the identities for the translation derivatives of $F_{h_a}$ in Lemma \ref{lem:translation-Weyl}, produces the critical point of Proposition \ref{prop:seed}.

\subsection{Weyl tensors and the radial energy}\label{sec:seed}
We fix an integer $M\ge0$, set $d\coloneqq2M+2$, and assume $n>2N+2d$.
For a nonzero constant four-tensor $W$ satisfying \eqref{eq:Weyl-symmetries}, we define
\begin{equation}\label{eq:Weylchain}
H_r(W)_{ij}(x)\coloneqq |x|^{2r}W_{ikjl}x^kx^l,
\qquad d_r\coloneqq2r+2,\quad 0\le r\le M.
\end{equation}
\begin{lemma}\label{lem:weyl-TT}
The tensors in \eqref{eq:Weylchain} satisfy
\begin{equation}\label{eq:Weylgauge}
\tr H_r=0,\qquad x^iH_{r,ij}=0,\qquad \partial^iH_{r,ij}=0.
\end{equation}
For each $0\le r\le M$, the tensor $\bar H_r\coloneqq\Omega^2H_r$ belongs to $H^N(\Sph^n;\operatorname{Sym}^2T^*\Sph^n)$ and is weakly TT with respect to $g_c$.
The tensors $\Omega^2\chi_RH_r$ are smooth global TT tensors and converge to $\bar H_r$ in this space as $R\to\infty$.
\end{lemma}
\begin{proof}
The trace and radial contraction vanish by the Weyl symmetries.
For the divergence identity, derivatives falling on $|x|^{2r}$ vanish by $x^iH_{r,ij}=0$; the two derivatives of $x^kx^l$ vanish by skew symmetry and the Weyl trace identity, respectively.

With $\Omega$ as in \eqref{eq:roundbubble}, any trace-free $H$ satisfies
\begin{equation}\label{eq:conformaldivergence}
\nabla_{g_c}^{i}(\Omega^2H)_{ij}
=\partial^iH_{ij}+n(\partial^i\log\Omega)H_{ij}.
\end{equation}
Thus \eqref{eq:Weylgauge} implies that $\bar H_r$ is TT away from the north pole.
A radial cutoff preserves each identity in \eqref{eq:Weylgauge}, so $\Omega^2\chi(|x|/R)H_r$ is a smooth global TT tensor.
By Lemma \ref{lem:cutoff-convergence}, these tensors converge to $\bar H_r$ in $H^N$.
\end{proof}
Consequently, the continuous form in Theorem \ref{thm:ambientHessian} applies to these compactified tensors, although they need not be smooth at the north pole.

We normalize $W$ once and for all so that
\begin{equation}\label{eq:Wnorm}
\|\bar H_0(W)\|_{L^2(g_c)}^2=1.
\end{equation}
Orthogonality of two generating tensors $W,V$ refers to the Euclidean inner product $\langle W,V\rangle_\delta\coloneqq\sum_{i,j,k,l=1}^nW_{ijkl}V_{ijkl}$.
Each tensor is then normalized by \eqref{eq:Wnorm}.

For the tensors $H_j(W)$ of \eqref{eq:Weylchain}, normalized by \eqref{eq:Wnorm}, we set
\begin{equation}\label{eq:Ematrix}
\begin{gathered}
E_{ij}\coloneqq D^2\mathcal Q(\bar H_i,\bar H_j),\qquad
E\coloneqq(E_{ij})_{i,j=0}^M,\\
D\coloneqq\operatorname{diag}(2,4,\ldots,2M+2),\qquad S\coloneqq DE+ED.
\end{gathered}
\end{equation}
The pairings $E_{ij}$ are defined by Lemma \ref{lem:cutoff-convergence}.
When $M=3$, $S_N(n)$ denotes this matrix $S$, as in \eqref{eq:intro-dilation-matrix}.

Proposition \ref{prop:energybridge} gives $\cE(H_i,H_j)=(n-2N)E_{ij}/8$.
In particular, the scalar correction vanishes at $U_{0,1}$ for every linear combination of the tensors in \eqref{eq:Weylchain} and their radial cutoffs.
Translations need not preserve the spherical TT condition, so the full corrected form \eqref{eq:correctedE} is retained when differentiating in the center.

We set $t\coloneqq|x|^2$ and write $\cA(f(t)\bar H_0)=(Af)(t)\bar H_0$.
The normalized radial pairing is $\langle f,g\rangle_\mu\coloneqq\int_0^\infty f(t)g(t)\,d\mu(t)$; we write $\langle f\rangle_\mu\coloneqq\int_0^\infty f(t)\,d\mu(t)$.

\begin{lemma}\label{lem:weyl-radial}
The radial operator $A$ and the measure induced by \eqref{eq:Wnorm} are
\begin{equation}\label{eq:Amasterdiff}
Af=-t(1+t)^2f''+
\frac{1+t}{2}\bigl((n-8)t-(n+4)\bigr)f'
+(n-2)(1+t)f,
\end{equation}
where primes denote $t$-derivatives, and
\begin{equation}\label{eq:moments}
d\mu\coloneqq\frac{t^{\frac{n}{2}+1}(1+t)^{-n}}
{B(\frac{n}{2}+2,\frac{n}{2}-2)}\,dt,
\quad \mathfrak m_0=1,
\quad \frac{\mathfrak m_{r+1}}{\mathfrak m_r}=\frac{n+4+2r}{n-6-2r},
\end{equation}
with $\mathfrak m_r\coloneqq\langle t^r\rangle_\mu$ whenever the moments converge.
In particular,
\begin{equation}\label{eq:Amaster}
At^j={}-\frac12j(n+2j+2)t^{j-1}
+\bigl(n-2-2j(j+2)\bigr)t^j
+\frac12(j+2)(n-2j-2)t^{j+1}.
\end{equation}
The operator is symmetric for this pairing when $\deg f+\deg g<n/2-3$.
The radial pairings $\langle f,g\rangle_\mu$ converge whenever $\deg f+\deg g\le N+2M$ and $n>2N+4M+4$.
\end{lemma}
\begin{proof}
We set $r\coloneqq\tan(\varphi/2)$, so $g_c=d\varphi^2+\sin^2\varphi\,g_{\Sph^{n-1}}$.
Let $Y$ be the restriction of $W_{ikjl}\theta^k\theta^l$ to tangent vectors of $\Sph^{n-1}$.
This tensor is TT and has rough-Laplacian eigenvalue $2(n-1)$; this follows by twice differentiating its defining expression $W_{ikjl}\theta^k\theta^l$ tangentially and using the Weyl contractions.
For a tensor $\sin^2\varphi\,u(\varphi)Y$, the radial expression for $\nabla^*\nabla+2$ is
\[
-u''-(n-1)\cot\varphi\,u'
+\frac{2n}{\sin^2\varphi}u.
\]
Substituting $u=r^2f(r^2)$ and $\partial_\varphi=(1+r^2)\partial_r/2$ proves \eqref{eq:Amasterdiff}; applying it to $t^j$ gives \eqref{eq:Amaster}.
Polar coordinates and the beta integral prove \eqref{eq:moments}.
With $\omega(t)\coloneqq d\mu/dt$, the operator in \eqref{eq:Amasterdiff} has the divergence form
\[
A=-\omega^{-1}\partial_t\bigl(t(1+t)^2\omega\partial_t\bigr)
+(n-2)(1+t).
\]
Integration by parts gives the symmetry assertion.
The largest moment index in a matrix entry is $N+2M$, whereas $n/2-2>N+2M$, so all required integrals converge and their boundary terms vanish.
\end{proof}

For $a\coloneqq(\tau,a_1,\ldots,a_M)^{\mathsf T}$, we define $h_a\coloneqq\tau H_0(W)+\sum_{r=1}^M a_rH_r(W)$ and $a_0\coloneqq\tau$.

\begin{proposition}\label{prop:weyl-radial-energy}
The entries in \eqref{eq:Ematrix} satisfy
\begin{equation}\label{eq:radial-E}
E_{ij}=\langle t^i,\mathscr P_N^{(n)}(-A)t^j\rangle_\mu.
\end{equation}
Moreover,
\[
F_{h_a}(0,\eps)=\sum_{i,j=0}^M E_{ij}a_ia_j\eps^{d_i+d_j},
\]
and
\begin{equation}\label{eq:scalehessians}
\partial_\eps F_{h_a}(0,1)=a^{\mathsf T}Sa,
\qquad \partial_\eps^2 F_{h_a}(0,1)=a^{\mathsf T}(DS+SD-S)a.
\end{equation}
\end{proposition}
\begin{proof}
Theorem \ref{thm:ambientHessian} and Lemmas \ref{lem:weyl-TT} and \ref{lem:weyl-radial} imply \eqref{eq:radial-E}.
Homogeneity and \eqref{eq:energybridge} give the displayed expansion of $F_{h_a}$.
Two differentiations at $\eps=1$ yield \eqref{eq:scalehessians}.
For example, \eqref{eq:lowP} and \eqref{eq:moments} yield
\[
E_{00}\big|_{N=1}=-\frac{n-2}{2(n-6)}.
\]
\end{proof}

\subsection{Translation derivatives and coefficient choice}
We index coordinates modulo $n$ and prescribe
\begin{equation}\label{eq:symmetric-Weyl-choice}
W_{ijij}\coloneqq\mathbf1_{\{j=i-1\text{ or }j=i+1\}}-\frac2{n-1}
\qquad(i\ne j).
\end{equation}
The other components are determined by \eqref{eq:Weyl-symmetries} or are zero: a component vanishes unless its two unordered index pairs agree.

\begin{lemma}\label{lem:translation-Weyl}
For $n\ge5$, \eqref{eq:symmetric-Weyl-choice} defines a nonzero Weyl tensor invariant under coordinate sign changes and cyclic permutation.
After normalizing it by \eqref{eq:Wnorm}, the associated reduced energy satisfies
\begin{equation}\label{eq:translationisotropy}
\nabla_\xi F_{h_a}(0,\eps)=0,\qquad
D_\xi^2F_{h_a}(0,1)=(a^{\mathsf T}Ta)\Id_n,\qquad
D_\xi\partial_\eps F_{h_a}(0,1)=0,
\end{equation}
where $T$ is the symmetric $(M+1)\times(M+1)$ coefficient matrix of $\partial_{\xi_1}^2F_{h_a}(0,1)$.
It depends on $N,n$ and the normalized tensor \eqref{eq:symmetric-Weyl-choice}, not on $a$.
For $M\ge1$, its first two entries satisfy
\begin{equation}\label{eq:Tentries}
T_{00}=0,
\qquad T_{01}=\frac{2(n+4)}nE_{00}.
\end{equation}
\end{lemma}
\begin{proof}
The prescribed components satisfy the first Bianchi identity.
The constant $2/(n-1)$ removes the Ricci contraction, since
\[
\sum_{j\ne i}W_{ijij}=2-(n-1)\frac2{n-1}=0,
\]
and its off-diagonal entries vanish.
The tensor is nonzero for $n\ge5$.
Each nonzero component contains every coordinate index an even number of times, so coordinate sign changes preserve $W$; cyclic permutation preserves the adjacency relation.
We multiply $W$ by a positive constant to satisfy \eqref{eq:Wnorm}.

Invariance of $F_{h_a}$ in \eqref{eq:modelF} under coordinate sign changes forces $\partial_{\xi_i}F_{h_a}$, $\partial_{\xi_i}\partial_\eps F_{h_a}$, and $\partial_{\xi_i}\partial_{\xi_j}F_{h_a}$ for $i\ne j$ to vanish at $\xi=0$.
Cyclic permutation makes the diagonal entries equal.
This proves \eqref{eq:translationisotropy}.

Differentiating \eqref{eq:Weylchain} and using the zero Ricci contraction of $W$ yields
\[
\Delta H_0=0,
\qquad \Delta H_1=2(n+4)H_0.
\]
Two differentiations of \eqref{eq:modelF} yield
\[
\Delta_\xi F_h(0,1)=\frac{16}{n-2N}
\left\{\sum_{i=1}^n\cE(\partial_i h,\partial_i h)
+\cE(h,\Delta h)\right\}.
\]
Each $\partial_iH_0$ is affine.
Lemma \ref{lem:gauge} therefore implies $\cE(\partial_iH_0,H)=0$ whenever $\Omega^2H\in H^N(\Sph^n;\operatorname{Sym}^2T^*\Sph^n)$.
Comparing the coefficients of $\tau^2$ and $\tau a_1$ and using \eqref{eq:energybridge} gives \eqref{eq:Tentries}.
\end{proof}

\medskip
We now choose the coefficients of $h$ so that $D_\xi^2F_h(0,1)$ is invertible and $\partial_\eps^2F_h(0,1)\ne0$.
\begin{lemma}\label{lem:matrix-nonproportionality}
Let $S$ be a real symmetric indefinite $m\times m$ matrix, and let $D$ be diagonal.
Assume that the restriction of $S$ to the eigenspace $V$ for the smallest diagonal entry of $D$ is definite.
Then $DS+SD-S$ is not a scalar multiple of $S$.
If a nonzero symmetric matrix $T$ has zero restriction to $V$, then $T$ is not a scalar multiple of $S$ either.
\end{lemma}
\begin{proof}
Write the diagonal entries of $D$ as $d_i$ and their minimum as $d_*$.
If $DS+SD-S=cS$, restriction to $V$ gives $c=2d_*-1$.
Hence $(d_i+d_j-2d_*)S_{ij}=0$ for every $i,j\in\{1,\ldots,m\}$, so $S$ is supported on $V$.
Since its restriction there is definite, $S$ cannot be indefinite.
Similarly, $T=cS$ and $T|_V=0$ imply $c=0$, contradicting $T\ne0$.
\end{proof}

\begin{lemma}\label{lem:nullcone}
Let $S$ be a real symmetric indefinite $m\times m$ matrix of rank at least three.
If a quadratic form vanishes on the zero set $\{a\in\R^m:a^{\mathsf T}Sa=0\}$, it is a scalar multiple of $a\mapsto a^{\mathsf T}Sa$.
A finite union of zero sets of quadratic forms not proportional to this form, and of proper linear subspaces, cannot cover that zero set.
\end{lemma}
\begin{proof}
After a linear change of coordinates, the defining form is $a_1^2-\mathsf B(a')$, where $a'\coloneqq(a_2,\ldots,a_m)$ and $\mathsf B$ is positive on a nonempty open subset of $\R^{m-1}$.
Division in $a_1$ writes any polynomial as a multiple of $a_1^2-\mathsf B(a')$ plus $a_1L(a')+R(a')$.
If the polynomial vanishes on the zero set, evaluation at both $a_1=\pm\sqrt{\mathsf B(a')}$ implies $L=R=0$ on this open set, hence identically.
Thus it is divisible by the defining quadratic.
For a quadratic form, the quotient is constant.

A quadratic form of rank at least three is irreducible, since a product of two linear forms has rank at most two.
A proper linear subspace is contained in a hyperplane.
If finitely many of the stated zero sets covered $\{a\in\R^m:a^{\mathsf T}Sa=0\}$, the product of their quadratic or linear defining polynomials would vanish there and would therefore be divisible by the defining quadratic.
Irreducibility would force one factor to be divisible by it, contrary to the hypotheses.
\end{proof}

\begin{proposition}\label{prop:seed}
Let $n\ge25$, $M\ge1$, $d=2M+2$, and $n>2N+2d$.
If $E_{00}<0$ and the matrix $S$ in \eqref{eq:Ematrix} is indefinite, there is a trace-free polynomial $h$ of degree $d$, with nonzero homogeneous term of degree two, for which $(0,1)$ is a nondegenerate critical point of $F_h$.
\end{proposition}
\begin{proof}
By Lemma \ref{lem:translation-Weyl}, $T_{00}=0$ and $T_{01}\ne0$, whereas $S_{00}=4E_{00}<0$.
Lemma \ref{lem:matrix-nonproportionality} therefore shows that neither $T$ nor $DS+SD-S$ is proportional to $S$.
For rank at least three, Lemma \ref{lem:nullcone}, applied to $S,DS+SD-S,T$, ensures the existence of $a$ satisfying
\[
a^{\mathsf T}Sa=0,\qquad a^{\mathsf T}(DS+SD-S)a\ne0,\qquad a^{\mathsf T}Ta\ne0,\qquad
\tau a_M\ne0.
\]
Equations \eqref{eq:scalehessians} and \eqref{eq:translationisotropy} show that $D^2_{(\xi,\eps)}F_{h_a}(0,1)$ is
\[
\begin{pmatrix}(a^{\mathsf T}Ta)\Id_n&0\\0&a^{\mathsf T}(DS+SD-S)a\end{pmatrix},
\]
which is invertible.

\medskip
If $S$ has rank two, we use two linearly independent Weyl tensors with the same symmetries under coordinate sign changes and cyclic permutation.
Take \eqref{eq:symmetric-Weyl-choice} and the tensor obtained by replacing $j=i\pm1$ by $j=i\pm2$.
Their entries on pairs of cyclic distances one, two, and three show independence for $n\ge25$.
Orthogonalizing in $\langle\cdot,\cdot\rangle_\delta$ and then normalizing by \eqref{eq:Wnorm} gives tensors $W_1,W_2$ with the same symmetries.

The fourth-moment identity on $\Sph^{n-1}$ and the Weyl symmetries give
\[
\int_{\Sph^{n-1}}
\langle H_0(W_1),H_0(W_2)\rangle_\delta\,d\sigma
=c_n\langle W_1,W_2\rangle_\delta=0,
\qquad c_n>0.
\]
The radial operator in Lemma \ref{lem:weyl-radial} acts only on the radial factor and is independent of the generating Weyl tensor.
Consequently all pairings under $D^2\mathcal Q|_{g_c}$ between the two radial families vanish, first after multiplication by radial cutoffs and then by $H^N$ continuity.
Each normalized family has the same matrix $E$.
The sum of the two radial families therefore has scale-derivative matrix $S\oplus S$, of rank four.
Every such sum $h$ still has the shared coordinate symmetries, so $D_\xi^2F_h(0,1)$ remains a scalar multiple of $\Id_n$, although its coefficient may contain mixed products of coefficients from both families.
Its restriction to the two degree-two coefficients vanishes by Lemma \ref{lem:gauge} and $\Delta H_0(W_i)=0$; within either family, \eqref{eq:Tentries} gives a nonzero degree-two--degree-four entry.
The restriction of $S\oplus S$ to the degree-two coefficients is $4E_{00}\Id_2$, which is negative definite.
Applying Lemma \ref{lem:matrix-nonproportionality} with $D\oplus D$ therefore excludes both proportionalities as before.
Lemma \ref{lem:nullcone} now applies to this rank-four form, avoiding also the proper linear subspaces where the homogeneous term of degree two or degree $d$ in the combined tensor vanishes.
Thus $D^2_{(\xi,\eps)}F_h(0,1)$ is invertible.
\end{proof}

\section{Exact formulas and signs for the energy and scale-derivative matrices}\label{sec:boundedambient}
In this section, unless otherwise stated, we assume $N\ge1$ and $n>2N+16$, the integrability range for the polynomial symmetric two-tensors of degree at most eight used in the matrices.

\subsection{Polynomial ambient extensions and scalar extraction}
With $a$ as in \eqref{eq:roundambient}, $t=|x|^2$, and $\nu\coloneqq N-n/2$, we set
\begin{equation}\label{eq:ambient-R}
R\coloneqq\frac{1+\frac{\rho}{2}}{1-\frac{\rho t}{2}}.
\end{equation}
Only the finite Taylor expansion of this expression at $\rho=0$ is used.

\begin{proposition}\label{prop:Qj}
For $0\le j\le3$, we define
\begin{equation}\label{eq:Qclosed}
\Phi_j(R,t)\coloneqq R^2\sum_{h=0}^j
\frac{(-1)^h(-j)_h(j+2)_h}{(1-\frac{n}{2})_h\,h!}
[1+(t-1)R]^{j-h}[(R-1)(1+tR)]^h.
\end{equation}
Then $\widehat\Phi_j(\rho,t)\coloneqq \Phi_j(R,t)$ solves
\begin{equation}\label{eq:QODE}
2\rho a^2\partial_\rho^2\widehat\Phi_j+
\bigl(n\rho a-(n-2)a^2\bigr)\partial_\rho\widehat\Phi_j+A\widehat\Phi_j=0,\qquad
\widehat\Phi_j(0,t)=t^j,
\end{equation}
and $\deg_R\Phi_j\le2j+2$.
\end{proposition}
\begin{proof}
The initial value in \eqref{eq:QODE} follows from $R=1$ at $\rho=0$.
The degree bound follows term by term.
The differential identity can be checked as a finite polynomial identity without computing further Taylor coefficients of the ambient metric.
The chain rules needed for the check are
\begin{equation}\label{eq:Rchainrules}
\partial_\rho=\frac{(1+tR)^2}{2(1+t)}\partial_R,
\quad
\partial_t\big|_\rho=\partial_t\big|_R+
\frac{R(R-1)}{1+t}\partial_R,
\quad
\rho=\frac{2(R-1)}{1+tR}.
\end{equation}
Substituting \eqref{eq:Rchainrules} in \eqref{eq:QODE} and dividing by $1+t$ yields the polynomial identity
\begin{equation}\label{eq:QpolynomialPDE}
\begin{split}
0={}&R^2(R-1)\partial_R^2\Phi_j-2Rt(R-1)\partial_R\partial_t\Phi_j
-\frac12R(2R+n-4)\partial_R\Phi_j\\
&-t(1+t)\partial_t^2\Phi_j
+\frac12\bigl((n-8)t-n-4\bigr)\partial_t\Phi_j+(n-2)\Phi_j.
\end{split}
\end{equation}
We insert \eqref{eq:Qclosed} and clear the denominators depending on $n$.
The resulting polynomial coefficients cancel using
\[
\frac{(-j)_{h+1}(j+2)_{h+1}}{(1-\frac{n}{2})_{h+1}(h+1)!}
=\frac{(h-j)(h+j+2)}{(h+1-\frac{n}{2})(h+1)}
\frac{(-j)_h(j+2)_h}{(1-\frac{n}{2})_h h!}.
\]
Thus the left-hand side of \eqref{eq:QpolynomialPDE} is the zero polynomial for each $j=0,1,2,3$.
\end{proof}

For $j=0$, $\Phi_0=R^2$ and $\partial_\rho\widehat\Phi_0(0,t)=1+t=A1/(n-2)$.
For $0\le j\le3$, $\Phi_j$ has degree at most eight in $R$, independently of $N$.

With $W$ normalized by \eqref{eq:Wnorm}, we set $\Theta_{ij}\coloneqq\Theta_{\bar H_i,\bar H_j}$ for $0\le i,j\le3$ and define $\kappa_{ij,\ell}(t)$ by
\begin{equation}\label{eq:kappadefinition}
\frac{1+t}{2}\int_1^R u^2\partial_u\Phi_i(u,t)\,\partial_u\Phi_j(u,t)\,du
\eqqcolon\sum_\ell\kappa_{ij,\ell}(t)(R^\ell-1).
\end{equation}
We set
\begin{equation}\label{eq:xvariable}
\upsilon\coloneqq\frac{2\rho}{(1+\frac{\rho}{2})^2},\qquad
C_{N,n}\coloneqq\frac N2(\frac{n}{2}-N+1)_{N-1}(\frac{n}{2}+1)_{N-1}>0.
\end{equation}
We define the truncated series
\begin{equation}\label{eq:GN}
\mathcal G_N(\upsilon)\coloneqq{}_3F_2\left(
\begin{matrix}\frac{n}{2}-N,\ 1-N-\frac{n}{2},\ \frac12-N\\
1-2N,\ 1-N\end{matrix};\upsilon\right)
\pmod{\upsilon^N}.
\end{equation}
Only coefficients of order less than $N$ are used.
We set $\mathscr H_{ij}(\upsilon)\coloneqq a^2\Theta'_{ij}$, expressing $\rho$ as a formal power series in $\upsilon$ by inverting \eqref{eq:xvariable} near zero.

\begin{lemma}\label{lem:weyl-scalar-extraction}
The determinant variation satisfies
\begin{equation}\label{eq:ThetaWeyl}
(a^2\Theta'_{ij})'=-a^2\langle \partial_\rho\widehat\Phi_i,\partial_\rho\widehat\Phi_j\rangle_\mu,\qquad
\Theta'_{ij}(0)=-\frac{\langle t^i,At^j\rangle_\mu}{2(n-1)},
\end{equation}
and hence has the finite expression
\begin{equation}\label{eq:Hfinite}
a^2\Theta'_{ij}
=\Theta'_{ij}(0)-
\sum_\ell\langle\kappa_{ij,\ell}(t)(R^\ell-1)\rangle_\mu.
\end{equation}
The normalized matrix entries are
\begin{equation}\label{eq:normalizedcoefficient}
e_{ij}\coloneqq \frac{E_{ij}}{C_{N,n}}
=\frac{[\upsilon^{N-1}]\mathcal G_N(\upsilon)
(\upsilon\partial_\upsilon+1-\frac{n}{2})\mathscr H_{ij}(\upsilon)}
{(1-\frac{n}{2})[\upsilon^{N-1}]\mathcal G_N(\upsilon)}.
\end{equation}
For $N=1$, this reduces to $E_{ij}=\Theta'_{ij}(0)/2$.
\end{lemma}
\begin{proof}
The bilinear form of \eqref{eq:ThetaODE} gives \eqref{eq:ThetaWeyl}.
The chain rules \eqref{eq:Rchainrules} give
\begin{equation}\label{eq:detpolyidentity}
a^2(\partial_\rho\widehat\Phi_i)(\partial_\rho\widehat\Phi_j)\,d\rho
=\frac{1+t}{2}R^2(\partial_R\Phi_i)(\partial_R\Phi_j)\,dR.
\end{equation}
Integration and \eqref{eq:kappadefinition} prove \eqref{eq:Hfinite}.

The quadratic transformation \cite[(15.8.16)]{DLMF}, applied to \eqref{eq:scalarhypergeom}, gives
\[
f_N=a^{\nu}\,{}_2F_1\left(\frac{\frac{n}{2}-N}{2},\frac{1-N-\frac{n}{2}}{2};1-N;
\upsilon\right)\pmod{\rho^N}.
\]
Clausen's identity \cite[(16.12.2)]{DLMF} therefore yields
\begin{equation}\label{eq:Clausenhere}
a^n f_N^2=a^{2N}\mathcal G_N(\upsilon)\pmod{\rho^N}.
\end{equation}
These identities at singular lower parameters are read coefficientwise below order $N$, where all denominators are nonzero.

The change of variables
\[
\rho^{-N}a^{2N-3}(1-\frac{\rho}{2})\,d\rho
=2^{N-1}\upsilon^{-N}\,d\upsilon
\]
and integration by parts in the formal residue transform \eqref{eq:WronskianHessian} into
\[
E_{ij}=\frac{4^{N-1}((N-1)!)^2}{2\nu}
[\upsilon^{N-1}]\mathcal G_N(\upsilon)
(\upsilon\partial_\upsilon+1-\frac{n}{2})\mathscr H_{ij}(\upsilon).
\]
Cancellation of the rising factorials in the highest retained coefficient of \eqref{eq:GN} yields
\[
\frac{4^{N-1}((N-1)!)^2(1-\frac{n}{2})}{2\nu}
[\upsilon^{N-1}]\mathcal G_N(\upsilon)
=\frac N2(\frac{n}{2}-N+1)_{N-1}(\frac{n}{2}+1)_{N-1}=C_{N,n}.
\]
This proves \eqref{eq:normalizedcoefficient}, including its value for $N=1$.
\end{proof}
Appendix \ref{app:cauchy} carries out the finite coefficient extraction from \eqref{eq:Hfinite}.

\subsection{Finite summation and matrix signs}\label{sec:gaussian}\label{sec:signtheorem}

Equations \eqref{eq:Hfinite}, \eqref{eq:normalizedcoefficient}, and \eqref{eq:moments} lead to the following representation.

\begin{proposition}\label{prop:summand}
For $0\le i\le j\le3$ and $m\coloneqq i+j$, one has
\begin{equation}\label{eq:finite-summand}
e_{ij}=c_{ij}+
\sum_{k=0}^{N-1}
\frac{(1-N)_k(1+N)_k}{(2-\frac{n}{2})_k(1+\frac{n}{2})_k}
r_{ij}(k,n),
\end{equation}
where $c_{ij}$ and $r_{ij}$ are rational functions independent of $N$, and
\begin{equation}\label{eq:summand-degree}
r_{ij}(k,n)(k+2-\frac{n}{2})_{m+2}
\end{equation}
is a polynomial in $k$ of degree at most $m+1$.
\end{proposition}
The proof, including the construction of $c_{ij}$ and $r_{ij}$ by finite polynomial division and beta integration, is given in Appendix \ref{app:cauchy}.
The specific finite identities used in that extraction are verified in Lemma \ref{lem:finite-extraction-checks}; Table \ref{tab:finite-extraction-checks} records their pole and degree data.

For example,
\[
c_{00}=-\frac15,
\]
and
\begin{equation}\label{eq:r00}
r_{00}(k,n)=
-\frac{4(2k+3)(n-4)(n-1)}
{15(2k-n+4)(2k-n+6)}.
\end{equation}
Also
\[
c_{01}=-\frac{(n+4)(7n-12)}{35(n-6)(n-2)},
\]
\begin{equation}\label{eq:r01}
r_{01}(k,n)=
\frac{4(2k+3)(n-4)(n-1)(n+4)(4k+7n-11)}
{105(n-2)(2k-n+4)(2k-n+6)(2k-n+8)}.
\end{equation}
The other eight entries are determined by the same finite formulas in Appendix \ref{app:cauchy}.

For $m\coloneqq i+j$ as in Proposition \ref{prop:summand}, we use $\mathcal P_{m+1}\coloneqq\{p\in\R[k]:\deg p\le m+1\}$ and a real parameter $X$.
For $p\in\mathcal P_{m+1}$, we define
\begin{equation}\label{eq:gaussian-operator}
\cL_{m,X}p(k)
\coloneqq\bigl((k+1)^2-X\bigr)p(k+1)
-(k+\frac{n}{2})(k-\frac{n}{2}+m+3)p(k).
\end{equation}
\begin{lemma}\label{lem:detL}
The operator $\cL_{m,X}$ preserves $\mathcal P_{m+1}$, and
\begin{equation}\label{eq:gaussian-determinant}
\det\cL_{m,X}
=\prod_{\ell=1}^{m+2}\bigl((\frac{n}{2}-\ell)^2-X\bigr).
\end{equation}
\end{lemma}

\begin{proof}
For $p=k^r$, the degree-$r+2$ terms cancel, and the coefficient of $k^{r+1}$ is $r-m-1$.
It vanishes when $r=m+1$; for $r\le m$ the degree is already at most $m+1$.
Thus $\mathcal P_{m+1}$ is preserved.
For example, at $m=0$ the matrix in the basis $(1,k)$ is
\[
\begin{pmatrix}
\frac{n^2}{4}-\frac{3n}{2}+1-X&1-X\\-1&\frac{n^2}{4}-\frac{3n}{2}+3-X
\end{pmatrix}.
\]
At $X=(n/2-\ell)^2$, we define
\[
p_\ell(k)\coloneqq
(k+\frac{n}{2}-\ell+1)_{\ell-1}
(k-\frac{n}{2}+\ell+1)_{m+2-\ell}.
\]
Then
\[
\frac{p_\ell(k+1)}{p_\ell(k)}
=\frac{(k+\frac{n}{2})(k-\frac{n}{2}+m+3)}
{(k+1)^2-(\frac{n}{2}-\ell)^2},
\]
so $p_\ell$ lies in the kernel.
For generic $n$ the $m+2$ displayed roots are distinct.
The determinant has degree $m+2$ in $X$, and its leading coefficient is $(-1)^{m+2}$ because the translation $p(k)\mapsto p(k+1)$ has determinant one in the monomial basis.
The identity follows for generic $n$, and hence for all $n\in\mathbb R$ by polynomial continuation.
\end{proof}

\begin{theorem}\label{thm:closedentry}
Let $0\le i\le j\le3$ and $m\coloneqq i+j$.
Under \eqref{eq:admissibility} for the profiles $H_i,H_j$, the system
\begin{equation}\label{eq:gaussian-system}
\cL_{m,N^2}p_{ij}(k)
=r_{ij}(k,n)
(k+2-\frac{n}{2})_{m+2}
\end{equation}
has a unique solution $p_{ij}\in\mathcal P_{m+1}$, and
\begin{equation}\label{eq:closed-entry}
e_{ij}
=c_{ij}
-\frac{n\,p_{ij}(0)}
{2(2-\frac{n}{2})_{m+1}}.
\end{equation}
Thus every normalized entry $e_{ij}$, $0\le i,j\le3$, is a rational function of $(n,N^2)$ obtained from a linear system of size independent of $N$.
\end{theorem}
\begin{proof}
Since $\deg H_j=2j+2$ and $m=i+j\le2j$, condition \eqref{eq:admissibility} for these profiles implies
\[
n>2N+2m+4,
\]
so every factor in \eqref{eq:gaussian-determinant} with $X=N^2$ is positive.
The same inequality keeps the endpoint $g_{ij}(N)$ in the telescoping argument finite and places the discarded Cauchy-transform term at order $n/2-3-m>N-1$ in Appendix \ref{app:cauchy}; these are the three uses of this bound in the finite extraction.
Lemma \ref{lem:detL} and the degree bound in Proposition \ref{prop:summand} give the unique solution of \eqref{eq:gaussian-system}.
The system has size $m+2$.
For entries with $0\le i,j\le2$, one has $m\le4$, so the system has size at most six; for $0\le i,j\le3$, its size is at most eight.

We set
\[
\mathsf w_{N,k}\coloneqq\frac{(1-N)_k(1+N)_k}{(2-\frac{n}{2})_k(1+\frac{n}{2})_k}
\]
and
\[
g_{ij}(k)\coloneqq
\frac{(k+\frac{n}{2})p_{ij}(k)}
{(k+2-\frac{n}{2})_{m+1}}.
\]
Using
\[
\frac{\mathsf w_{N,k+1}}{\mathsf w_{N,k}}
=\frac{(k+1)^2-N^2}{(k+2-\frac{n}{2})(k+1+\frac{n}{2})},
\]
equation \eqref{eq:gaussian-system} is exactly
\begin{equation}\label{eq:telescoping}
\mathsf w_{N,k}r_{ij}(k,n)
=\mathsf w_{N,k+1}g_{ij}(k+1)-\mathsf w_{N,k}g_{ij}(k).
\end{equation}
Since $\mathsf w_{N,0}=1$ and $\mathsf w_{N,N}=0$, the sum in \eqref{eq:finite-summand} telescopes.
The inequality $n>2N+2m+4$ makes $g_{ij}(N)$ finite, so the term at $k=N$ is zero and the remaining term is $-g_{ij}(0)$.
This proves \eqref{eq:closed-entry}.
\end{proof}

\begin{corollary}\label{cor:weyl-leading-sign}
For $n>2N+4$, the first normalized entry is
\begin{equation}\label{eq:e00}
e_{00}=-\frac{F_N(n)}{
((n-2)^2-4N^2)((n-4)^2-4N^2)},
\end{equation}
where
\begin{equation}\label{eq:Fpolynomial}
F_N(n)\coloneqq{}n^4-8n^3
+\frac{68-8N^2}{3}n^2
+\frac{32N^2-80}{3}n
+\frac{16}{5}(N^2-1)(N^2-4).
\end{equation}
In particular, $e_{00}<0$ and
\begin{equation}\label{eq:B00}
-S_{00}=-4E_{00}>0.
\end{equation}
\end{corollary}
\begin{proof}
For $i=j=0$, the scalar extraction and finite-summation argument require only $n>2N+4$.
Solving the size-two system in Theorem \ref{thm:closedentry} gives \eqref{eq:e00}.
In \eqref{eq:finite-summand}, $c_{00}=-1/5$, $\mathsf w_{N,k}>0$, and \eqref{eq:r00} gives $r_{00}(k,n)<0$ for $0\le k<N$ when $n>2N+4$.
Thus $e_{00}<0$, and \eqref{eq:B00} follows from $S_{00}=4E_{00}$ and $C_{N,n}>0$.
\end{proof}

We use the rational matrix formulas to prove uniform signs in the required dimension ranges.

For $M=3$, we write $\widehat S\coloneqq-S/C_{N,n}$, so that $\widehat S_{ij}=-2(i+j+2)e_{ij}$ and $\widehat S_{00}>0$ by \eqref{eq:B00}.
For a fixed tail $\mathbf c=(c_1,c_2,c_3)$, we define
\begin{equation}\label{eq:tail-discriminant}
\mathfrak d_{\mathbf c}(N,n)\coloneqq
\left(\sum_{i=1}^3\widehat S_{0i}c_i\right)^2
-\widehat S_{00}\sum_{i,j=1}^3\widehat S_{ij}c_ic_j.
\end{equation}
\begin{lemma}\label{lem:tail-discriminant}
If $\mathfrak d_{\mathbf c}>0$, then $S$ is indefinite, and the equation $(\tau,\mathbf c)^{\mathsf T}S(\tau,\mathbf c)=0$ has the two roots
\begin{equation}\label{eq:tail-roots}
\tau=\frac{-\sum_{i=1}^3\widehat S_{0i}c_i
\pm\sqrt{\mathfrak d_{\mathbf c}(N,n)}}{\widehat S_{00}}.
\end{equation}
\end{lemma}
\begin{proof}
The quantity $\mathfrak d_{\mathbf c}$ is one quarter of the discriminant of $\tau\mapsto(\tau,\mathbf c)^{\mathsf T}\widehat S(\tau,\mathbf c)$.
The quadratic formula gives \eqref{eq:tail-roots}.
Completing the square yields
\begin{equation}\label{eq:tail-negative-direction}
\min_{\tau\in\R}(\tau,\mathbf c)^{\mathsf T}\widehat S(\tau,\mathbf c)
=-\frac{\mathfrak d_{\mathbf c}(N,n)}{\widehat S_{00}}.
\end{equation}
This exhibits a negative direction for $\widehat S$, whose $(0,0)$ entry is positive; hence $S$ is indefinite.
\end{proof}

\begin{theorem}\label{thm:algebraic}
For every pair $(N,n)$ satisfying \eqref{eq:dimensionrange} with $N\ge4$, the matrix $S$ for $M=3$ is indefinite.
More precisely, $\mathfrak d_{\mathbf c}(N,n)>0$ for the fixed tails in \eqref{eq:fixed-tails}.
\end{theorem}
\begin{proof}
The rational entries in Theorem \ref{thm:closedentry} give
\[
\mathfrak d_{\mathbf c}(N,n)
=\frac{\mathfrak A_{\mathbf c}(\frac{n}{2},N^2)}
{\mathfrak D(\frac{n}{2},N^2)}.
\]
The denominator is the same for both tails.
Its factorization \eqref{eq:tail-denominator} is strictly positive when $n-2N\ge17$.
The numerator normalization and the exact coefficient checks are given in Proposition \ref{prop:tail-certificates}.

For $\mathbf c=(150,-29,1)$, substitute
\[
N=18+u,\qquad n=2N+17+\eta.
\]
The polynomial
\[
2^{39}\mathfrak A_{\mathbf c}
(18+u+\frac{17+\eta}{2},(18+u)^2)
\]
has $858$ nonzero integer coefficients, all strictly positive, including its constant term.
Its degrees in $u,\eta$ are $32,41$.
Hence $\mathfrak d_{\mathbf c}>0$ for all $N\ge18$ and $n\ge2N+17$, without a subdivision of the dimension range.

For $\mathbf c=(100,-20,1)$ and each $4\le N\le17$, let $n_{\min}(N)$ be the first dimension in \eqref{eq:dimensionrange}.
The polynomial
\[
\mathfrak A_{\mathbf c}(\frac{n_{\min}(N)}{2}+u,N^2)
\]
has degree $41$ and all $42$ coefficients strictly positive.
These fourteen coefficient-positivity checks give $\mathfrak d_{\mathbf c}>0$ for every $n\ge n_{\min}(N)$.
Lemma \ref{lem:tail-discriminant} completes the proof.
\end{proof}

The fixed tails are used to prove indefiniteness of $S$ and to solve the scale equation $\partial_\eps F_{h_a}(0,1)=0$ in \eqref{eq:scalehessians}.
Proposition \ref{prop:seed} then chooses a zero of the scale-derivative form outside the zero sets of $a\mapsto a^{\mathsf T}(DS+SD-S)a$ and $a\mapsto a^{\mathsf T}Ta$, using Lemma \ref{lem:nullcone}; positivity of the discriminant alone does not impose these nondegeneracy conditions.

\begin{remark}\label{rem:no-universal-tail}
Within the family spanned by $H_0,H_1,H_2,H_3$, a single nonzero tail independent of both $N$ and $n$ cannot solve the scale equation throughout the range $N\ge4$ in \eqref{eq:dimensionrange}.
\end{remark}
\begin{proof}
The quadratic form
\[
\mathbf c\mapsto
\mathfrak d_{\mathbf c}(4,28)+\mathfrak d_{\mathbf c}(18,53)
\]
is negative definite by Lemma \ref{lem:common-tail-obstruction}.
Real roots in \eqref{eq:tail-roots} would require each discriminant to be nonnegative, which is impossible for a nonzero common tail.
Thus the use of two tails reflects an obstruction in this family.
\end{proof}

\section{Localization on a fixed metric}
\label{sec:analysis}
We fix a trace-free polynomial tensor $h$ of degree at most $d$, with $n>2N+2d$, and use the projected family of Proposition \ref{prop:projected-family}.
The estimates in Proposition \ref{prop:localreduction} concern one compactly supported metric perturbation.
We now place infinitely many such perturbations in a single smooth metric and estimate their contributions to the reduced energy near each concentration point.

Let $\mathcal U\Subset\R^n\times(0,\infty)$ be a small neighborhood of $(0,1)$.
The parameter norms are taken over $\overline{\mathcal U}$; constants may depend on $N,n,h,\chi,\mathcal U$, but not on the concentration scales or the index of a perturbation.

\subsection{The localized metric and reduced energy}\label{sec:fixedmetric}
For all sufficiently large $j\in\mathbb N$, we take
\[
x_j\coloneqq(2^{-j},0,\ldots,0),\qquad
\rho_j\coloneqq2^{-j^2-10},\qquad \mu_j\coloneqq\mu_0 2^{-j^3},
\]
where $\mu_0>0$ will be small.
The balls $B_{2\rho_j}(x_j)$ are disjoint.
We choose $0<\lambda_j<\rho_j$ and set
\begin{equation}\label{eq:globalbumps}
k_j(x)\coloneqq\mu_j\lambda_j^d\chi(\frac{|x-x_j|}{\rho_j})
h(\frac{x-x_j}{\lambda_j}),\qquad
k_\infty\coloneqq\sum_j k_j.
\end{equation}
We write $\eta_j\coloneqq\mu_j\lambda_j^d$ and $R_j\coloneqq\rho_j/\lambda_j$, and require
\begin{equation}\label{eq:outerchoice}
R_j\to\infty,\qquad
\frac{R_j^{2N-n}}{\eta_j^2}
=\frac{\lambda_j^{n-2N-2d}}{\mu_j^2\rho_j^{n-2N}}\to0.
\end{equation}
The second condition is equivalent to
\[
R_j^{-\frac{n-2N}{2}}=o(\eta_j),\qquad
R_j^{2N-n}=o(\eta_j^2).
\]
These are the decay rates for the equation error and energy contributions, respectively, of the other perturbations in \eqref{eq:remote-bump-estimates}.
We may choose each $\lambda_j$ so that $R_j^{-1}\le2^{-j}$ and the second quantity in \eqref{eq:outerchoice} is at most $2^{-j}$.
This is possible because $n>2N+2d$.
The jet bounds \eqref{eq:jetbound} are independent of these choices.

We set $\tilde{k}_j(y)\coloneqq k_\infty(x_j+\lambda_j y)$, so the rescaled metric is $\exp(\tilde{k}_j)$.
Let $w_{j,\xi,\eps}\coloneqq w_{\tilde{k}_j,\xi,\eps}$ be the correction from Proposition \ref{prop:projected-family}, and define
\[
J_j(\xi,\eps)\coloneqq J_{\tilde{k}_j}(U_{\xi,\eps}+w_{j,\xi,\eps}).
\]
The change of variables $x=x_j+\lambda_jy$, with the scalar factor multiplied by $\lambda_j^{-(n-2N)/2}$ as in Lemma \ref{lem:energy-invariance}, identifies $J_j$ with the energy on the fixed metric $\exp(k_\infty)$.

\begin{lemma}\label{lem:localized-energy}
For sufficiently small $\mu_0$, the tensor $k_\infty$ is smooth, compactly supported, and vanishes to infinite order at zero.
The projected corrections satisfy
\begin{equation}\label{eq:fixed-corrector-convergence}
\|w_{j,\xi,\eps}\|_{C^1(\overline{\mathcal U};\mathcal H)}\le C\eta_j,\qquad
\eta_j^{-1}w_{j,\xi,\eps}\to-L_{U_{\xi,\eps}}^{-1}f_{h,\xi,\eps}
\quad\text{in }C^1(\overline{\mathcal U};\mathcal H),
\end{equation}
and
\begin{equation}\label{eq:fixedreducedexpansion}
J_j(\xi,\eps)=J_0(U)+\frac{n-2N}{8}\eta_j^2
\bigl(F_h(\xi,\eps)+o_{C^1(\overline{\mathcal U})}(1)\bigr).
\end{equation}
\end{lemma}
\begin{proof}
In this proof only, write $z=(\xi,\eps)$ and abbreviate the corresponding parameter subscripts and arguments by $z$; write $z_0=(\xi_0,\eps_0)$ when a reference parameter is needed.
The series is locally finite away from zero.
For every fixed $r,m\in\mathbb N\cup\{0\}$, \eqref{eq:jetbound} implies
\[
\|\partial^r k_j\|_{L^\infty(\R^n)}\le C_r\mu_j\rho_j^{d-r}
=o(|x_j|^m).
\]
It therefore extends smoothly by zero at the accumulation point, with every derivative zero there.
Disjointness and \eqref{eq:jetbound} also imply $\kappa(k_\infty)\le C\sum_j\mu_j\rho_j^d\le C\mu_0$.
The bound $\kappa(k_\infty)\le C\mu_0$ is invariant under rescaling, so Proposition \ref{prop:projected-family} applies to $\tilde{k}_j$ for each retained index $j\in\mathbb N$.

The other perturbations contribute the differential operator
\[
\mathcal B_j\coloneqq P_{2N,\exp(\tilde{k}_j)}-P_{2N,\exp(\eta_jh_{R_j})}.
\]
The coefficients of $\mathcal B_j$ vanish on $B_{2R_j}(0)$; this operator is independent of $z$, and $\|\mathcal B_j\|_{\mathcal H\to\mathcal H^*}\le C\mu_0$.
For $|\beta|\le1$, differentiation of \eqref{eq:roundbubble} yields
\begin{equation}\label{eq:bubble-tail-H}
\|(1-\chi_R)\partial_z^\beta U_z\|_{\mathcal H}
\le CR^{-\frac{n-2N}{2}}.
\end{equation}
Indeed, \eqref{eq:roundbubble} implies $|\partial_x^N\partial_z^\beta U_z(x)|\le C|x|^{N-n}$ for $|x|\ge R$.
Squaring and integrating over $\R^n\setminus B_R(0)$ bounds the squared norm by $O(R^{2N-n})$; the terms in $\partial_x^N[(1-\chi_R)\partial_z^\beta U_z]$ with derivatives on $\chi_R$ satisfy the same squared $L^2$ bound on $R\le|x|\le2R$.
Locality of the bilinear form allows both bubble factors in $\langle \mathcal B_jU_z,U_z\rangle$ to be replaced by $(1-\chi_{R_j})U_z$.
Consequently
\begin{equation}\label{eq:remote-bump-estimates}
\|\mathcal B_jU_z\|_{C^1(\overline{\mathcal U};\mathcal H^*)}\le CR_j^{-\frac{n-2N}{2}},\qquad
\|\langle \mathcal B_jU_z,U_z\rangle\|_{C^1(\overline{\mathcal U})}\le CR_j^{2N-n}.
\end{equation}
Combining \eqref{eq:local-residual-estimate} and \eqref{eq:remote-bump-estimates} yields
\begin{equation}\label{eq:fixed-residual-limit}
\left\|\eta_j^{-1}\mathscr E_{\tilde{k}_j}(U_z)-f_{h,z}
\right\|_{C^1(\overline{\mathcal U};\mathcal H^*)}
\le C\left(R_j^{d-\frac{n-2N}{2}}+\mu_j\rho_j^d
+\eta_j^{-1}R_j^{-\frac{n-2N}{2}}\right)\to0.
\end{equation}
Proposition \ref{prop:projected-family} proves the first bound in \eqref{eq:fixed-corrector-convergence}.

The operators $\mathcal B_j$ need not tend to zero in global operator norm.
They do converge strongly from $\mathcal H$ to $\mathcal H^*$: on each compactly supported smooth test function their action is zero for all large $j\in\mathbb N$, with a uniform operator bound from \eqref{eq:operatorperturbation}.
Density of $C_c^\infty(\R^n)$ in $\mathcal H$ proves the assertion.
This convergence is uniform on compact subsets of $\mathcal H$.
By \eqref{eq:operatorperturbation}, $P_{2N,\exp(\eta_jh_{R_j})}-(-\Delta)^N$ tends to zero in operator norm since $\eta_jR_j^d=\mu_j\rho_j^d\to0$.

For $v,\phi\in K_z^\perp$, we define
\[
\langle A_{j,z}v,\phi\rangle
\coloneqq\left\langle
\bigl(L_{U_z}+P_{2N,\exp(\tilde{k}_j)}-(-\Delta)^N\bigr)v,\phi
\right\rangle,
\qquad \hat w_z\coloneqq-L_{U_z}^{-1}f_{h,z}.
\]
Thus $A_{j,z}:K_z^\perp\to(K_z^\perp)^*$ is the projected linear operator.
The inverse of $A_{j,z}$ and its first parameter derivatives are uniformly bounded after identifying $K_z^\perp$ with $K_{z_0}^\perp$ by Lemma \ref{lem:Jacobi-identifications}.
The families $\hat w_z$ and $\partial_z\hat w_z$ have compact image in $\mathcal H$.
Thus the strong convergence just proved, and the fact that the metric operators are independent of $z$, imply
\[
(A_{j,z}-L_{U_z})\hat w_z\to0
\quad\text{in }C^1(\overline{\mathcal U};(K_z^\perp)^*).
\]
The projection derivatives are norm-continuous by Lemma \ref{lem:Jacobi-identifications}; they therefore act on compact families in $\mathcal H$, so the same strong convergence applies to these terms.
Subtracting $L_{U_z}\hat w_z=-f_{h,z}$ from \eqref{eq:projectedequation} divided by $\eta_j$, we obtain
\[
\frac{w_{j,z}}{\eta_j}-\hat w_z
=-A_{j,z}^{-1}\left\{
\frac{\mathscr E_{\tilde{k}_j}(U_z)}{\eta_j}-f_{h,z}
+(A_{j,z}-L_{U_z})\hat w_z
-\frac{c_{n,N}}{\eta_j}\mathcal N_p(U_z,w_{j,z})\right\}.
\]
Equations \eqref{eq:fixed-residual-limit}, \eqref{eq:nonlinear-remainder-bound}, and \eqref{eq:nonlinear-U-derivative} make the right side tend to zero in $C^1(\overline{\mathcal U};\mathcal H)$.
This proves the second assertion in \eqref{eq:fixed-corrector-convergence}.

For the energy at $U_z$, the change in the quadratic metric term caused by replacing $h$ with $h_{R_j}$ and the terms of metric order at least three are bounded by $C\eta_j^2(R_j^{2d+2N-n}+\mu_j\rho_j^d)$, as in \eqref{eq:quadraticmajorant}.
The contribution $\langle\mathcal B_jU_z,U_z\rangle/2$ is $O(R_j^{2N-n})=o(\eta_j^2)$ in $C^1(\overline{\mathcal U})$ by \eqref{eq:remote-bump-estimates}.
In the expansion at $U_z+w_{j,z}$, \eqref{eq:fixed-residual-limit} and \eqref{eq:fixed-corrector-convergence} identify the mixed and scalar quadratic terms of \eqref{eq:correctedE}.
The remaining operator pairing also vanishes after division by $\eta_j^2$.
In fact,
\[
\eta_j^{-2}\langle\mathcal B_jw_{j,z},w_{j,z}\rangle\to0,\qquad
\eta_j^{-2}\partial_{z_\alpha}\langle\mathcal B_jw_{j,z},w_{j,z}\rangle
=2\eta_j^{-2}\langle\mathcal B_jw_{j,z},\partial_{z_\alpha}w_{j,z}\rangle
\to0.
\]
These limits follow from strong convergence, uniform boundedness, and the convergence of $w_{j,z}/\eta_j$ to $\hat w_z$ in $C^1(\overline{\mathcal U};\mathcal H)$.
After division by $\eta_j^2$, the pairing of $(P_{2N,\exp(\eta_jh_{R_j})}-(-\Delta)^N)w_{j,z}$ with $w_{j,z}$ is bounded by $C\mu_j\rho_j^d$ in $C^1(\overline{\mathcal U})$.
Lemma \ref{lem:nonlinear-remainders} bounds the scalar Taylor remainder by $O(\eta_j^{2+\gamma})$ in $C^1(\overline{\mathcal U})$.
This proves \eqref{eq:fixedreducedexpansion}.
\end{proof}

\subsection{Positivity and passage to a fixed metric}
\begin{lemma}\label{lem:greenpositive}
For fixed $n>2N$, there is a nonnegative integer $L_0$ such that every smooth metric $g$ in a sufficiently small $C^{L_0}$ neighborhood of $g_c$ has coercive $P_{2N,g}$ and a Green function that is strictly positive off the diagonal.
\end{lemma}
\begin{proof}
The factors in \eqref{eq:roundfactorization} have positive zeroth-order coefficients.
Thus $P_{2N,g_c}$ is coercive, and composition of the strictly positive Green kernels of these second-order factors gives a strictly positive round Green kernel.
Continuity of the quadratic form gives a finite integer $L_0$ and a neighborhood $\mathcal V$ of $g_c$ in the $C^{L_0}$ topology such that
\[
\int_{\Sph^n}vP_{2N,g}v\,dv_g
\ge c_0\|v\|_{H^N(g_c)}^2
\qquad(g\in\mathcal V)
\]
for a constant $c_0>0$ independent of $g$.
We increase $L_0$ and shrink $\mathcal V$ below as necessary; all metrics under consideration are smooth.
For fixed $n,N$, the constants in this proof are uniform over $\mathcal V$.

The local elliptic parametrix has the leading term
\begin{equation}\label{eq:greenleading}
G_g(x,y)=\gamma_{n,N}d_g(x,y)^{2N-n}
+o(d_g(x,y)^{2N-n}),\qquad
\gamma_{n,N}\coloneqq\frac{\Gamma(\frac{n}{2}-N)}{4^N\pi^{\frac{n}{2}}\Gamma(N)}>0.
\end{equation}
For the corresponding expansion on a fixed closed manifold, see \cite[(161)]{RobertLocalization}.
Here the relative remainder is uniform as the metric varies:
\[
\lim_{r\downarrow0}\;
\sup_{g\in\mathcal V}\;
\sup_{0<d_g(x,y)<r}
\left|d_g(x,y)^{n-2N}G_g(x,y)-\gamma_{n,N}\right|=0.
\]
To obtain this uniformity, construct the local parametrix to a fixed finite order sufficient to leave a bounded remainder kernel.
The coefficients involve finitely many metric jets, bounded uniformly in $\mathcal V$ once $L_0$ is large enough.
The metrics have a common positive injectivity-radius bound and uniformly equivalent coordinate norms.
The remaining kernel is controlled by elliptic estimates and the uniform inverse bound supplied by coercivity.
All lower singularities, including possible logarithmic terms, are $o(d_g^{2N-n})$ uniformly.
It follows that $G_g>0$ in a common neighborhood of the diagonal.

For each fixed $r>0$, the resolvent identity and interior elliptic estimates also give
\[
\sup_{d_{g_c}(x,y)\ge r}|G_g(x,y)-G_{g_c}(x,y)|\to0
\qquad\text{as }g\to g_c\text{ in }C^{L_0}.
\]
Choose $r$ so that $d_{g_c}(x,y)<r$ lies in the preceding common diagonal neighborhood.
The round kernel has a positive minimum on $\{d_{g_c}(x,y)\ge r\}$, so a further shrinking of $\mathcal V$ gives positivity there as well.
\end{proof}

\begin{lemma}\label{lem:positive-regularity}
For a metric in the neighborhood of Lemma \ref{lem:greenpositive}, every nonzero weak solution $u\in H^N(\Sph^n)$ of $P_{2N,g}u=c_{n,N}(u_+)^{(n+2N)/(n-2N)}$ is smooth and strictly positive.
\end{lemma}
\begin{proof}
The smooth, formally self-adjoint operator $P_{2N,g}$ has principal part $(-\Delta_g)^N$ and is coercive by Lemma \ref{lem:greenpositive}.
We apply \cite[Proposition 8.3]{MazumdarRegularity} with $k=N$, $\Omega_M=M=\Sph^n$ (so the boundary condition is empty), and $f=c_{n,N}$.
This regularity result includes the power of $u_+$ in our equation and gives $u\in C^{2N}(\Sph^n)$.

If $u_+$ were identically zero, coercivity would imply $u=0$.
Thus Lemma \ref{lem:greenpositive} and the Green representation
\[
u(x)=c_{n,N}\int_{\Sph^n}G_g(x,y)(u_+(y))^{\frac{n+2N}{n-2N}}\,dv_g(y)
\]
give $u(x)>0$ everywhere.
Since the manifold is compact, $u$ has a positive minimum.
The nonlinearity is smooth on its range, and elliptic regularity then gives $u\in C^\infty(\Sph^n)$.
\end{proof}

\begin{theorem}\label{thm:fixed-metric-transfer}
Let $h$ be a trace-free polynomial tensor on $\R^n$, with $\deg h\le d$ and $n>2N+2d$.
Suppose $F_h$ has a nondegenerate critical point at $(0,1)$.
Then there are a single smooth metric $g$ on $\Sph^n$ and positive smooth functions $u_j$ such that
\[
P_{2N,g}u_j=c_{n,N}u_j^{\frac{n+2N}{n-2N}},\qquad
\max_{\Sph^n}u_j\to\infty\quad\text{as }j\to\infty.
\]
The metric can be chosen arbitrarily close to $g_c$ in any prescribed finite $C^L$ norm, with coercive GJMS operator and positive Green function.
If in addition $h(0)=\partial h(0)=0$ and the degree-two part of $h$ is $H_0(W)$ for a nonzero constant tensor $W$ satisfying \eqref{eq:Weyl-symmetries}, then $g$ is not locally conformally flat.
\end{theorem}
\begin{proof}
We choose $\mathcal U$ to be a small ball about $(0,1)$ in the coordinates $(\xi,\eps)$, with no other critical point of $F_h$ in its closure.
The local degree of $\nabla F_h$ at $(0,1)$ is $\operatorname{sign}\det D^2F_h(0,1)\ne0$.
We construct $k_\infty$ by \eqref{eq:globalbumps}--\eqref{eq:outerchoice}, and put
\begin{equation}\label{eq:fixedsphereg}
g\coloneqq\Omega^2\exp(k_\infty)
\end{equation}
in the stereographic chart.
It equals the round metric $g_c$ outside a compact subset of the chart and extends smoothly to $\Sph^n$.
The jet bounds are independent of $\lambda_j$, so $\mu_0$ can be chosen to make $g-g_c$ as small as prescribed in $C^L$ and in the finite differentiability norm required by Lemma \ref{lem:greenpositive}.

By \eqref{eq:fixedreducedexpansion}, $8\nabla J_j/((n-2N)\eta_j^2)$ converges uniformly to $\nabla F_h$.
Thus $C^1$ convergence of the energies supplies the uniform gradient convergence needed for degree stability; no $C^2$ convergence of $J_j$ is required.
Since $\nabla F_h$ is bounded away from zero on $\partial\mathcal U$, homotopy invariance of degree gives critical points $(\xi_j,\eps_j)$ of $J_j$ in $\mathcal U$.
Every accumulation point in $\overline{\mathcal U}$ is a zero of $\nabla F_h$ by uniform convergence.
Since $(0,1)$ is the only such zero, $(\xi_j,\eps_j)\to(0,1)$.
Equation \eqref{eq:multiplier-matrix} makes every multiplier zero.
Rescaling as in \eqref{eq:globalbumps} yields weak solutions
\[
u_j^\flat(x)\coloneqq\lambda_j^{-\frac{n-2N}{2}}(U_{\xi_j,\eps_j}+w_{j,\xi_j,\eps_j})
(\frac{x-x_j}{\lambda_j}),\qquad
P_{2N,\exp(k_\infty)}u_j^\flat=c_{n,N}(u_j^\flat)_+^{\frac{n+2N}{n-2N}}.
\]
By \eqref{eq:covariance} and Proposition \ref{prop:projected-inverse}, $u_j\coloneqq\Omega^{-(n-2N)/2}u_j^\flat$ belongs to $H^N(\Sph^n)$ and satisfies the required equation away from the north pole.
For a smooth test function, its product with a standard cutoff equal to one in the geodesic ball of radius $r$ about the north pole and zero outside the ball of radius $2r$ has $H^N(\Sph^n)$ norm $O(r^{(n-2N)/2})$.
Both sides of the equation define continuous functionals on $H^N(\Sph^n)$.
Testing away from the north pole and letting $r\downarrow0$ therefore proves the weak equation on all of $\Sph^n$.
The solutions are nonzero by their $L^{2n/(n-2N)}$ closeness to a bubble.
Lemma \ref{lem:positive-regularity} proves smoothness and strict positivity.

Equation \eqref{eq:fixed-corrector-convergence} implies
\begin{equation}\label{eq:rescaledconvergence}
\lambda_j^{\frac{n-2N}{2}}u_j^\flat(x_j+\lambda_j y)\to U(y)
\quad\text{in }\dot H^N(\R^n)\text{ and }L^{\frac{2n}{n-2N}}(\R^n).
\end{equation}
Equivalently,
\[
\lambda_j^{\frac{n-2N}{2}}\Omega(x_j+\lambda_j y)^{\frac{n-2N}{2}}
u_j(x_j+\lambda_j y)\to U(y).
\]
The factor $\Omega$ is bounded above and below near zero.
A fixed positive amount of $L^{2n/(n-2N)}$ mass lies in $B_{C\lambda_j}(x_j)$, whose $g$-volume is $O(\lambda_j^n)$.
Hence
\[
\|u_j\|_{L^\infty(\Sph^n)}\ge c\lambda_j^{-\frac{n-2N}{2}}\to\infty.
\]

Under the additional hypothesis on the second jet, $k_j$ and its first derivatives vanish at $x_j$.
Its quadratic term is $\mu_j\lambda_j^{d-2}H_0(W)$.
The normal-coordinate curvature formula therefore implies nonzero Weyl curvature at $x_j$.
Higher polynomial terms and the nonlinear terms in the exponential do not affect this second jet.
Conformal multiplication by $\Omega^2$ preserves nonvanishing of the Weyl tensor.
This proves the last assertion.
\end{proof}

\subsection{Positive mass for the constructed metrics}\label{sec:positive-mass}
We prove that the metrics constructed from the polynomial perturbations can also be chosen to satisfy the positive-mass condition \eqref{eq:mass-condition}.
Throughout this subsection, $n>2N+3$.
The set of poles $\mathcal Z(g)$ is defined by the Weyl-derivative conditions in \eqref{eq:mass-locus}.
We fix a sufficiently large integer $T_0>n+2N+8$ and a smooth family of conformal-normal representatives
\[
g_\xi:=\Lambda_\xi^{\frac{4}{n-2N}}g,\qquad
\Lambda_\xi(\xi)=1,\qquad d\Lambda_\xi(\xi)=0,
\]
whose geodesic normal coordinates satisfy $\det g_\xi(x)=1+O(|x|^{T_0})$.
We choose this family by the equivariant construction in Lemma \ref{lem:CN-finite-part} below.
The local relation between metric Taylor coefficients and covariant derivatives of the Weyl tensor in \cite[(3.21)]{CompactnessMass} identifies the Weyl-derivative conditions defining $\mathcal Z(g)$ with vanishing of the conformal-normal metric terms through degree $\lfloor(n-2N)/2\rfloor$.
The Green function expansion in \cite[Theorem 5.8]{CompactnessMass} therefore gives, at every $\xi\in\mathcal Z(g)$,
\begin{equation}\label{eq:CN-mass-expansion}
\begin{split}
G_{g_\xi}(\xi,x)={}&\gamma_{n,N}r^{2N-n}
+\sum_{e=d_0}^{n-2N}r^{e+2N-n}a_{\xi,e}(\theta)
+m_g(\xi)+O_{C^{2N}}(r^\eps),\\
d_0:={}&\left\lfloor\frac{n-2N}{2}\right\rfloor+1,
\qquad r:=|x|,\qquad \theta:=\frac{x}{r}.
\end{split}
\end{equation}
Here $\gamma_{n,N}$ is the fundamental-solution coefficient in \eqref{eq:greenleading}, every $a_{\xi,e}$ has zero spherical mean, and
$0<\eps<\min\{1,2d_0+2N-n\}$.
The remainder notation means that its derivatives of order $j$ are $O(r^{\eps-j})$ for $0\le j\le2N$.
Thus $m_g(\xi)$ is the finite part after subtraction of the radial singularity and every displayed angular term, including the bounded term with $e=n-2N$.
Equivalently, it is the limit of the spherical average after subtraction of $\gamma_{n,N}r^{2N-n}$.

\begin{proposition}[Positive mass along a fixed perturbation]\label{prop:mass-ray}
Let $H$ be a smooth symmetric two-tensor on $\Sph^n$ such that $\dot W:=DW_{g_c}[H]$ is not identically zero, and set
\begin{equation}\label{eq:mass-ray}
g_t:=g_c\exp(tg_c^{-1}H).
\end{equation}
There are constants $t_0>0$ and $c>0$, depending on this fixed variation, such that
\begin{equation}\label{eq:uniform-CN-mass}
m_{g_t}(\xi)\ge ct^2
\qquad\bigl(0<|t|<t_0,\ \xi\in\mathcal Z(g_t)\bigr).
\end{equation}
After decreasing $t_0$, the operators $P_{2N,g_t}$ are coercive and their Green functions are strictly positive off the diagonal.
\end{proposition}

The operator assertions follow from Lemma \ref{lem:greenpositive}.
To prove the mass estimate on the moving set $\mathcal Z(g_t)$, we first extend its finite part to all poles.

\begin{lemma}[A smooth finite part at every pole]\label{lem:CN-finite-part}
For metrics near $g_c$, the conformal-normal family can be chosen so that a function $\mathfrak m(g,\xi)$ is smooth on every smooth finite-dimensional family of metrics and poles,
agrees with $m_g(\xi)$ for $\xi\in\mathcal Z(g)$, and satisfies
\begin{equation}\label{eq:CN-stationarity}
\mathfrak m(g_c,\xi)=0,\qquad
D_g\mathfrak m(g_c,\xi)[H]=0
\qquad\text{for every }H\text{ and }\xi.
\end{equation}
The construction is equivariant under round isometries.
\end{lemma}
\begin{proof}
The value and gradient conditions on $\Lambda_\xi$ fix its Taylor terms of degrees zero and one.
At degree $e\ge2$, a further conformal change $g\mapsto e^{2\varphi_e}g$, where $\varphi_e$ is homogeneous of degree $e$, changes the degree-$e$ term of $\log\det g$ in the new normal coordinates by
\[
\frac{2}{e(e+1)}\bigl(r^2\Delta\varphi_e+e(e-1)(n-2)\varphi_e\bigr),
\]
as follows from the linearized geodesic equation.
The operator $r^2\Delta$ has nonnegative eigenvalues on homogeneous polynomials of degree $e$, so this map is invertible.
The determinant normalization therefore determines the required conformal Taylor coefficients uniquely and smoothly from the metric jet.
The recursion and geodesic coordinates are natural under rotations of the orthonormal frame.
At $g_c$, choose the normalized representative to be exactly flat near $\xi$.
For nearby metrics, retain the exact round-to-flat conformal factor there and extend the difference of the prescribed Taylor coefficients with a fixed radial cutoff.
Both the round factor and the cutoff can be chosen equivariantly under round isometries.
Two such extensions with the same prescribed jet differ by a conformal factor $1+O(r^{T_0})$, and their centered normal coordinates differ by $x\mapsto x+O(r^{T_0+1})$.
Conformal covariance and this coordinate change alter the leading singular term of the Green function by $O(r^{T_0+2N-n})$ and the lower singularities by terms of still higher order.
Since $T_0>n-2N$, they give the same finite part.

In the resulting coordinates, we construct the parametrix for the Green function through homogeneity zero; see \cite[Theorem 5.1]{Ludewig} for the general expansion of elliptic Green functions of this type.
For a spherical harmonic of degree $\ell$, the indicial roots of $(-\Delta)^N$ are
\[
\ell+2a\quad\text{and}\quad-\ell-n+2+2a,
\qquad 0\le a\le N-1.
\]
There is no root strictly between $2N-n$ and $0$, while at homogeneity zero the only kernel is the constant angular mode.
The negative homogeneous corrections are therefore uniquely determined.
At homogeneity zero, prescribe zero spherical mean for the bounded angular correction and include a radial logarithm when its source is nonzero.
Fix $\log r$ without a parameter-dependent additive constant.
This specifies a cutoff local singular function $S_{g,\xi}$ containing all singular and bounded nonconstant terms.
Its distributional residual satisfies
\[
E_{g,\xi}:=P_{2N,g_\xi}S_{g,\xi}-\delta_\xi
=O\bigl(r^{1-2N}(1+|\log r|)^b\bigr)
\]
for a fixed nonnegative integer $b$.
The corrections have larger homogeneity than the fundamental solution, so their boundary terms at $r=0$ vanish and produce no additional distributions supported at $\xi$.
Choose $a$ with $n/(2N)<a<n/(2N-1)$.
After pulling each moving pole to the coordinate origin, every fixed parameter derivative of the residual obeys the same $O(r^{1-2N}(1+|\log r|)^b)$ bound, uniformly on compact parameter sets.
The parametrix coefficients depend smoothly on the metric jets, so dominated differentiation proves smooth dependence of the residual in $L^a$.
Coercivity, the resolvent identity, and elliptic estimates give the same dependence for
\[
G_{g_\xi}(\xi,\cdot)-S_{g,\xi}
=-P_{2N,g_\xi}^{-1}E_{g,\xi}
\quad\text{in }W^{2N,a}\hookrightarrow C^0.
\]
We define $\mathfrak m(g,\xi)$ to be the value of this remainder at the fixed coordinate origin.
Frame equivariance makes these local definitions agree as the pole varies.

At $\xi\in\mathcal Z(g)$, the conformal-normal metric logarithm starts at degree $d_0$.
Since $2d_0>n-2N$, every nonlinear metric contribution has degree too large to enter the parametrix through homogeneity zero.
Each remaining linear scalar source has zero spherical mean: averaging its metric coefficient over $O(n)$ gives zero, because an invariant homogeneous symmetric tensor that is both trace-free and tangential must vanish.
Thus every radial correction beyond the leading fundamental solution, including the radial logarithm, vanishes at these poles, and the expansion is precisely \eqref{eq:CN-mass-expansion}.
This proves agreement with $m_g$.

For a conformally round metric, the explicit conformal transform of the round Green function has zero finite part in a locally flat representative.
A local Euclidean M\"obius transformation imposes the required value and gradient normalizations on its conformal factor.
Uniqueness of the normalized conformal Taylor recursion then identifies its jet with the chosen representative to the order needed for the finite part.
Hence $\mathfrak m$ also vanishes for every conformally round metric near $g_c$.
Fix $\xi$ and average $H$ over the orthogonal stabilizer of $\xi$ in the round isometry group.
Equivariance and linearity give the same value of $D_g\mathfrak m(g_c,\xi)$ on $H$ and on this averaged tensor.
The average is rotationally symmetric, and its exponential path from $g_c$ is conformally round.
Indeed, for a rotationally symmetric metric $a(r)^2dr^2+b(r)^2g_{\Sph^{n-1}}$, the change of variable
$d\vartheta/\sin\vartheta=a(r)\,dr/b(r)$ gives its conformal equivalence to the round metric.
The coordinate change and conformal factor extend smoothly at both poles by the smoothness conditions on $a$ and $b$.
The finite part along the averaged path is consequently zero, proving \eqref{eq:CN-stationarity}.
\end{proof}

Away from $\mathcal Z(g)$, only the fixed finite-part convention just constructed is used.
In particular, the proof does not require conformal covariance of a general renormalized mass in even dimensions; see \cite[Theorem 7.1]{Ludewig}.

We next record the flat quadratic form that determines $D_g^2\mathfrak m(g_c,\xi)$.
We use the componentwise extension of $\mathcal H=\dot H^N(\R^n)$ from \eqref{eq:homogeneous-norm} for tensors and set $P_0:=(-\Delta)^N$.
Writing $\mathcal F$ for the unitary Fourier transform, the same norm is
\[
\|v\|_{\mathcal H}^2=\int_{\R^n}|\zeta|^{2N}|(\mathcal Fv)(\zeta)|^2\,d\zeta.
\]
For compactly supported perturbations of $\delta$, we define $\mathcal Q(g):=\int_{\R^n}Q_{2N,g}\,dv_g$ as in the closed case.
The flat metric is stationary for this integral, because its linearization is an integrated divergence.
Theorem \ref{thm:ambientHessian} and its leading coefficient \eqref{eq:TT-leading-coefficient} give
\begin{equation}\label{eq:flat-mass-TT}
D^2\mathcal Q|_\delta[T,T]
=-\hat c_{N,n}\|T\|_{\mathcal H}^2,
\qquad
\hat c_{N,n}:=(-1)^{N+1}[b^N]\mathscr P_N^{(n)}(b)>0
\end{equation}
for Euclidean TT tensors $T$.
For $N=1$, the same definition gives $\hat c_{1,n}=1/[4(n-1)]$.
Indeed, the operator representing $D^2\mathcal Q|_\delta$ has constant coefficients and total differential order $2N$; on Euclidean TT tensors, its symbol agrees with the principal symbol of the operator representing $D^2\mathcal Q|_{g_c}$.
All lower-order terms contain background curvature and vanish at the flat metric.

For a general symmetric tensor $K\in\mathcal H$, the Euclidean decomposition is
\begin{equation}\label{eq:flat-mass-decomposition}
K=T+\mathcal L_X\delta+2\omega\delta,
\qquad T:=\Pi_{\mathrm{TT}}K,
\qquad \nabla X,\omega\in\mathcal H.
\end{equation}
Here $\mathcal L_X\delta$ is the Lie derivative of the Euclidean metric.
The TT projection is given, for $\zeta\ne0$, by
\[
\Pi(\zeta):=I-\frac{\zeta\otimes\zeta}{|\zeta|^2},\qquad
\mathcal F(\Pi_{\mathrm{TT}}K)(\zeta)
:=\Pi(\zeta)(\mathcal F K)(\zeta)\Pi(\zeta)
-\frac{\tr\bigl(\Pi(\zeta)(\mathcal F K)(\zeta)\bigr)}{n-1}\Pi(\zeta).
\]
This bounded degree-zero symbol defines the projection on $\mathcal H$ and determines the remaining components in \eqref{eq:flat-mass-decomposition}.
The linearization of $Q_{2N,g}$ at the flat metric is
\begin{equation}\label{eq:flat-Q-linearization}
DQ_{2N,\delta}[K]
=\frac{(-\Delta)^{N-1}}{2(n-1)}
\bigl(\partial_i\partial_jK_{ij}-\Delta\tr K\bigr)
=P_0\omega.
\end{equation}
It vanishes on the TT and Lie derivative components, and the conformal component follows from \eqref{eq:covariance}.
For compactly supported $\omega$, conformal covariance also gives the exact identity
\[
\mathcal Q(e^{2t\omega}\delta)
=\frac2{n-2N}\int_{\R^n}
e^{\frac{(n-2N)t\omega}{2}}P_0e^{\frac{(n-2N)t\omega}{2}}\,dx.
\]
Its second derivative is $(n-2N)\|\omega\|_{\mathcal H}^2$.
Diffeomorphism invariance eliminates the Lie derivative component of the Hessian, while \eqref{eq:flat-Q-linearization} eliminates the mixed TT--conformal term.
It follows that
\begin{equation}\label{eq:flat-mass-Hessian}
D^2\mathcal Q|_\delta[K,K]
=-\hat c_{N,n}\|T\|_{\mathcal H}^2
+(n-2N)\|\omega\|_{\mathcal H}^2.
\end{equation}
The identities first hold for compactly supported variations and extend as quadratic forms on $\mathcal H$ by their Fourier symbols and density.

We define the compact set determined by the linearized Weyl derivatives of the fixed variation in \eqref{eq:mass-ray} by
\begin{equation}\label{eq:linearized-mass-locus}
\mathcal Z_{\mathrm{lin}}:=\left\{\xi\in\Sph^n:
\nabla_{g_c}^j\dot W(\xi)=0\quad
0\le j\le\left\lfloor\frac{n-2N-4}{2}\right\rfloor\right\}.
\end{equation}

\begin{lemma}[Second variation of the finite part]\label{lem:CN-mass-Hessian}
For every $\xi\in\mathcal Z_{\mathrm{lin}}$, the normalized first metric variation determines a tensor $K_\xi\in\mathcal H$ such that
\begin{equation}\label{eq:CN-mass-Hessian}
\left.\frac{d^2}{dt^2}\right|_{t=0}\mathfrak m(g_t,\xi)
=\gamma_{n,N}^2\frac{n-2N}{2}\hat c_{N,n}
\|\Pi_{\mathrm{TT}}K_\xi\|_{\mathcal H}^2>0.
\end{equation}
\end{lemma}
\begin{proof}
Fix $\xi\in\mathcal Z_{\mathrm{lin}}$.
Let $\mathcal B_{\xi,t}$ denote the conformal normalization and centered coordinate change of $g_t$, extended smoothly away from its centered chart, and set $\bar g_t:=\mathcal B_{\xi,t}(g_t)$.
We define the endomorphism
\[
L_\xi:=\left.\frac{d}{dt}\right|_{t=0}
\log(\bar g_0^{-1}\bar g_t).
\]
Linearizing the relation between metric Taylor coefficients and covariant derivatives of the Weyl tensor used in \eqref{eq:CN-mass-expansion} shows that $L_\xi$ starts at degree $d_0$ near the pole.
Lemma \ref{lem:CN-finite-part} makes $\bar g_0$ exactly flat in that chart.
Since the coordinates are normal for every sufficiently small $t\in\mathbb R$, $L_\xi(x)x=0$ there, and determinant normalization gives $\tr L_\xi=O(r^{T_0})$.

The original nonlinear path can have lower metric jets of order $t^2$.
To obtain the decay needed after inversion, we replace it by the normalized auxiliary path
\[
\bar g_t^*:=\bar g_0\exp(tL_\xi),\qquad
g_t^*:=\mathcal B_{\xi,t}^{-1}(\bar g_t^*).
\]
The compact paths $g_t^*$ and $g_t$ have the same value and first derivative at zero.
The frozen logarithm $tL_\xi$ preserves the normal-coordinate identity and the determinant normalization through order $T_0$.
The conformal factor used in $\mathcal B_{\xi,t}$ still has value one and zero gradient at the pole.
Uniqueness of the normalized conformal Taylor recursion therefore identifies $\bar g_t^*$ with the chosen conformal-normal representative of $g_t^*$ to the order needed for the finite part.
Its logarithm starts at degree $d_0$ for every sufficiently small $t\in\mathbb R$, so $\xi\in\mathcal Z(g_t^*)$.
Since the two paths have the same first derivative, the chain rule expresses the difference of their second finite-part derivatives as $D_g\mathfrak m(g_c,\xi)$ applied to the difference of their accelerations.
This term vanishes by \eqref{eq:CN-stationarity}.

We apply the fixed conformal factor $(G_{\bar g_0}(\xi,\cdot)/\gamma_{n,N})^{4/(n-2N)}$ to $\bar g_t^*$.
The background becomes Euclidean on the punctured sphere, with coordinates $y=x/|x|^2$ near infinity.
We write the resulting path as $\hat g_t=\exp(tK_\xi)$ in the global Euclidean coordinates.
The tensor $K_\xi$ satisfies
\begin{equation}\label{eq:mass-AF-decay}
\partial^jK_\xi=O(|y|^{-d_0-j}),\qquad
K_\xi(y)y=0,\qquad \tr K_\xi=O(|y|^{-T_0})
\end{equation}
near infinity, through every fixed derivative order used below.
Since $2d_0+2N>n$, cutoff approximation places $K_\xi$ in the tensor space $\mathcal H$.

We define
\[
f_t:=P_{2N,\hat g_t}1=\frac{n-2N}{2}Q_{2N,\hat g_t},
\qquad
u_t:=1-P_{2N,\hat g_t}^{-1}f_t.
\]
The operator and forcing are interpreted weakly using $dv_{\hat g_t}$, so
$\langle f_t,v\rangle:=\int_{\R^n}f_tv\,dv_{\hat g_t}$ for compactly supported test functions.
The linear part of $f_t$ belongs to $\mathcal H^*$ by \eqref{eq:flat-Q-linearization}, and its nonlinear remainder is $O(|y|^{-2d_0-2N})$ and belongs to $L^{2n/(n+2N)}$.
The inverse $P_{2N,\hat g_t}^{-1}:\mathcal H^*\to\mathcal H$ exists for small $t\in\mathbb R$ by conformal transport of the coercive compact quadratic form.
Conformal covariance and uniqueness identify $u_t$, in the punctured-sphere coordinates, with
\[
u_t=\frac{G_{\bar g_t^*}(\xi,\cdot)}{G_{\bar g_0}(\xi,\cdot)}.
\]
Indeed, this quotient solves $P_{2N,\hat g_t}u_t=0$, and its difference from one belongs to $\mathcal H$ by \eqref{eq:CN-mass-expansion}.
Writing $m_t:=m_{g_t^*}(\xi)$, the same expansion gives
\[
\frac1{|\Sph^{n-1}|}\int_{\Sph^{n-1}}u_t(R\theta)\,d\theta
=1+\frac{m_t}{\gamma_{n,N}}R^{2N-n}
+O(R^{2N-n-\eps}).
\]
Its retained nonspherical terms start at order $R^{-d_0}$ and have zero spherical mean.
Their coefficients are linear in $t$, since quadratic metric terms have degree at least $2d_0>n-2N$ before inversion.
Differentiating the parameter-dependent parametrix and inverse, followed by scaled interior Schauder estimates, gives these expansions with their scaled derivatives through order $2N$ for the first two parameter derivatives as well.

Green's identity applied to $1$ and $u_t$ yields
\[
m_t=-\gamma_{n,N}^2\lim_{R\to\infty}
\int_{B_R}f_tu_t\,dv_{\hat g_t}.
\]
We justify the limit and its two derivatives explicitly.
In the formally self-adjoint divergence form of $P_{2N,\hat g_t}$, the potential cancels from the skew Green identity.
The remaining terms have at least one derivative on each scalar factor before integration by parts.
Every boundary term beyond the flat expression contains a metric coefficient variation and a derivative of $u_t-1$.
The decay in \eqref{eq:mass-AF-decay} and the differentiated Green expansion bound its integral by
\begin{equation}\label{eq:mass-boundary-error}
O\bigl(R^{n-2N-2d_0}\bigr)\to0.
\end{equation}
For the flat operator, spherical averaging removes every angular term and gives
\[
\int_{\partial B_R}\partial_\nu(-\Delta)^{N-1}(u_t-1)\,d\sigma
\to-\frac{m_t}{\gamma_{n,N}^2},
\]
where $\nu$ is the outward Euclidean unit normal.
This fixes the minus sign in the expression for $m_t$ obtained from Green's identity.
For the linear curvature term, we set $V_i:=\partial_j(K_\xi)_{ij}$.
Tangency gives $y\cdot V=-\tr K_\xi$, so the spherical mean of
$\partial_iV_i-\Delta\tr K_\xi$ is $O(R^{-T_0-2})$.
Spherical averaging commutes with $\Delta$, and the boundary integral obtained from the linearized curvature in \eqref{eq:flat-Q-linearization} is therefore $O(R^{n-2N-T_0})$ and tends to zero.
All quadratic and higher curvature-density terms are $O(R^{-2d_0-2N})$ and are integrable.
The product of the linear curvature term with $u_t-1$ has the same integrable decay.
These bounds hold for the first two parameter derivatives, which justifies differentiating the improper integrals twice.

Since $u_t-1=-P_{2N,\hat g_t}^{-1}f_t$, the identity becomes
\begin{equation}\label{eq:mass-action}
m_t=-\gamma_{n,N}^2\left(
\lim_{R\to\infty}\int_{B_R}f_t\,dv_{\hat g_t}
-\langle f_t,P_{2N,\hat g_t}^{-1}f_t\rangle\right).
\end{equation}
The brackets are the dual pairing of $\mathcal H^*$ and $\mathcal H$.
At zero, $f_0=0$.
Decompose $K_\xi=T_\xi+\mathcal L_X\delta+2\omega_\xi\delta$ as in \eqref{eq:flat-mass-decomposition}; then
\[
f'_0=\frac{n-2N}{2}P_0\omega_\xi.
\]
The quadratic integrals in \eqref{eq:flat-mass-Hessian} pass to $K_\xi$ by cutoff approximation in $\mathcal H$.
The acceleration of the exponential path is $K_\xi^2=O(R^{-2d_0})$, whose linear-curvature boundary contribution also tends to zero by \eqref{eq:mass-boundary-error}.
Thus differentiating \eqref{eq:mass-action} twice gives
\[
\begin{aligned}
m''_0={}&-\gamma_{n,N}^2\frac{n-2N}{2}
\left(-\hat c_{N,n}\|T_\xi\|_{\mathcal H}^2
+(n-2N)\|\omega_\xi\|_{\mathcal H}^2\right)\\
&\quad+2\gamma_{n,N}^2\left(\frac{n-2N}{2}\right)^2
\|\omega_\xi\|_{\mathcal H}^2\\
={}&\gamma_{n,N}^2\frac{n-2N}{2}\hat c_{N,n}
\|T_\xi\|_{\mathcal H}^2.
\end{aligned}
\]
The scalar contribution cancels exactly.

Finally, $T_\xi$ cannot vanish.
For $n\ge4$, the kernel of the principal symbol of the linearized Weyl tensor consists of infinitesimal conformal and coordinate changes.
For example, when the frequency $\zeta$ is parallel to $e_1$, a TT tensor satisfies $(\mathcal FT)_{1j}=0$ and
\[
\mathcal F(DW_\delta[T])_{1a1b}
=\frac{n-3}{2(n-2)}|\zeta|^2(\mathcal FT)_{ab}
\qquad(2\le a,b\le n),
\]
up to the common choice of curvature sign.
The conformal and Lie derivative components have zero linearized Weyl tensor.
If $T_\xi=0$, it follows that $DW_\delta[K_\xi]=0$ on $\R^n$.
Conformal and coordinate transport would then give $DW_{g_c}[H]=0$ on the punctured sphere and, by smoothness, on all of $\Sph^n$.
This contradicts the hypothesis of Proposition \ref{prop:mass-ray} and proves \eqref{eq:CN-mass-Hessian}.
\end{proof}

\begin{proof}[Proof of Proposition \ref{prop:mass-ray}]
Lemma \ref{lem:CN-finite-part} gives the uniform expansion
\begin{equation}\label{eq:CN-uniform-Taylor}
\mathfrak m(g_t,\xi)=\frac{t^2}{2}B(\xi)+O(t^3),
\qquad
B(\xi):=\left.\frac{d^2}{dt^2}\right|_{t=0}\mathfrak m(g_t,\xi),
\end{equation}
where $B$ is continuous on the whole sphere.
If $\mathcal Z_{\mathrm{lin}}$ is nonempty, Lemma \ref{lem:CN-mass-Hessian} gives a positive minimum of $B$ on this compact set.
Choose a neighborhood $V$ of $\mathcal Z_{\mathrm{lin}}$ and $b>0$ such that $B\ge b$ on $V$.
For every required derivative order,
\begin{equation}\label{eq:mass-Weyl-limit}
t^{-1}\nabla_{g_t}^jW_{g_t}
\to\nabla_{g_c}^j\dot W
\quad\text{uniformly on }\Sph^n\text{ as }t\to0.
\end{equation}
Every limit of points $\xi_i\in\mathcal Z(g_{t_i})$ with $t_i\to0$ therefore lies in $\mathcal Z_{\mathrm{lin}}$.
Compactness implies $\mathcal Z(g_t)\subset V$ for all sufficiently small nonzero $t$.
On this set $\mathfrak m=m_g$, and \eqref{eq:CN-uniform-Taylor} gives
$m_{g_t}(\xi)\ge bt^2/4$ after decreasing $t_0$.
If $\mathcal Z_{\mathrm{lin}}$ is empty, \eqref{eq:mass-Weyl-limit} instead implies that $\mathcal Z(g_t)$ is empty for sufficiently small nonzero $t$, so the asserted estimate is vacuous.
\end{proof}

\begin{proposition}[Mass for the noncompactness metrics]\label{prop:mass-noncompactness}
For $N\ge4$ and $n$ in \eqref{eq:dimensionrange}, the metrics supplied by Theorem \ref{thm:fixed-metric-transfer} for the polynomial symmetric two-tensors of Proposition \ref{prop:seed} with $M=3$ can also be chosen to satisfy \eqref{eq:mass-condition}.
The same assertion holds for the sixth-order construction of \cite[Theorem 1.4]{GongKimWei} in every dimension $n\ge27$.
In each case, the metric remains arbitrarily close to $g_c$ in any prescribed finite $C^L$ norm and supports an $L^\infty$-unbounded sequence of positive solutions.
\end{proposition}
\begin{proof}
For $N\ge4$ and $n$ in \eqref{eq:dimensionrange}, fix the polynomial $h$ supplied by Proposition \ref{prop:seed} with $M=3$, whose degree is at most $d=8$.
Fix the points $x_j$, radii $\rho_j$, and relative amplitudes $\bar\mu_j:=\bar\mu_0 2^{-j^3}$ as in Subsection \ref{sec:fixedmetric}, with $\bar\mu_0>0$ fixed and sufficiently small.
Choose every $\lambda_j$ before choosing a common additional amplitude, so that \eqref{eq:outerchoice} holds with $\mu_j=\bar\mu_j$.
Let $\bar k_\infty$ be the resulting tensor in \eqref{eq:globalbumps}, and consider the fixed ray
\begin{equation}\label{eq:fixed-mass-ray}
g_t:=\Omega^2\exp(t\bar k_\infty),\qquad
H:=\Omega^2\bar k_\infty.
\end{equation}
By Lemma \ref{lem:localized-energy}, $\bar k_\infty$ is smooth, compactly supported, and vanishes to infinite order at the accumulation point.
The identities in \eqref{eq:Weylgauge} show that each translated radial cutoff preserves the Euclidean trace and divergence conditions.
Thus $\bar k_\infty$ is a nonzero Euclidean TT tensor.
The calculation of the principal symbol of the linearized Weyl tensor in Lemma \ref{lem:CN-mass-Hessian} implies $DW_\delta[\bar k_\infty]\not\equiv0$, and conformal covariance gives $DW_{g_c}[H]\not\equiv0$.

Replacing $\bar\mu_j$ by $\mu_j:=t\bar\mu_j$ multiplies the second quotient in \eqref{eq:outerchoice} by $t^{-2}$.
Its limit remains zero for every fixed $t>0$.
Consequently \eqref{eq:fixedreducedexpansion} and the proof of Theorem \ref{thm:fixed-metric-transfer} give an unbounded sequence of positive solutions for every sufficiently small fixed $t>0$ in \eqref{eq:fixed-mass-ray}.
Proposition \ref{prop:mass-ray} then permits this same $t$ to be chosen so that \eqref{eq:mass-condition} holds.

\medskip
For $N=3$, use the polynomial symmetric two-tensor $h$ of degree ten and the configuration of \cite[Section 11.7]{GongKimWei}:
\[
x_j=(j^{-1},0,\ldots,0),\qquad
\lambda_j=2^{-j},\qquad
\rho_j=(4j^2)^{-1},\qquad
\bar\mu_j=2^{-\frac{j}{4}}.
\]
Fix the initial index and the full tensor
\[
\bar k_\infty(x):=\sum_j\bar\mu_j\lambda_j^{10}
\chi(\frac{|x-x_j|}{\rho_j})h(\frac{x-x_j}{\lambda_j})
\]
before choosing $t$ in \eqref{eq:fixed-mass-ray}.
Scaling the generating Weyl tensor by $t$ multiplies $\bar k_\infty$ by $t$ and preserves the prescription $\bar\mu_j=\lambda_j^{1/4}$ used there.
The complete quadratic reduced energy is multiplied by $t^2$, so its strict local minimum from \cite[Corollary 11.31]{GongKimWei} persists.
The remote error in \cite[(11.34)]{GongKimWei}, divided by the leading quadratic scale, is bounded for fixed $t>0$ by a constant times
\[
t^{-2}(4j^2)^{n-6}2^{-j(n-\frac{53}{2})}\to0
\qquad(n\ge27).
\]
The normalized local error tends to zero as well, since it is $O_t(\bar\mu_j\lambda_j^{20\eps_1})$ for the fixed $\eps_1>0$ in that estimate.
Thus their reduction still gives infinitely many solutions on each sufficiently small fixed metric $g_t$.
The tensor $\bar k_\infty$ is smooth, compactly supported, nonzero, and Euclidean TT, so $DW_{g_c}[H]\not\equiv0$ and Proposition \ref{prop:mass-ray} applies.
Finally, decreasing the fixed amplitude $t$ gives any prescribed finite-norm closeness in both constructions.
\end{proof}

\medskip
For $N=1,2$, the positive-mass assertion for the known noncompactness examples follows from \cite[Corollary 5.12]{CompactnessMass}.
Its hypotheses hold because the metrics can be chosen sufficiently close to $g_c$ to have positive Yamabe constant and the operator properties of Lemma \ref{lem:greenpositive}, while their nonzero Weyl curvature excludes the round conformal class.
That corollary uses the mass identities in \cite[Proposition 5.11]{CompactnessMass}, together with the positive mass theorem at order two and the positive energy theorem at order four.
For order six and above, Proposition \ref{prop:mass-noncompactness} supplies the additional positivity argument.

Condition \eqref{eq:mass-condition} is nonvacuous for these metrics: it includes the open round region outside all perturbations and their accumulation point.
At the accumulation point the conformal-normal metric jet is flat to every order, so the angular terms in \eqref{eq:CN-mass-expansion} vanish.
The uniform estimate also covers every additional point in a transition region where the required Weyl derivatives vanish.

\section{Proofs of Theorems \ref{thm:main} and \ref{thm:cross}}
\label{sec:applications}\label{sec:finish}\label{sec:asymptoticcross}
\begin{proof}[Proof of Theorem \ref{thm:main}]
For $N=1$, the noncompactness assertion follows from Brendle \cite{BrendleBlowup} when $n\ge52$ and Brendle--Marques \cite{BrendleMarques} when $25\le n\le51$.
For $N=2$ it follows from Wei--Zhao \cite[Theorem 1.1]{WeiZhao}, and for $N=3$ from Gong--Kim--Wei \cite[Theorem 1.4]{GongKimWei}.
A constant rescaling of the solutions gives the normalization in \eqref{eq:introPDE}.

The finite-norm closeness in our statement also follows from these constructions.
Their final perturbations are sums of shrinking bumps with exponentially decaying amplitudes and polynomially growing cutoff derivatives; for every fixed $L\in\mathbb N\cup\{0\}$, the series converges in $C^L$ and its tail tends to zero in that norm.
Increasing the initial gluing index retains infinitely many bubbles while making the metric arbitrarily close to $g_c$; see \cite[Proposition 25]{BrendleBlowup}, \cite[Proposition 25]{BrendleMarques}, \cite[Proposition 12.2]{WeiZhao}, and \cite[Section 11.7]{GongKimWei}.
Choosing this closeness in the finite norm required by Lemma \ref{lem:greenpositive} also gives coercivity and strict positivity of the Green function.
For $N=1,2$, condition \eqref{eq:mass-condition} follows from \cite[Corollary 5.12]{CompactnessMass}, since the examples have positive Yamabe constant and are not conformally round.
For $N=3$, the fixed-amplitude choice in Proposition \ref{prop:mass-noncompactness} proves the same mass assertion.

\medskip
We now assume $N\ge4$ and fix a pair in \eqref{eq:dimensionrange}.
Theorem \ref{thm:algebraic} and \eqref{eq:B00} give indefiniteness of $S=S_N(n)$ and $E_{00}<0$.
The family $H_0,H_1,H_2,H_3$ has maximal degree $d=8$, and $n-2N>16=2d$.
In particular, $n\ge28$.

Proposition \ref{prop:seed} supplies a polynomial $h$ with a nondegenerate critical point at $(0,1)$ and nonzero homogeneous term of degree two.
By \eqref{eq:Weylchain}, $h(0)=\partial h(0)=0$.
Thus all hypotheses of Theorem \ref{thm:fixed-metric-transfer} hold.
That theorem constructs the fixed non-locally-conformally-flat metric, arbitrary finite-norm closeness to $g_c$, and the positive blow-up sequence.
Proposition \ref{prop:mass-noncompactness} chooses the amplitude of this fixed configuration so that \eqref{eq:mass-condition} holds and the positive blow-up sequence persists.
Conformal covariance \eqref{eq:covariance} implies $Q_{2N,u_j^{4/(n-2N)}g}=2c_{n,N}/(n-2N)$ for every $j\in\mathbb N$.
\end{proof}

\begin{proof}[Proof of Theorem \ref{thm:cross}]
We use the rational entries in Theorem \ref{thm:closedentry}.
For $M=3$, the entries of $-S_N(n)/C_{N,n}$ satisfy, as $N\to\infty$,
\begin{equation}\label{eq:matrix-asymptotic}
-\frac{(S_N(n))_{ij}}{C_{N,n}}
=N^2\left(M_{0,ij}(n-2N)+\frac{M_{1,ij}(n-2N)}{N}
+O(N^{-2})\right).
\end{equation}
The remainder is uniform, together with each fixed number of $n$-derivatives, when $n-2N$ belongs to a compact subset of $(16,\infty)$.
Indeed, substitution of $n=2N+x$ into the rational entries, followed by division by $N^2$, yields rational functions of $1/N$ and $x$ whose denominators have nonzero limits at $1/N=0$ for $x>16$.

For each fixed $(i,j)\in\{0,1,2,3\}^2$, write $e_{ij}=\mathsf U/\mathsf V$ and let $\alpha$ be the degree in $N$ of the numerator after substituting $n=2N+x$.
The numerator and denominator then have the expansions
\[
\begin{aligned}
\mathsf U&=\mathsf u_0(n-2N)N^\alpha+\mathsf u_1(n-2N)N^{\alpha-1}+\cdots,\\
\mathsf V&=\mathsf v_0(n-2N)N^{\alpha-2}+\mathsf v_1(n-2N)N^{\alpha-3}+\cdots.
\end{aligned}
\]
The quotient expansion is
\[
M_{0,ij}=-2(i+j+2)\frac{\mathsf u_0}{\mathsf v_0},\qquad
M_{1,ij}=-2(i+j+2)\frac{\mathsf u_1\mathsf v_0-\mathsf u_0\mathsf v_1}{\mathsf v_0^2},
\]
with all functions evaluated at $n-2N$.
The entries of $M_0$ are displayed in Appendix \ref{app:crosspoly}.

Taking the determinant yields
\begin{equation}\label{eq:determinant-asymptotic}
\frac{\det S_N(n)}{C_{N,n}^4}
=N^8\left(\mathcal F_0(n-2N)+\frac{\mathcal F_1(n-2N)}{N}+O(N^{-2})\right),
\end{equation}
where
\[
\mathcal F_0\coloneqq\det M_0,\qquad
\mathcal F_1\coloneqq\tr\bigl(\operatorname{adj}(M_0)M_1\bigr).
\]
Exact simplification yields
\begin{equation}\label{eq:determinant-denominator}
\begin{split}
\mathcal D(n-2N)\coloneqq{}&66467128245643125\,(n-2N-16)\\
&\cdot(n-2N-14)^2(n-2N-12)^3
\prod_{a=1}^5(n-2N-2a)^4
\end{split}
\end{equation}
and
\begin{equation}\label{eq:determinant-coefficients}
\mathcal F_0(n-2N)=\frac{2^{32}G_0(n-2N)}{\mathcal D(n-2N)},\qquad
\mathcal F_1(n-2N)=\frac{2^{32}G_1(n-2N)}{\mathcal D(n-2N)}.
\end{equation}
The polynomials $G_0,G_1$ are given in Appendix \ref{app:crosspoly}.

Lemma \ref{lem:crossing-signs} gives a unique simple root $a_*$ of $G_0$ in $(16,17)$, with $G_0'(a_*)<0$, and positivity of the first three leading principal minors of $M_0(a_*)$.
Since $\mathcal D(a_*)>0$, we have $\mathcal F_0'(a_*)<0$.

We define
\[
\Xi(\hbar,x)\coloneqq
\left.N^{-8}C_{N,n}^{-4}\det S_N(n)
\right|_{N=\frac{1}{\hbar},\ n=\frac{2}{\hbar}+x}.
\]
Before substitution, the quotient is expressed through the rational entries $e_{ij}$; no operator of noninteger order or continuation of the finite product $C_{N,n}$ is required.
By \eqref{eq:matrix-asymptotic}, $\Xi$ extends analytically to $\hbar=0$ near $x=a_*$ and satisfies
\[
\Xi(\hbar,x)=\mathcal F_0(x)+\hbar \mathcal F_1(x)+O(\hbar^2).
\]
The implicit-function theorem ensures a unique zero $x(\hbar)$ near $a_*$, with $x(0)=a_*$ and
\[
x'(0)=-\frac{\mathcal F_1(a_*)}{\mathcal F_0'(a_*)}.
\]
Since the two determinant coefficients have the same denominator and $G_0(a_*)=0$, this derivative is
\begin{equation}\label{eq:crossing-correction}
c_*=-\frac{G_1(a_*)}{G_0'(a_*)}.
\end{equation}
Evaluation yields the constants in Theorem \ref{thm:cross}.
Setting $n_N=2N+x(1/N)$ proves \eqref{eq:crossstatement}.
The first three leading principal minors of $-S_N(n)/C_{N,n}$ remain positive on a fixed neighborhood of $n_N-2N$, while its determinant decreases through zero.
Sylvester's criterion proves that $S_N(n)$ is negative definite just below $n_N$ and has exactly one positive eigenvalue just above it.
Shrinking the fixed neighborhood if necessary proves all assertions of Theorem \ref{thm:cross}.
\end{proof}

\appendix

\section{Finite extraction of the rational summands}\label{app:cauchy}
As in Section \ref{sec:boundedambient}, we assume $N\ge1$ and $n>2N+16$.
We construct the functions $c_{ij}$ and $r_{ij}$ in \eqref{eq:finite-summand} from the profiles $\Phi_j$ and the polynomials $\kappa_{ij,\ell}$ in \eqref{eq:Qclosed} and \eqref{eq:kappadefinition}; no interpolation in $N$ is involved.

We set
\[
c\coloneqq\sqrt{1-\upsilon},\qquad \zeta\coloneqq\frac{t-1}{t+1}.
\]
Then $R=(1-\zeta)/(c-\zeta)$, and the measure in \eqref{eq:moments} is proportional to $(1-\zeta^2)^{n/2-3}(1+\zeta)^4\,d\zeta$.
We define
\[
\langle f\rangle_0\coloneqq
\frac{\displaystyle\int_{-1}^1 f(\zeta)(1-\zeta^2)^{\frac{n}{2}-3}\,d\zeta}
{\displaystyle\int_{-1}^1(1-\zeta^2)^{\frac{n}{2}-3}\,d\zeta}.
\]
For integers $r\ge0$ and $c>1$, we define
\[
\nu_r\coloneqq\langle\zeta^r\rangle_0,\qquad
\mathcal Z_4\coloneqq\langle(1+\zeta)^4\rangle_0,\qquad
\mathfrak F(c)\coloneqq\left\langle\frac1{c-\zeta}\right\rangle_0.
\]
\begin{lemma}\label{lem:cauchy-transform}
The moments satisfy
\begin{equation}\label{eq:evenmoments}
\nu_{2r}=\frac{(\frac{1}{2})_r}{(\frac{n}{2}-\frac{3}{2})_r},\qquad \nu_{2r+1}=0,\qquad
\mathcal Z_4=\frac{n(n+2)}{(n-3)(n-1)}.
\end{equation}
The Cauchy transform satisfies
\begin{equation}\label{eq:cauchyODE}
(c^2-1)\mathfrak F'(c)=2(\frac{n}{2}-3)c\mathfrak F(c)-(n-5).
\end{equation}
Its regular coefficients at $\upsilon=0$ are
\begin{equation}\label{eq:cauchyseries}
\mathfrak F(\sqrt{1-\upsilon})
=(\frac{n}{2}-\frac{5}{2})\sum_{k\ge0}
\frac{(\frac{1}{2})_k}{k!(\frac{n}{2}-3-k)}\upsilon^k,
\end{equation}
up to a homogeneous solution of order $\upsilon^{n/2-3}$; if $n/2-3$ is an integer, a logarithmic term can occur at that order.
Only coefficients of order less than $n/2-3$ are asserted and used.
\end{lemma}
\begin{proof}
Symmetry gives the odd moments, and the beta integral gives the even moments in \eqref{eq:evenmoments}.
Expanding $(1+\zeta)^4$ then gives $\mathcal Z_4$.
Integrating the $\zeta$-derivative of $(1-\zeta^2)^{n/2-2}/(c-\zeta)$ over $(-1,1)$ gives \eqref{eq:cauchyODE}; its boundary term vanishes because $n/2-2>0$.
After $c=\sqrt{1-\upsilon}$, equation \eqref{eq:cauchyODE} becomes
\[
\upsilon\partial_\upsilon\mathfrak F-(\frac{n}{2}-3)\mathfrak F
=-(\frac{n}{2}-\frac{5}{2})(1-\upsilon)^{-\frac{1}{2}}.
\]
Matching coefficients proves \eqref{eq:cauchyseries}.
The homogeneous equation has solution $\upsilon^{n/2-3}$, and resonance at an integer exponent accounts for the possible logarithmic term.
\end{proof}

We now use this Cauchy transform to eliminate the radial integrals.
\begin{proof}[Proof of Proposition \ref{prop:summand}]
Fix $0\le i\le j\le3$.
For each nonzero $\kappa_{ij,\ell}$, we form the polynomial
\begin{equation}\label{eq:PiCauchy}
\Pi_\ell(\zeta)\coloneqq
\kappa_{ij,\ell}\left(\frac{1+\zeta}{1-\zeta}\right)
(1+\zeta)^4(1-\zeta)^\ell.
\end{equation}
For $0\le i\le j\le3$, substitution of \eqref{eq:Qclosed} verifies $\deg_t\kappa_{ij,\ell}\le\ell$, so $\Pi_\ell$ is a polynomial.
Lemma \ref{lem:finite-extraction-checks} verifies this polynomiality and the Laurent, pole, and degree data in Table \ref{tab:finite-extraction-checks}.
If $\Pi(\zeta)=\sum_m\pi_m\zeta^m$, division by $c-\zeta$ yields
\[
T_\Pi(c)\coloneqq\Pi(c)\mathfrak F(c)
-\sum_m\pi_m\sum_{r=0}^{m-1}c^{m-1-r}\nu_r
=\left\langle\frac{\Pi(\zeta)}{c-\zeta}\right\rangle_0.
\]
Differentiation therefore yields
\begin{equation}\label{eq:Cauchyreduction}
\langle\kappa_{ij,\ell}R^\ell\rangle_\mu
=\frac{(-1)^{\ell-1}}{(\ell-1)!\mathcal Z_4}
\partial_c^{\ell-1}T_{\Pi_\ell}(c).
\end{equation}
We substitute \eqref{eq:Cauchyreduction} in \eqref{eq:Hfinite} and eliminate derivatives of $\mathfrak F$ with \eqref{eq:cauchyODE}.
The result has the form
\begin{equation}\label{eq:Hcauchyform}
\mathscr H_{ij}(\upsilon)=A_{ij}(\upsilon)\mathfrak F(\sqrt{1-\upsilon})
+\sqrt{1-\upsilon}\,B_{ij}(\upsilon)+c_{ij},
\end{equation}
where the finite Laurent polynomials $A_{ij}$ and $B_{ij}$ are obtained by the preceding substitutions, and
\begin{equation}\label{eq:cijfinite}
c_{ij}\coloneqq-\frac{\langle t^i,At^j\rangle_\mu}{2(n-1)}
+\sum_\ell\langle\kappa_{ij,\ell}\rangle_\mu.
\end{equation}
The exact Laurent bounds in Table \ref{tab:finite-extraction-checks} show that neither polynomial contains a power below $-(i+j)$.
Hence the homogeneous solution described after \eqref{eq:cauchyseries} contributes only at order at least $\upsilon^{n/2-3-i-j}$.
Since $n>2N+16$ and $i+j\le6$,
\[
\frac{n}{2}-3-i-j>N-1.
\]
Thus neither this term nor a logarithmic term at that order contributes to \eqref{eq:normalizedcoefficient}.
This also justifies coefficient extraction at the endpoint of the radial integral.

For each integer $\ell$, we set $a_{ij,\ell}\coloneqq[\upsilon^\ell]A_{ij}(\upsilon)$ and $b_{ij,\ell}\coloneqq[\upsilon^\ell]B_{ij}(\upsilon)$, with zero coefficients outside the Laurent ranges in Table \ref{tab:finite-extraction-checks}.
For an indeterminate $k$, we define
\[
\beta_\ell(k)\coloneqq
\begin{cases}
\displaystyle\frac{(-k)_\ell}{(\frac{1}{2}-k)_\ell} &\text{for } \ell\ge0,\\[2mm]
\displaystyle\frac{(k+\frac{1}{2})_{-\ell}}{(k+1)_{-\ell}} &\text{for } \ell<0.
\end{cases}
\]
For integers $k\ge\max\{\ell,0\}$, $\beta_\ell(k)$ is the ratio of $(1/2)_{k-\ell}/(k-\ell)!$ to $(1/2)_k/k!$.
Hence \eqref{eq:cauchyseries} and \eqref{eq:Hcauchyform} imply
\begin{align}
R_{ij}(k,n)
\coloneqq{}&\sum_\ell a_{ij,\ell}(\frac{n}{2}-\frac{5}{2})
\frac{\beta_\ell(k)}{\frac{n}{2}-3-k+\ell}
+\sum_\ell b_{ij,\ell}
\frac{\beta_\ell(k)}{1-2k+2\ell},\label{eq:Rfinite}\\
[\upsilon^k]\mathscr H_{ij}(\upsilon)
={}&\frac{(\frac{1}{2})_k}{k!}R_{ij}(k,n)
+c_{ij}\mathbf1_{k=0}.\label{eq:Hcoefficients}
\end{align}
For integers $0\le k<\ell$, the numerator $(-k)_\ell$ vanishes, as required for a shifted power series.
Thus the formulas also cover the initial coefficients.

If $g_r\coloneqq[\upsilon^r]\mathcal G_N(\upsilon)$, cancellation of rising factorials in \eqref{eq:GN} yields
\begin{equation}\label{eq:coefficientratio}
\frac{g_{N-1-k}}{g_{N-1}}=
\frac{(1-N)_k(1+N)_k(1)_k}
{(2-\frac{n}{2})_k(1+\frac{n}{2})_k(\frac{3}{2})_k},\qquad 0\le k<N.
\end{equation}
Inserting \eqref{eq:Hcoefficients} in \eqref{eq:normalizedcoefficient} proves \eqref{eq:finite-summand}, with
\begin{equation}\label{eq:rCauchydefinition}
r_{ij}(k,n)=\frac{k+1-\frac{n}{2}}{(1-\frac{n}{2})(2k+1)}R_{ij}(k,n).
\end{equation}
Equations \eqref{eq:Qclosed}, \eqref{eq:kappadefinition}, and \eqref{eq:PiCauchy}--\eqref{eq:rCauchydefinition} determine all ten summands by rational operations.
The last finite check, recorded in Table \ref{tab:finite-extraction-checks}, is that multiplication by $(k+2-n/2)_{i+j+2}$ gives a polynomial of degree $i+j+1$.
This proves the bound in Proposition \ref{prop:summand}.

For the entry $(i,j)=(0,1)$, we write $v=n/2$.
Equations \eqref{eq:Qclosed} and \eqref{eq:kappadefinition} give just three nonzero coefficients:
\[
\kappa_{01,5}=\frac{2(v+2)(1+t)}{5(v-1)},\qquad
\kappa_{01,6}=\frac{(v+2)(t^2-1)}{2(v-1)},\qquad
\kappa_{01,7}=-\frac{12t(1+t)}{7(v-1)}.
\]
Cauchy reduction gives $c_{01}$ and $r_{01}$ displayed in Section \ref{sec:gaussian}; in particular,
\[
r_{01}(k,2v)(k+2-v)_3
=\frac{(v-2)(v+2)(2v-1)}{105(v-1)}
(2k+3)(4k+14v-11).
\]
Thus the right-hand side of \eqref{eq:gaussian-system} is quadratic.
Solving that size-three system and using \eqref{eq:closed-entry} yields
\begin{equation}\label{eq:worked-e01}
\begin{split}
e_{01}={}&-\frac{(v+2)(7v-6)}{35(v-3)(v-1)}
+\frac{v(v+2)(2v-1)(10v-11)}{210(v-1)((v-1)^2-N^2)}\\
&+\frac{4v(v+2)(2v-1)}{105((v-2)^2-N^2)}
-\frac{3v(v+2)(2v-3)(2v-1)}{70(v-3)((v-3)^2-N^2)}.
\end{split}
\end{equation}
\end{proof}

\section{The determinant coefficients}\label{app:crosspoly}
All coefficients in this appendix are determined before a vector $a=(\tau,a_1,a_2,a_3)^{\mathsf T}$ is chosen.
The ambient definition of $Q_{2N}$ determines the Hessian polynomial $\mathscr P_N^{(n)}$; its pairings on the four tensors $H_0,H_1,H_2,H_3$ determine $E$ and the scale-derivative matrix $S_N(n)$ through \eqref{eq:Ematrix}.
The expansions \eqref{eq:matrix-asymptotic}--\eqref{eq:determinant-coefficients} then determine $M_0,M_1,G_0,G_1$.
The profile coefficients only select a vector on which the already determined quadratic form $a^{\mathsf T}S_N(n)a$ is evaluated.
In particular, both fixed tails in \eqref{eq:fixed-tails} are tested against the same matrix and do not affect any coefficient listed here.

The displayed rational and integer coefficients use the specified basis, normalization, and common denominator.
A fixed invertible change of basis acts on $S_N(n)$ by congruence and multiplies its determinant by a nonzero square.
Consequently the crossing dimensions and the constants $a_*,c_*$ are properties of the scale variation of the restricted Hessian, independent of such coordinate choices.
Their dependence on the chosen four-dimensional space of polynomial symmetric two-tensors remains part of their definition.

We use the Pochhammer identity $\prod_{\ell=1}^r(n-2N-2\ell)=2^r(n/2-N-r)_r$ to write the denominators below.
The leading matrix-valued function in \eqref{eq:matrix-asymptotic} has entries
\begin{align}\label{eq:M0}
M_{0,00}(n-2N)&=\frac{32}{15\cdot2^2(\frac{n}{2}-N-2)_2},\\
M_{0,01}(n-2N)&=\frac{16(7(n-2N)-24)}{35\cdot2^3(\frac{n}{2}-N-3)_3},\notag\\
M_{0,02}(n-2N)&=\frac{64(7(n-2N)^2-62(n-2N)+112)}{105\cdot2^4(\frac{n}{2}-N-4)_4},\notag\\
M_{0,03}(n-2N)&=\frac{16\left(77(n-2N)^3-1254(n-2N)^2{}+5896(n-2N)-7360\right)}{231\cdot2^5(\frac{n}{2}-N-5)_5},\notag\\
M_{0,11}(n-2N)&=\frac{64(n-2N)(7(n-2N)-44)}{105\cdot2^4(\frac{n}{2}-N-4)_4},\notag\\
M_{0,12}(n-2N)&=\frac{16\left(77(n-2N)^3-858(n-2N)^2{}+880(n-2N)+6400\right)}{231\cdot2^5(\frac{n}{2}-N-5)_5},\notag\\
M_{0,13}(n-2N)&=\frac{32\left(\begin{gathered}1001(n-2N)^4-18018(n-2N)^3\\{}+59488(n-2N)^2+292448(n-2N)-1121280\end{gathered}\right)}
{5005\cdot2^6(\frac{n}{2}-N-6)_6},\notag\\
M_{0,22}(n-2N)&=\frac{32\left(\begin{gathered}3003(n-2N)^4-46332(n-2N)^3\\{}+52052(n-2N)^2+1263600(n-2N)-2563200\end{gathered}\right)}
{15015\cdot2^6(\frac{n}{2}-N-6)_6},\notag\\
M_{0,23}(n-2N)&=\frac{16\left(\begin{gathered}1001(n-2N)^5-21736(n-2N)^4\\{}+55484(n-2N)^3+1203904(n-2N)^2\\{}-6221632(n-2N)+3064320\end{gathered}\right)}
{2145\cdot2^7(\frac{n}{2}-N-7)_7},\notag\\
M_{0,33}(n-2N)&=\frac{128\left(\begin{gathered}17017(n-2N)^6-466752(n-2N)^5\\{}+1409980(n-2N)^4+51095200(n-2N)^3\\{}-407661632(n-2N)^2+427000832(n-2N)\\{}+1451950080\end{gathered}\right)}
{255255\cdot2^8(\frac{n}{2}-N-8)_8}.\notag
\end{align}
The remaining entries are determined by $M_{0,ij}=M_{0,ji}$.

The leading determinant polynomial is
\begin{equation}\label{eq:G0}
\begin{aligned}
G_0(n-2N)\coloneqq{}&26794508741(n-2N)^{14}-5999186112920(n-2N)^{13}\\
&+506809345021408(n-2N)^{12}-23932779365064856(n-2N)^{11}\\
&+735249277022163332(n-2N)^{10}\\
&-15860702168480103792(n-2N)^9\\
&+250690571756780680096(n-2N)^8\\
&-2972645826389547138688(n-2N)^7\\
&+26730031012975840494400(n-2N)^6\\
&-182244042721887163732224(n-2N)^5\\
&+931191022949146422606848(n-2N)^4\\
&-3470331005064168386539520(n-2N)^3\\
&+8946493873369876360396800(n-2N)^2\\
&-14324027140196950966272000(n-2N)\\
&+10776521802741362196480000.
\end{aligned}
\end{equation}
The first correction polynomial is
\begin{equation}\label{eq:G1}
\begin{aligned}
G_1(n-2N)\coloneqq{}&241150578669(n-2N)^{15}-53590061423899(n-2N)^{14}\\
&+4407318986030532(n-2N)^{13}-200391992731357856(n-2N)^{12}\\
&+5863842030770392812(n-2N)^{11}\\
&-118889284143889031916(n-2N)^{10}\\
&+1733931186168043500352(n-2N)^9\\
&-18451377129657163696576(n-2N)^8\\
&+142131831056651623825792(n-2N)^7\\
&-758711188116215386120256(n-2N)^6\\
&+2413518937445144513746688(n-2N)^5\\
&-1028217016766028118428672(n-2N)^4\\
&-30799572236242831569694720(n-2N)^3\\
&+154489901430717390761164800(n-2N)^2\\
&-350635730795977337044992000(n-2N)\\
&+331932342914338801582080000.
\end{aligned}
\end{equation}

\section{Exact arithmetic for the sign and crossing computations}
\label{app:certificates}
The finite computations in this appendix were carried out by computer using exact rational arithmetic.
We specify their inputs, operations, and results here; no external coefficient tables or source files are needed to reproduce them.
Floating-point arithmetic is used only for the decimal evaluations in Theorem \ref{thm:cross} and the numerical illustration in Figure \ref{fig:dimension-crossing}.

\begin{lemma}\label{lem:finite-extraction-checks}
For $0\le i\le j\le3$, every expression $\Pi_\ell$ in \eqref{eq:PiCauchy} is a polynomial.
The functions $A_{ij},B_{ij}$ in \eqref{eq:Hcauchyform} are Laurent polynomials, and their lowest powers, together with the pole and degree data of $r_{ij}$ defined by \eqref{eq:rCauchydefinition}, are as in Table \ref{tab:finite-extraction-checks}.
\end{lemma}
\begin{proof}
To verify these identities and reconstruct the ten normalized entries, work over $\mathbb Q(v,X)$, with $n=2v$, as follows.
\begin{enumerate}
\item Form $\Phi_j$ for $0\le j\le3$ using \eqref{eq:Qclosed}, and
integrate their products as in \eqref{eq:kappadefinition} to obtain $\kappa_{ij,\ell}$ for $0\le i\le j\le3$.
\item Apply \eqref{eq:PiCauchy}--\eqref{eq:Cauchyreduction}, eliminating
derivatives of $\mathfrak F$ by \eqref{eq:cauchyODE}.
Extract the finite Laurent polynomials in \eqref{eq:Hcauchyform}; then \eqref{eq:cijfinite}, \eqref{eq:Rfinite}, and \eqref{eq:rCauchydefinition} give $c_{ij}$ and $r_{ij}$.
\item Clear denominators to verify the degree bound
\eqref{eq:summand-degree} for these expressions.
Put $m=i+j$, replace $N^2$ by $X$ in \eqref{eq:gaussian-system}, and solve by matching coefficients of $1,k,\ldots,k^{m+1}$.
Equation \eqref{eq:closed-entry} gives $e_{ij}$.
Set $\widehat S_{ij}=-2(i+j+2)e_{ij}$ and use symmetry for $i>j$.
\end{enumerate}
All integrations in this procedure are integrations of polynomials or the moments \eqref{eq:moments} and \eqref{eq:evenmoments}.
Every linear system has size at most eight.
Clearing denominators verifies the ambient identity \eqref{eq:QpolynomialPDE}, the absence of powers below $-(i+j)$ in \eqref{eq:Hcauchyform}, and the degree bound used in step 3.
Substitution verifies each solved system before its entries are used.

Table \ref{tab:finite-extraction-checks} records the finite identities behind Proposition \ref{prop:summand}.
The second column is the least power of $\upsilon$ present in each of $A_{ij}$ and $B_{ij}$.
The third column lists the shifts in the monic denominator $\prod_\ell(k-v+\ell)$ of $r_{ij}(k,2v)$ over $\mathbb Q(v)$; every listed factor is simple.
The degree refers to the polynomial in \eqref{eq:summand-degree}.
In the same ten checks, direct substitution of \eqref{eq:Qclosed} gives $\deg_t\kappa_{ij,\ell}\le\ell$, which verifies the polynomiality in \eqref{eq:PiCauchy}.
These are identities in the rational function fields, not tests at sampled values of $n$ or $N$.
\begin{table}[ht]
\centering
\begin{tabular}{ccccc}\toprule
$(i,j)$ & Lowest Laurent power & Pole shifts $\ell$ & Degree in $k$ & System size\\\midrule
$(0,0)$ & $0$  & $2,3$             & $1$ & $2$\\
$(0,1)$ & $-1$ & $2,\ldots,4$      & $2$ & $3$\\
$(0,2)$ & $-2$ & $2,\ldots,5$      & $3$ & $4$\\
$(0,3)$ & $-3$ & $2,\ldots,6$      & $4$ & $5$\\
$(1,1)$ & $-2$ & $2,\ldots,5$      & $3$ & $4$\\
$(1,2)$ & $-3$ & $2,\ldots,6$      & $4$ & $5$\\
$(1,3)$ & $-4$ & $2,\ldots,7$      & $5$ & $6$\\
$(2,2)$ & $-4$ & $2,\ldots,7$      & $5$ & $6$\\
$(2,3)$ & $-5$ & $2,\ldots,8$      & $6$ & $7$\\
$(3,3)$ & $-6$ & $2,\ldots,9$      & $7$ & $8$\\\bottomrule
\end{tabular}
\caption{Exact pole and degree checks for the ten rational summands.}
\label{tab:finite-extraction-checks}
\end{table}
\end{proof}

For either tail $\mathbf c$ in \eqref{eq:fixed-tails}, we form \eqref{eq:tail-discriminant} from the entries reconstructed in Lemma \ref{lem:finite-extraction-checks}.
We define
\begin{equation}\label{eq:tail-denominator}
\begin{split}
\mathfrak D(v,X)\coloneqq{}&294819525(v-8)(v-7)(v-6)
(v-5)^2(v-4)^2(v-3)^4(v-2)^2(v-1)^2\\
&\quad\cdot\prod_{\ell=1}^{5}\bigl((v-\ell)^2-X\bigr)^2
\prod_{\ell=6}^{8}\bigl((v-\ell)^2-X\bigr),
\end{split}
\end{equation}
and define $\mathfrak A_{\mathbf c}(v,X)$ by multiplying $\mathfrak d_{\mathbf c}(N,n)$ by $\mathfrak D(v,X)$, with $v=n/2$ and $X=N^2$.
This fixes the numerator normalization; no arbitrary scalar multiple is taken.

\begin{proposition}\label{prop:tail-certificates}
For both tails, $\mathfrak A_{\mathbf c}$ belongs to $\mathbb Z[v,X]$ and has degrees $41$ in $v$ and $13$ in $X$.
The denominator $\mathfrak D(n/2,N^2)$ is positive for $N\ge1$ and $n-2N\ge17$.
For $4\le N\le17$, let $n_{\min}(N)$ be the first dimension in \eqref{eq:dimensionrange}.
The coefficients of
\[
\mathfrak A_{(100,-20,1)}(\frac{n_{\min}(N)}{2}+u,N^2),\qquad
2^{39}\mathfrak A_{(150,-29,1)}(\frac{53}{2}+u+\frac{\eta}{2},(18+u)^2)
\]
have the following supports and minimum nonzero values:
\begin{center}
\begin{tabular}{llll}\toprule
Polynomial & Nonzero indices & Number & Smallest coefficient\\\midrule
Each low-order polynomial & $0\le r\le41$ & $42$ & $4680554778900$\\
High-order polynomial & $0\le r\le32,\ 0\le s\le41-r$ & $858$ & $2660746213125$\\\bottomrule
\end{tabular}
\end{center}
The indices in the first and second rows refer to the monomials $u^r$ and $u^r\eta^s$, respectively.
The first row applies separately to each $N=4,\ldots,17$.
The second polynomial has integer coefficients.
Coefficients outside the stated index sets vanish; in particular, both constant coefficients are positive.
\end{proposition}
\begin{proof}
Exact cancellation after the three reconstruction steps in the proof of Lemma \ref{lem:finite-extraction-checks} gives the claimed integer numerators and their degrees.
Every factor of \eqref{eq:tail-denominator} is positive in the stated range, since
\[
(v-\ell)^2-N^2
=\left(\frac{n-2N}{2}-\ell\right)
\left(2N+\frac{n-2N}{2}-\ell\right)>0
\qquad(1\le\ell\le8).
\]
To specify the remaining finite checks, write $\mathfrak A_{\mathbf c}(v,X)=\sum_{p,q}\alpha_{pq}^{\mathbf c}v^pX^q$.
For $\mathbf c=(100,-20,1)$, the binomial theorem gives
\begin{equation}\label{eq:low-tail-coefficient-check}
[u^r]\mathfrak A_{\mathbf c}(\frac{n_{\min}(N)}{2}+u,N^2)
=\sum_{p\ge r,q}\alpha_{pq}^{\mathbf c}\binom pr
\left(\frac{n_{\min}(N)}2\right)^{p-r}N^{2q}.
\end{equation}
For $\mathbf c=(150,-29,1)$, the coefficient of $u^r\eta^s$ in the second polynomial is
\begin{equation}\label{eq:high-tail-coefficient-check}
2^{39-s}\sum_{p\ge s,q}\alpha_{pq}^{\mathbf c}\binom ps
\sum_{j=\max\{0,r-2q\}}^{\min\{r,p-s\}}
\binom{p-s}{j}\binom{2q}{r-j}
\left(\frac{53}2\right)^{p-s-j}18^{2q-r+j}.
\end{equation}
Exact rational evaluation of these finite sums gives the displayed supports and minima.
The first check uses the unscaled polynomial for each of the fourteen orders; the second includes the factor $2^{39}$.
Thus the signs hold on the entire parameter ranges, not only at sampled dimensions.
\end{proof}

\begin{lemma}\label{lem:common-tail-obstruction}
The following $3\times3$ matrix is negative definite:
\[
\left(\widehat S_{0i}\widehat S_{0j}
-\widehat S_{00}\widehat S_{ij}\right)_{i,j=1}^3
\bigg|_{(N,n)=(4,28)}
+\left(\widehat S_{0i}\widehat S_{0j}
-\widehat S_{00}\widehat S_{ij}\right)_{i,j=1}^3
\bigg|_{(N,n)=(18,53)}.
\]
\end{lemma}
\begin{proof}
We evaluate the rational entries at the two displayed pairs.
Unpivoted $LDL^{\mathsf T}$ elimination gives diagonal pivots $d_1,d_2,d_3$ satisfying the exact rational bounds
\[
-112<d_1<-111,\qquad -194<d_2<-193,\qquad -1958<d_3<-1957.
\]
Each bound is checked by clearing a positive denominator.
Hence the matrix is negative definite.
The same entries and inequalities were independently checked from the ambient recurrences \eqref{eq:qrecurrence}, \eqref{eq:thetarecurrence}, and \eqref{eq:WronskianHessian}, using \eqref{eq:Amaster} and \eqref{eq:moments}.
This check uses no closed-entry input.
\end{proof}

\medskip
Finally, we locate the determinant crossing and determine its direction.
\begin{lemma}\label{lem:crossing-signs}
The polynomial $G_0$ in \eqref{eq:G0} has a unique root $a_*$ in $(16,17)$, and this root is simple.
A rational isolating interval is
\[
16.20187147804706198873008171619853
<a_*
<16.20187147804706198873008171619854.
\]
On this interval, the first three leading principal minors of $M_0$ are positive, $G_0'<0$, and $G_1>0$.
Consequently $c_*=-G_1(a_*)/G_0'(a_*)>0$.
\end{lemma}
\begin{proof}
The rational entries determine $M_0,M_1$ by \eqref{eq:matrix-asymptotic}; their determinant and first correction give \eqref{eq:G0}--\eqref{eq:G1}.
To reproduce the root counts for a polynomial $P$, start with $P,P'$ and repeatedly take the negative polynomial remainder.
At an endpoint, discard zeros and count sign changes in the resulting Sturm sequence.
The sequence of $G_0$ has variations $7$ and $6$ at $16$ and $17$, respectively, and its last nonzero remainder is constant.
Thus its unique root in $(16,17)$ is simple.

For each of the first three leading principal minors of $M_0$ in \eqref{eq:M0}, reduce its quotient to coprime polynomials.
Sturm counts give no root of either polynomial on the isolating interval, and rational endpoint evaluation gives a positive quotient.
The same procedure gives no zero of $G_0'$ or $G_1$ there, with $G_0'<0$ and $G_1>0$.
All endpoints and sign comparisons are rational; the decimal values in the introduction are outputs, not inputs.
\end{proof}

\section{Obstruction tensors and factorization at \texorpdfstring{$n=2N$}{n=2N}}
\label{app:obstruction-comparison}

This appendix compares the linearizations of the ambient obstruction tensor and Graham's extended obstruction tensors in TT directions.
The scalar ambient equation \eqref{eq:TTambient} computes both, but the product formula comes from the residue of $\boldsymbol\Omega^{(N-1)}$ at $n=2N$.
This extended tensor is not defined as a natural tensor on general metrics at that dimension, and its linearization for $n>2N$ need not have the same real linear factors.
Throughout this appendix, we assume $\Ric(g_E)=(n-1)g_E$, write $a=1+\rho/2$, and take $h$ to be a TT eigentensor satisfying $(\Delta_L+2(n-1))h=bh$.

We write $\boldsymbol\Omega^{(r)}$ for Graham's $r$-th extended obstruction tensor, reserving the unadorned $\Omega$ for the stereographic conformal factor.
In the ambient normal form of \eqref{eq:ambientnormal}, its definition is \cite[Definition 2.2]{GrahamExtended}
\begin{equation}\label{eq:extended-obstruction-definition}
\boldsymbol\Omega^{(r)}_{ij}
=\left.(\widetilde\nabla_{\partial_\rho}^{\,r-1}
\widetilde{\mathrm{Rm}})_{\rho i j\rho}
\right|_{\mathsf t=1,\rho=0},\qquad n>2r+2.
\end{equation}
The restriction on $n$ suffices for the comparison below.
For $r=1$, Graham's normalization gives $\boldsymbol\Omega^{(1)}=(4-n)^{-1}B$, where $B$ is the Bach tensor \cite[(1.5)]{GrahamExtended}.
Accordingly, \eqref{eq:extended-first} agrees with the Bach-tensor linearization of Lin--Yuan \cite[Proposition 3.4]{LinYuanII}, after identifying their Einstein operator with $-\cA$.

\begin{proposition}\label{prop:extended-TT-linearization}
For $r\ge1$, $n>2r+2$, and a TT eigentensor $h$ satisfying $(\Delta_L+2(n-1))h=bh$, the linearized extended obstruction tensor is
\begin{equation}\label{eq:extended-obstruction-linearization}
D\boldsymbol\Omega^{(r)}_{g_E}[h]
=\frac12\left.
\partial_\rho^{r-1}(\psi_b''+a^{-1}\psi_b')
\right|_{\rho=0}h.
\end{equation}
In the first three orders it is given by \eqref{eq:extended-first}--\eqref{eq:extended-third}.
For $r=3$ and every integer $n\ge9$, its quadratic factor in $b$ has no real roots.
\end{proposition}
\begin{proof}
The curvature convention here is fixed by the ambient identity \cite[(2.12)]{GrahamExtended}
\[
\left.\widetilde{\mathrm{Rm}}_{\rho i j\rho}\right|_{\mathsf t=1}
=\frac12\left(g''_{\rho,ij}
-\frac12g_\rho^{kl}g'_{\rho,ik}g'_{\rho,jl}\right).
\]
Lemma \ref{lem:Einstein-ambient} gives $g_\rho=a^2g_E$, so these curvature components and all their repeated normal derivatives vanish.
Along $g_E\exp(sg_E^{-1}h)$, the first ambient variation is $a^2\psi_bh$.
Differentiating the displayed curvature identity gives
\[
D\left(\left.\widetilde{\mathrm{Rm}}_{\rho i j\rho}
\right|_{\mathsf t=1}\right)[h]
=\frac12(a^2\psi_b''+a\psi_b')h_{ij}.
\]
By \cite[(2.14)]{GrahamExtended}, with $\widetilde\Gamma^k_{\rho j}=(2a)^{-1}\delta^k_j$ and $\widetilde\nabla_{\partial_\rho}\partial_\rho=0$, each further normal covariant derivative acts on this coefficient by $\partial_\rho-a^{-1}$.
The variations of the connection multiply vanishing background components.
The identity $\partial_\rho-a^{-1}=a^2\partial_\rho a^{-2}$ therefore proves \eqref{eq:extended-obstruction-linearization}.
Thus the recurrence \eqref{eq:qrecurrence} determines every required linearized extended tensor without its nonlinear curvature formula.
For $\boldsymbol\Omega^{(r)}$, the same recurrence is continued through $\psi_{b,r+1}$; its denominators are nonzero when $n>2r+2$.

For the first three orders this yields
\begin{align}
D\boldsymbol\Omega^{(1)}_{g_E}[h]
&=\frac{b(b-n+2)}{2(n-2)(n-4)}h,
\label{eq:extended-first}\\
D\boldsymbol\Omega^{(2)}_{g_E}[h]
&=-\frac{b(b-n+2)(b-6)}{2(n-2)(n-4)(n-6)}h,
\label{eq:extended-second}\\
D\boldsymbol\Omega^{(3)}_{g_E}[h]
&=\frac{b(b-n+2)
\bigl((b-10)(b-12)+(n-8)(n+6)\bigr)}
{2(n-2)(n-4)(n-6)(n-8)}h.
\label{eq:extended-third}
\end{align}
The quadratic factor in \eqref{eq:extended-third} has discriminant $4\bigl(50-(n-1)^2\bigr)$, which is negative at every integer $n\ge9$.
It therefore does not split into real linear factors there.
At $n=8$, however, the term $(n-8)(n+6)$ contributes nothing to the residue, and the remaining factors are $b(b-6)(b-10)(b-12)$.
\end{proof}

We now take the residue at $n=2N$ to obtain the factorization.
Here the \emph{residue} is the coefficient of $(n-2N)^{-1}$ in the natural rational tensor formula, before specializing the dimension.

\begin{proposition}\label{prop:critical-obstruction-product}
For every $N\ge2$, the residue at $n=2N$ of the linearization of $\boldsymbol\Omega^{(N-1)}$ in TT directions is
\begin{equation}\label{eq:critical-extended-residue}
\operatorname*{Res}_{n=2N}
D\boldsymbol\Omega^{(N-1)}_{g_E}[h]
=\frac{(-1)^N}{2^N(N-1)!}
\prod_{j=0}^{N-1}[b-j(2N-j-1)]h.
\end{equation}
Let $\mathcal O^{(2N)}$ denote the obstruction tensor in Graham's normalization.
Its linearization is
\begin{equation}\label{eq:critical-obstruction-linearization}
D\mathcal O^{(2N)}_{g_E}[h]
=-\frac1{4(N-1)}
\prod_{j=0}^{N-1}[b-j(2N-j-1)]h.
\end{equation}
The variational identity recovers \eqref{eq:Matsumoto-critical}.
For $n>2N$, the operator representing the fixed-volume Hessian of total $Q_{2N}$-curvature on TT tensors need not be a constant multiple of $D\boldsymbol\Omega^{(N-1)}$.
\end{proposition}
\begin{proof}
We regard $n$ as a formal parameter in the rational coefficient formulas and set
\[
\sigma\coloneqq\frac{\rho}{\rho+2},\qquad
\Psi_b(\sigma)\coloneqq
\psi_b\left(\frac{2\sigma}{1-\sigma}\right).
\]
The scalar version of \eqref{eq:TTambient} becomes
\begin{equation}\label{eq:extended-hypergeometric-ode}
\sigma(1-\sigma)\Psi_b''
+\left((n-2)\sigma-\frac{n-2}{2}\right)\Psi_b'
-b\Psi_b=0.
\end{equation}
Primes in this display denote $\sigma$-derivatives.
If $\Psi_b=\sum_{m\ge0}\mathsf c_m\sigma^m$, its coefficient equation is
\[
(m+1)(m+1-\frac{n}{2})\mathsf c_{m+1}
=\bigl[b-m(n-m-1)\bigr]\mathsf c_m,
\qquad \mathsf c_0=1.
\]
Consequently, as a formal rational identity,
\begin{equation}\label{eq:extended-factorial-coefficients}
\mathsf c_m=
\frac{\displaystyle\prod_{j=0}^{m-1}
[b-j(n-j-1)]}
{m!(1-\frac{n}{2})_m}.
\end{equation}
Only finite jets are being asserted; no convergence or continuation of the ambient metric through order $N$ at $n=2N$ is needed.

Fix $N\ge2$.
At $n=2N$, all $\mathsf c_m$ with $m<N$ are regular, whereas $(1-n/2)_N$ has a simple zero.
Since $\sigma=\rho/2+O(\rho^2)$, only $\mathsf c_N\sigma^N$ contributes to the residue of $\psi_{b,N}=\psi_b^{(N)}(0)$.
We obtain
\begin{equation}\label{eq:critical-profile-residue}
\operatorname*{Res}_{n=2N}\psi_{b,N}
=\frac{(-1)^N}{2^{N-1}(N-1)!}
\prod_{j=0}^{N-1}[b-j(2N-j-1)].
\end{equation}
All lower derivatives in \eqref{eq:extended-obstruction-linearization}, with $r=N-1$, are regular at this dimension.
Thus \eqref{eq:critical-extended-residue} follows.

Graham's residue identity \cite[Proposition 2.8]{GrahamExtended} is
\[
\operatorname*{Res}_{n=2N}\boldsymbol\Omega^{(N-1)}
=\frac{(-1)^{N-1}}{2^{N-2}(N-2)!}\mathcal O^{(2N)},
\]
in the normalization for which $\mathcal O^{(4)}$ is the Bach tensor.
Metric variation commutes with the residue of this rational natural tensor expression.
Thus \eqref{eq:critical-extended-residue} gives \eqref{eq:critical-obstruction-linearization}, also obtained directly by combining \cite[Theorem 3.4]{Matsumoto} with \cite[Theorem 4.2]{Matsumoto}.
Since the Einstein background is obstruction-flat, \cite[(0-4)]{Matsumoto} gives
\[
D^2\mathcal Q|_{g_E}(h,k)
=(-1)^N(N-1)
\langle h,D\mathcal O^{(2N)}_{g_E}[k]\rangle_{L^2(g_E)},
\]
which recovers \eqref{eq:Matsumoto-critical}.

The product in \eqref{eq:critical-obstruction-linearization} comes from the resonant coefficient at $n=2N$.
It describes the residue of the extended tensor, rather than its value at the pole.
Formula \eqref{eq:extended-third} shows concretely why the same factorization cannot simply be imposed on the extended tensor away from that dimension.
Moreover, for $n>2N$ the operator representing the fixed-volume Hessian of total $Q_{2N}$-curvature on TT tensors is not in general a constant multiple of $D\boldsymbol\Omega^{(N-1)}$.
Already at $N=2$, \eqref{eq:lowP} gives its nonzero polynomial root as
\[
\frac{(n^2-4)(n-2)^2}{8(n-1)},
\]
whereas \eqref{eq:extended-first} has nonzero root $n-2$.
These roots coincide at $n=4$ and differ for every $n>4$.
The determinant equation and scalar coefficient extraction in Section \ref{sec:ambient} therefore supply information not obtained by substituting a dimension $n>2N$ into the obstruction product for $n=2N$.
\end{proof}

\section{Second variation of total renormalized volume coefficients}
\label{app:volume-hessians}

The determinant equation in Lemma \ref{lem:ThetaODE} supplies a common input for two families of metric functionals.
The fixed-volume Hessian of total $Q_{2N}$-curvature is determined by the weighted derivative of $\Theta_{h,k}$ in \eqref{eq:WronskianHessian}; the Hessians of total renormalized volume coefficients are determined by the coefficients of $a^n\Theta_{h,k}$ itself.
We keep the Einstein normalization $\Ric(g_E)=(n-1)g_E$, the operator $\cA$ in \eqref{eq:Aconvention}, and the fixed-volume exponential path \eqref{eq:fixed-volume-path}.
In this appendix, we compute second variations as mixed derivatives along \eqref{eq:fixed-volume-path}, as in Theorem \ref{thm:ambientHessian}.

\subsection{Coefficient extraction and negativity on TT tensors}

For an integer $N\ge1$ with $n\ge\max\{3,2N\}$, we define the ambient volume coefficients by
\begin{equation}\label{eq:volume-coefficients}
v(g,\rho)\coloneqq\left(\frac{\det g_\rho}{\det g}\right)^{\frac{1}{2}},
\qquad v_j(g)\coloneqq[\rho^j]v(g,\rho),\qquad 0\le j\le N.
\end{equation}
Here $v_0=1$ and $g_\rho$ is the normal-form family in \eqref{eq:ambientnormal}.
At $n=2N$, its order-$N$ trace is determined even though the trace-free part at that order is not; hence $v_N$ is still a natural scalar invariant \cite{FG,GrahamExtended}.
Our convention is $v_j=(-2)^jv^{(2j)}$, where $v^{(2j)}$ is the coefficient of $r^{2j}$ in the Poincar\'e volume density and $\rho=-r^2/2$ \cite[(3.2)--(3.3)]{ChangFangGraham}.
In particular, $v_1=R_g/(2(n-1))$, and for locally conformally flat metrics $v_j=\sigma_j(g^{-1}\mathsf P_g)$, where $\sigma_j$ is the $j$th elementary symmetric polynomial of the eigenvalues.
At the Einstein reference, $v(g_E,\rho)=a^n$ and $v_j(g_E)=2^{-j}\binom nj$.
We define
\begin{equation}\label{eq:total-volume-functionals}
\mathcal V_j(g)\coloneqq\int_Mv_j(g)\,dv_g,
\qquad
\widehat{\mathcal V}_j(g)\coloneqq
\Vol(g)^{-\frac{n-2j}{n}}\mathcal V_j(g).
\end{equation}
These are integrals of local volume coefficients; the finite constant term in the renormalized volume of a Poincar\'e--Einstein filling is a different functional, studied in \cite{GuillarmouMoroianuSchlenker}.
Chang--Fang--Graham \cite[Theorem 3.1]{ChangFangGraham} compute the second variation of $\mathcal V_j$ restricted to metrics of fixed volume in a conformal class; below we compute it in TT directions.

\begin{theorem}\label{prop:volume-TT-hessian}
Let $N\ge1$, $n\ge\max\{3,2N\}$, and let $(M^n,g_E)$ be a smooth closed Einstein manifold with $\Ric(g_E)=(n-1)g_E$.
For smooth TT tensors $h,k$, the first variation $D_gv_N|_{g_E}[h]$ vanishes pointwise, and
\begin{equation}\label{eq:volume-TT-hessian}
D^2\mathcal V_N|_{g_E}(h,k)
=\frac12[\rho^N]\bigl(a^n\Theta_{h,k}\bigr)
=\left\langle h,\mathscr R_N^{(n)}(-\cA)k\right\rangle_{L^2(g_E)},
\end{equation}
where $\mathscr R_N^{(n)}$ is a polynomial of degree $N$ with coefficients rational in $n$.
For the normalized functional, the right-hand side is multiplied by $\Vol(g_E)^{-(n-2N)/n}$.
Both forms extend continuously to TT tensors in $H^N(M;\operatorname{Sym}^2T^*M)$, defined with respect to $g_E$.
If $\Theta_b$ denotes the formal solution in Lemma \ref{lem:ThetaODE} for a unit TT eigentensor with $-\cA h=bh$, then
\begin{equation}\label{eq:volume-polynomial}
\mathscr R_j^{(n)}(b)=\frac12[\rho^j]\bigl(a^n\Theta_b\bigr),
\qquad 1\le j\le N,
\end{equation}
and the generating identity through degree $N$ is
\begin{equation}\label{eq:volume-generating}
\begin{split}
2\sum_{j=1}^N\mathscr R_j^{(n)}(b)\rho^j
&=a(\rho)^{n-1}\left[
\frac{b\rho}{2(n-1)}
-\int_0^\rho(\rho-t)a(t)(\psi_b'(t))^2\,dt\right]
+O(\rho^{N+1}).
\end{split}
\end{equation}
\end{theorem}
\begin{proof}
Along \eqref{eq:fixed-volume-path}, the volume form of $g_0$ is fixed pointwise.
We write $\ell_h(\rho)$ for the first variation in direction $h$ of $\log(\det g_\rho/(a^{2n}\det g_E))$.
The first variation of the normal Ricci equation \eqref{eq:FGnormal} gives
\[
\ell_h''+a^{-1}\ell_h'=0,\qquad
\ell_h(0)=0,\qquad
\ell_h'(0)=\frac{D R_{g_E}[h]}{n-1}=0.
\]
The last equality uses the TT condition and the Einstein property.
Thus $\ell_h$ vanishes through degree $N$, including at $n=2N$, without requiring the undetermined trace-free order-$N$ ambient coefficient.
Taking coefficients proves $D_gv_N|_{g_E}[h]=0$.
Differentiating the square root of the determinant therefore gives
\[
\left.\partial_{s_1}\partial_{s_2}
\int_M v(g_0(s_1,s_2),\rho)\,dv_{g_E}\right|_{0,0}
=\frac12a^n\Theta_{h,k}(\rho).
\]
Coefficient extraction proves the first equality in \eqref{eq:volume-TT-hessian}.
The recurrences \eqref{eq:qrecurrence} and \eqref{eq:thetarecurrence} give \eqref{eq:volume-polynomial} and the degree bound.
For $N\ge2$, the coefficient of $b^N$ is
\[
\frac{(-1)^{N+1}}{2^{N+1}N!}
\sum_{r=1}^{N-1}\binom{N-2}{r-1}
\frac{1}
{(\frac{n}{2}-r)_r\,(\frac{n}{2}-N+r)_{N-r}},
\]
obtained from the quadratic source for $\theta_N$ in \eqref{eq:thetarecurrence}.
It is nonzero in the stated dimension range, so the degree is exactly $N$; for $N=1$, \eqref{eq:volume-low-one} gives the assertion.
Self-adjointness of $\cA$, followed by polarization, gives the operator formula as in the proof of Theorem \ref{thm:ambientHessian}.
Integration by parts and density of smooth TT tensors give the $H^N$ extension.
Since the volume is constant along the path, normalization contributes only the stated positive factor.

For \eqref{eq:volume-generating}, rewrite \eqref{eq:ThetaODE} as
\[
(a\Theta_b)''=-a(\psi_b')^2,
\qquad (a\Theta_b)(0)=0,
\qquad (a\Theta_b)'(0)=\frac b{2(n-1)}.
\]
Two integrations and multiplication by $a^{n-1}$ prove the identity.
For $N\ge2$, only $\psi_b$ through degree $N-1$ is needed; the recurrence denominators remain nonzero at $n=2N$.
\end{proof}

\begin{corollary}\label{cor:volume-TT-sign}
For $n\ge\max\{3,2N\}$,
\begin{equation}\label{eq:volume-alternating}
\mathscr R_N^{(n)}(b)
=\sum_{j=1}^N(-1)^{j+1}d_{N,j}(n)b^j,
\qquad d_{N,j}(n)>0.
\end{equation}
The second variations of $\mathcal V_N$ and $\widehat{\mathcal V}_N$ are strictly negative on every nonzero TT tensor on the round sphere.
The first three polynomials, in their respective dimension ranges, are
\begin{align}
\mathscr R_1^{(n)}(b)&=\frac b{4(n-1)},
\label{eq:volume-low-one}\\
\mathscr R_2^{(n)}(b)&=\frac b8-\frac{b^2}{4(n-2)^2},
\label{eq:volume-low-two}\\
\mathscr R_3^{(n)}(b)&=\frac{n-2}{32}b
-\frac{3n-8}{24(n-2)(n-4)}b^2
+\frac{b^3}{6(n-2)^2(n-4)}.
\label{eq:volume-low-three}
\end{align}
\end{corollary}
\begin{proof}
Put $b=-x$.
For $1\le m\le N-1$, the polynomial $\psi_{-x,m}$ has strictly positive coefficients of $x,x^2,\ldots,x^m$.
Indeed, $\psi_{-x,1}=x/(n-2)$, and \eqref{eq:qrecurrence} gives this assertion inductively: after multiplication by the positive factor $n-2m-2$, its right-hand side becomes
\[
(2m^2+x)\psi_{-x,m}
+\frac14m(m-1)(2m+n-2)\psi_{-x,m-1}.
\]
In \eqref{eq:volume-generating}, the integral term consequently has nonnegative coefficients in $\rho$ and $x$ through the required order.
The factors $a(t)$ and $a(\rho)^{n-1}$ preserve this property, and integration sends $t^m$ to $\rho^{m+2}/((m+1)(m+2))$.
For $N\ge2$, the contribution of $\psi_{-x,1}\psi_{-x,N-1}$ to $[\rho^{N-2}](\psi_{-x}')^2$ supplies every power $x^2,\ldots,x^N$ with a strictly positive coefficient.
The first term in \eqref{eq:volume-generating} supplies the coefficient
\[
-\binom{n-1}{N-1}\frac{x}{2^{N+1}(n-1)}
\]
in $\mathscr R_N^{(n)}(-x)$.
This proves \eqref{eq:volume-alternating}, including $N=1$.
On the round sphere, $\cA=\nabla^*\nabla+2$ is positive on TT tensors, so the spectral formula \eqref{eq:volume-TT-hessian} gives strict negativity.
The displayed low-order formulas follow by substituting \eqref{eq:qrecurrence}--\eqref{eq:thetarecurrence} into \eqref{eq:volume-polynomial}.
\end{proof}

\subsection{Comparison with total \texorpdfstring{$Q_{2N}$}{Q2N}-curvature}

\begin{proposition}\label{prop:Q-volume-transform}
For $n>2N$, the Hessian polynomials satisfy
\begin{equation}\label{eq:Q-volume-transform}
\mathscr P_N^{(n)}(b)
=\sum_{j=1}^N\beta_{N,j}(n)\mathscr R_j^{(n)}(b),
\end{equation}
where, with $\nu=N-n/2$ and $f_N$ as in \eqref{eq:scalarhypergeom},
\begin{equation}\label{eq:Q-volume-coefficients}
\beta_{N,j}(n)
=2^N((N-1)!)^2[\rho^{N-j}]
\left(j-\frac{n\rho}{2a}\right)
\left(\frac12f_N^2-\frac{\rho}{\nu}f_Nf_N'\right).
\end{equation}
These coefficients are polynomials in $n$ with rational coefficients,
\[
\beta_{N,N}=2^{N-1}N!(N-1)!,
\]
and $\beta_{N,j}(n)$ is divisible by $n-2N$ when $j<N$.
Thus \eqref{eq:Q-volume-transform} also relates the fixed-volume Hessians of $\mathcal Q$ and $\mathcal V_j$ evaluated on TT tensors.
\end{proposition}
\begin{proof}
Differentiating \eqref{eq:volume-polynomial} as a finite generating series gives
\[
a^n\Theta_b'
=2\sum_{j=1}^N\mathscr R_j^{(n)}(b)
\left(j\rho^{j-1}-\frac n{2a}\rho^j\right)+O(\rho^N).
\]
Substitution into \eqref{eq:WronskianHessian} proves \eqref{eq:Q-volume-transform}--\eqref{eq:Q-volume-coefficients}.
The constant coefficient in the final parenthesis of \eqref{eq:Q-volume-coefficients} is $1/2$, giving $\beta_{N,N}$.
Every nonconstant coefficient of $f_N$ is a polynomial in $n$ divisible by $\nu$, as follows from \eqref{eq:scalarhypergeom}.
Consequently the apparent division by $\nu$ introduces no pole, and each $\beta_{N,j}$ is a polynomial in $n$.

To check the lower coefficients at $n=2N$, set
\[
\mathsf B(\rho)\coloneqq
\left.\left(\frac12f_N^2-\frac{\rho}{\nu}f_Nf_N'\right)\right|_{n=2N},
\qquad \mathsf B_d\coloneqq[\rho^d]\mathsf B.
\]
The finite coefficients in \eqref{eq:scalarhypergeom} give
\[
\mathsf B_0=\frac12,\qquad
\mathsf B_d=\frac{(N)_d}{2^d(N-d)_d},
\qquad 1\le d\le N-1.
\]
Comparing consecutive coefficients after multiplication by $a=1+\rho/2$ yields
\[
[\rho^{d-1}](\frac{\mathsf B}{a})=\frac{N-d}{N}\mathsf B_d.
\]
With $d=N-j$, the coefficient of $\rho^d$ in $(j-N\rho/a)\mathsf B$ is therefore zero.
This proves $\beta_{N,j}(2N)=0$ for $j<N$, and hence the divisibility.
For example, when $N=2$, the relation is
\begin{equation}\label{eq:Q-volume-low-transform}
\mathscr P_2^{(n)}=4\mathscr R_2^{(n)}
+\frac{n(n-4)}2\mathscr R_1^{(n)},
\end{equation}
in agreement with \eqref{eq:lowP} and \eqref{eq:volume-low-one}--\eqref{eq:volume-low-two}.
\end{proof}

\begin{corollary}\label{cor:volume-equal-dimension}
If $n=2N\ge4$, then on a closed manifold
\begin{equation}\label{eq:Q-volume-equal-dimension}
\mathcal Q(g)=2^{N-1}N!(N-1)!\,\mathcal V_N(g).
\end{equation}
At the Einstein reference, the polynomial representing the fixed-volume Hessian of $\mathcal V_N$ on TT tensors therefore has the product formula
\begin{equation}\label{eq:volume-equal-dimension-product}
\mathscr R_N^{(2N)}(b)
=\frac{(-1)^{N+1}}{2^{N+1}N!(N-1)!}
\prod_{j=0}^{N-1}[b-j(2N-j-1)].
\end{equation}
\end{corollary}
\begin{proof}
The integrated holographic identity \cite[(0.2)]{GrahamJuhl} gives \eqref{eq:Q-volume-equal-dimension} after the conversion $v_N=(-2)^Nv^{(2N)}$ in \eqref{eq:volume-coefficients}.
Differentiating it and applying Proposition \ref{prop:Matsumoto-comparison} proves \eqref{eq:volume-equal-dimension-product}.
This also identifies the normalization with the obstruction-tensor calculation in Appendix \ref{app:obstruction-comparison}.
\end{proof}

\end{document}